\documentclass[a4paper,12pt]{amsart}
\usepackage[utf8]{inputenc}
\usepackage[a4paper, top=3cm, left=2.5cm, right=2.5cm, bottom=3cm]{geometry}

\usepackage{relsize}
\usepackage{lineno}
\usepackage{hyperref}
\usepackage{amsmath,amsthm,pb-diagram,amssymb,comment}
\usepackage{amsfonts,graphicx,color}
\usepackage{enumitem, fancyhdr, dsfont}
\usepackage{thmtools,cleveref}
\usepackage[normalem]{ulem}
\usepackage{stmaryrd}
\usepackage{textcomp}
\usepackage{contour}
\usepackage{mathbbol}
\usepackage{varwidth}
\usepackage{tasks}
\usepackage{tikz,float}

\usepackage{thmtools,xcolor} 
\usepackage{pifont}

\usepackage{halloweenmath}

\usepackage{fontawesome}
\newcommand{\startlist}{\ \@beginparpenalty=10000}

\hypersetup{
    colorlinks=true,
    linkcolor=dodger,
    filecolor=dodger,
    urlcolor=dodger,
    citecolor=dodger,
}

\setlist{itemsep=1ex, topsep=0ex}

    \newcommand{\thzfc}{\mathrm{ZFC}}

     \newcommand{\Ed}{\mathbf{Ed}}
    \newcommand{\Mbf}{\mathbf{M}}

    \newcommand{\Cn}{\mathbf{Cn}}
    
    \newcommand{\Awf}{\mathcal{A}}

    \newcommand{\Ewf}{\mathcal{E}}
    
    \newcommand{\Gwf}{\mathcal{G}}
    
    \newcommand{\Iwf}{\mathcal{I}}
    \newcommand{\Jwf}{\mathcal{J}}
    \newcommand{\Mwf}{\mathcal{M}}
    \newcommand{\Nwf}{\mathcal{N}}
    
    \newcommand{\Pwf}{\mathcal{P}}
    \newcommand{\Swf}{\mathcal{S}}

    \newcommand{\bfrak}{\mathfrak{b}}
    \newcommand{\cfrak}{\mathfrak{c}}
    \newcommand{\dfrak}{\mathfrak{d}}

    \newcommand{\sfrak}{\mathfrak{s}}

    \newcommand{\menos}{\smallsetminus}
    
    \newcommand{\pts}{\mathcal{P}}
    
    \newcommand{\frestr}{\!\!\upharpoonright\!\!}

    \newcommand{\add}{\mbox{\rm add}}
    \newcommand{\cov}{\mbox{\rm cov}}
    \newcommand{\non}{\mbox{\rm non}}
    \newcommand{\cof}{\mbox{\rm cof}}
    \DeclareMathOperator{\limdir}{limdir}
    \DeclareMathOperator{\Fin}{Fin}

    \newcommand{\Bor}{\mathds{B}}
    \newcommand{\Cor}{\mathds{C}}
    \newcommand{\Dor}{\mathds{D}}

    \newcommand{\Loc}{\mathds{LOC}}

    \newcommand{\Por}{\mathds{P}}
    
    \newcommand{\Qor}{\mathds{Q}}
    \newcommand{\Ior}{\mathds{I}}

    \newcommand{\Qnm}{\dot{\mathds{Q}}}

    \newcommand{\SNwf}{\mathcal{SN}}
    \newcommand{\SMwf}{\mathcal{SM}}

    \newcommand{\cf}{\mbox{\rm cf}}

    \newcommand{\sii}{{\ \mbox{$\Leftrightarrow$} \ }}

    \newcommand{\la}{\langle}
    \newcommand{\ra}{\rangle}

\newcommand{\minLc}{\mathrm{minLc}}

\newcommand{\sqsubm}{\sqsubset^{\rm m}}
\newcommand{\sqrb}{\sqsubset^{\bullet}}

\newcommand{\Seq}{\mathrm{seq}}
\newcommand{\Fr}{\mathrm{Fr}}
\newcommand{\Rbf}{\mathbf{R}}

\newcommand{\Cbf}{\mathbf{C}}
\newcommand{\Lc}{\mathbf{Lc}}

\newcommand{\Lb}{\mathbf{Lb}}
\newcommand{\Hcal}{\mathcal{H}}
\newcommand{\Scal}{\mathcal{S}}
\newcommand{\id}{\mathrm{id}}
\newcommand{\blc}{\mathfrak{b}^{\mathrm{Lc}}}
\newcommand{\dlc}{\mathfrak{d}^{\mathrm{Lc}}}

\newcommand{\balc}{\mathfrak{b}^{\mathrm{aLc}}}
\newcommand{\dalc}{\mathfrak{d}^{\mathrm{aLc}}}

\DeclareMathOperator{\scf}{{\rm scf}}
\newcommand{\Fn}{\mathrm{Fn}}

\newcommand{\leqT}{\preceq_{\mathrm{T}}}
\newcommand{\eqT}{\cong_{\mathrm{T}}}

\newcommand{\gen}{\mathrm{gen}}

\newcommand{\hgt}{\mathrm{ht}}

\newcommand{\MAwf}{\mathcal{MA}}
\newcommand{\IAwf}{\mathcal{IA}}

\newcommand{\NAwf}{\Nwf\!\Awf}
\newcommand{\EAwf}{\mathcal{EA}}

\newcommand{\st}{\mid}
\newcommand{\set}[2]{\{#1 \st\, #2\}}
\newcommand{\largeset}[2]{\left\{#1 \;\middle|\; #2\right\}}
\newcommand{\seq}[2]{\la #1 \st\, #2\ra}

\newcommand{\baire}{\omega^\omega}

\newcommand{\cantor}{2^\omega}
\newcommand\subsetdot{\mathrel{\ooalign{$\subset$\cr
  \hidewidth\hbox{$\cdot\mkern3mu$}\cr}}}

\contourlength{0.8pt}

\contourlength{0.8pt}

	\definecolor{ultramarineblue}{rgb}{0.25, 0.4, 0.96}
\definecolor{cornellred}{rgb}{0.7, 0.11, 0.11}
\definecolor{cobalt}{rgb}{0.0, 0.28, 0.67}
\definecolor{bleudefrance}{rgb}{0.19, 0.55, 0.91}
\definecolor{darkblue}{rgb}{0.0, 0.0, 0.55}
\definecolor{ferrarired}{rgb}{1.0, 0.11, 0.0}
\definecolor{brandeisblue}{rgb}{0.0, 0.44, 1.0}
\definecolor{azure(colorwheel)}{rgb}{0.0, 0.5, 1.0}
\definecolor{aqua}{rgb}{0.0, 1.0, 1.0}
\definecolor{aguamarina}{cmyk}{0.85,0,0.33,0}
\definecolor{cafe}{cmyk}{0,0.81,1,0.60}
\definecolor{canela}{cmyk}{0.14,0.42,0.56,0}
\definecolor{darkgray}{cmyk}{0,0,0,0.50}
\definecolor{emerald}{cmyk}{0.91,0,0.88,0.12}
\definecolor{fresa}{cmyk}{0,1,0.50,0}
\definecolor{gold}{cmyk}{0,0.10,0.84,0}
\definecolor{lightgray}{cmyk}{0,0,0,0.30}
\definecolor{marron}{cmyk}{0,0.72,1,0.45}
\definecolor{melon}{cmyk}{0,0.29,0.84,0}
\definecolor{ladri}{cmyk}{0,0.77,0.87,0}
\definecolor{olive}{cmyk}{0.64,0,0.95,0.40}
\definecolor{orange}{cmyk}{0,0.42,1,0}
\definecolor{peach}{cmyk}{0,0.46,0.50,0}
\definecolor{pink}{cmyk}{0,0.10,0.10,0}
\definecolor{orange}{cmyk}{0,0.42,1,0}
\definecolor{pine}{cmyk}{0.92,0,0.59,0.25}
\definecolor{purple}{cmyk}{0.45,0.86,0,0}
\definecolor{violet}{cmyk}{0.07,0.90,0,0.34}
\definecolor{craneorange}{RGB}{252,187,6}
\definecolor{red(ncs)}{rgb}{0.77, 0.01, 0.2}

\newcommand{\blue}[1]{{\color{blue}#1}}

\usepackage{tikz}

\definecolor{sub0}{RGB}{29,32,137}
\definecolor{sub1}{RGB}{1,71,157}
\definecolor{sub2}{RGB}{1,104,183}
\definecolor{sub3}{RGB}{0,160,234}
\definecolor{sug}{RGB}{0,154,68}
\definecolor{suy}{RGB}{208,219,1}

\newcommand{\subiii}[1]{{\color{sub3}#1}}

\DeclareSymbolFont{extraup}{U}{zavm}{m}{n}
\DeclareMathSymbol{\varheart}{\mathalpha}{extraup}{86}
\DeclareMathSymbol{\vardiamond}{\mathalpha}{extraup}{87}

\definecolor{dodger}{rgb}{0.0,0.5,1.0}

\newcommand{\red}[1]{{\color{red}#1}}

\definecolor{amber}{rgb}{1.0,0.49,0.0}

\definecolor{ogreen}{RGB}{107,142,35}

\title[The uniformities of the null-additive and meager-additive ideals]{Uniformity numbers of the null-additive and meager-additive ideals}

\author[M.A.~Cardona]{Miguel A. Cardona}
\address{Einstein Institute of Mathematics, Edmond J. Safra Campus, Givat Ram, The Hebrew University of Jerusalem, Jerusalem, 91904, Israel}
\email{\href{mailto:miguel.cardona@mail.huji.ac.il}{miguel.cardona@mail.huji.ac.il}}
\urladdr{\url{https://sites.google.com/view/miacardonamo}}

\author[D.A.~Mej\'ia]{Diego A.\ Mej\'ia}
\address{Graduate School of System Informatics, Kobe University. 1-1 Rokkodai-cho, Nada-ku, Kobe, Hyogo 657-8501 Japan}
\email{\href{mailto:damejiag@people.kobe-u.ac.jp}{damejiag@people.kobe-u.ac.jp}}
\urladdr{\url{https://researchmap.jp/mejia?lang=en}}

\author[I.E.~Rivera-Madrid]{Ismael E.\ Rivera-Madrid}
\address{Faculty of Engineering, Instituci\'on Universitaria Pascual Bravo. Calle 73 No. 73A - 226, Medell\'in, Colombia.}
\email{\href{mailto:ismael.rivera@pascualbravo.edu.co}{ismael.rivera@pascualbravo.edu.co}}

\thanks{This work is supported by the Slovak Research and Development Agency under Contract No. APVV-20-0045 and by Pavol Jozef \v{S}af\'arik University at a postdoctoral position (first author), the Grants-in-Aid for Scientific Research (C) 23K03198, Japan Society for the Promotion of Science (second author), and by the grant No.~IN202204, Direcci\'on de Tecnolog\'ia e Investigaci\'on and Oficina de Internacionalizaci\'on, Instituci\'on Universitaria Pascual Bravo (all authors).}



\date{}

\definecolor{burntumber}{rgb}{0.54, 0.2, 0.14}
\definecolor{burgundy}{rgb}{0.5, 0.0, 0.13}

\numberwithin{equation}{section}

\begin{document}

\makeatletter
\def\@roman#1{\romannumeral #1}
\makeatother

\newcounter{enuAlph}
\renewcommand{\theenuAlph}{\Alph{enuAlph}}


\theoremstyle{plain}
  \newtheorem{theorem}[equation]{Theorem}
  \newtheorem{corollary}[equation]{Corollary}
  \newtheorem{lemma}[equation]{Lemma}
  \newtheorem{mainlemma}[equation]{Main Lemma}
  \newtheorem*{mainthm}{Main Theorem}
  \newtheorem{prop}[equation]{Proposition}
  \newtheorem{clm}[equation]{Claim}
  \newtheorem{fact}[equation]{Fact}
  \newtheorem{exer}[equation]{Exercise}
  \newtheorem{question}[equation]{Question}
  \newtheorem{problem}[equation]{Problem}
  \newtheorem{conjecture}[equation]{Conjecture}
  \newtheorem{assumption}[equation]{Assumption}
  \newtheorem*{thm}{Theorem}
  \newtheorem{teorema}[enuAlph]{Theorem}
  \newtheorem*{corolario}{Corollary}
\theoremstyle{definition}
  \newtheorem{definition}[equation]{Definition}
  \newtheorem{example}[equation]{Example}
  \newtheorem{remark}[equation]{Remark}
  \newtheorem{notation}[equation]{Notation}
  \newtheorem{context}[equation]{Context}

  \newtheorem*{defi}{Definition}
  \newtheorem*{acknowledgements}{Acknowledgements}

\def\sectionautorefname{Section}
\def\subsectionautorefname{Subsection}

\begin{abstract}

Denote by $\mathcal{NA}$ and $\mathcal{MA}$ the ideals of null-additive and meager-additive subsets of~$2^\omega$, respectively. We prove in ZFC that $\mathrm{add}(\mathcal{NA})=\mathrm{non}(\mathcal{NA})$ and introduce a new (Polish) relational system to reformulate Bartoszy\'nski's and Judah's characterization of the uniformity of $\mathcal{MA}$, which is helpful to understand the combinatorics of $\mathcal{MA}$ and to prove consistency results. As for the latter, we prove that $\mathrm{cov}(\mathcal{MA})<\mathfrak{c}$ (even $\mathrm{cov}(\mathcal{MA})<\mathrm{non}(\mathcal{N})$) is consistent with ZFC, as well as several constellations of Cicho\'n's diagram with $\mathrm{non}(\mathcal{NA})$, $\mathrm{non}(\mathcal{MA})$ and $\mathrm{add}(\mathcal{SN})$, which include $\mathrm{non}(\mathcal{NA})<\mathfrak{b}< \mathrm{non}(\mathcal{MA})$ and $\mathfrak{b}< \mathrm{add}(\mathcal{SN})<\mathrm{cov}(\mathcal{M})<\mathfrak{d}=\mathfrak{c}$.
\end{abstract}
\maketitle











\section{Introduction and Preliminaries}\label{sec:intro}

This work forms part of the study of the cardinal characteristics of the continuum related to the ideals of null-additive and meager-additive subsets of~$\cantor$, with particular focus on the uniformity number of these ideals. 
The study of these cardinals has been ongoing for some time. Some of the first results were achieved by Pawlikowski~\cite{paw85}, who studied these cardinal characteristics under the name of \emph{transitive additivity}. 
Later, Bartoszy\'nski and Judah~\cite[Thm.~2.2]{bartJudah} and Shelah~\cite{shmn} formulated very practical characterizations of the null-additive and meager-additive ideals, and provided combinatorial characterizations of their uniformity numbers (see \autoref{ch:NA-MA} and \autoref{addtch} below).

The goal of this work is to prove new results about the combinatorics of the null-additive and meager-additive ideals, mostly
concerning their uniformity numbers. We also consider the additivity of the strong measure zero ideal and prove several consistency results, strengthening those from Pawlikowski~\cite{paw85}.

Before plunging into details, we review some basic notation:   

\begin{notation}\

\begin{enumerate}[label=\rm(\arabic*)]
    \item Given a formula $\phi$, $\forall^\infty\, n<\omega\colon \phi$ means that all but finitely many natural numbers satisfy $\phi$; $\exists^\infty\, n<\omega\colon \phi$ means that infinitely many natural numbers satisfy $\phi$.

    \item Denote by $\Nwf$ and $\Mwf$ the $\sigma$-ideals of Lebesgue null sets and of meager sets in~$\cantor$, respectively, and let $\Ewf$ be the $\sigma$-ideal generated by the closed measure zero subsets of $\cantor$. It is well-known that $\Ewf\subseteq\Nwf\cap\Mwf$. Even more, it was proved that $\Ewf$ is a proper subideal of $\Nwf\cap\Mwf$ (see~\cite[Lemma 2.6.1]{BJ}).  

    \item $\cfrak:=2^{\aleph_0}$.
\end{enumerate}
\end{notation}

Let $\Iwf$ be an ideal of subsets of $X$ such that $\{x\}\in \Iwf$ for all $x\in X$. Throughout this paper, we demand that all ideals satisfy this latter requirement. We introduce the following four \emph{cardinal characteristics associated with $\Iwf$}: 
\begin{align*}
 \add(\Iwf)&=\min\largeset{|\Jwf|}{\Jwf\subseteq\Iwf,\,\bigcup\Jwf\notin\Iwf},\\
 \cov(\Iwf)&=\min\largeset{|\Jwf|}{\Jwf\subseteq\Iwf,\,\bigcup\Jwf=X},\\
 \non(\Iwf)&=\min\set{|A|}{A\subseteq X,\,A\notin\Iwf},\textrm{\ and}\\
 \cof(\Iwf)&=\min\set{|\Jwf|}{\Jwf\subseteq\Iwf,\ \forall\, A\in\Iwf\ \exists\, B\in \Jwf\colon A\subseteq B}.
\end{align*}
These cardinals are referred to as the \emph{additivity, covering, uniformity} and \emph{cofinality of $\Iwf$}, respectively. The relationship between the cardinals defined above is illustrated in \autoref{diag:idealI}.

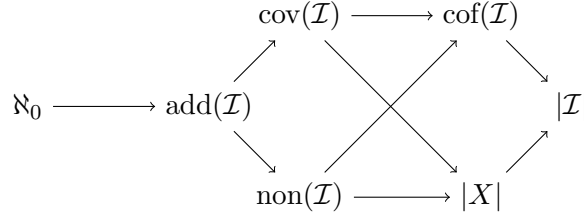
\begin{figure}[ht!]
\centering
\begin{tikzpicture}[scale=1.2]
\small{
\node (azero) at (-1,1) {$\aleph_0$};
\node (addI) at (1,1) {$\add(\Iwf)$};
\node (covI) at (2,2) {$\cov(\Iwf)$};
\node (nonI) at (2,0) {$\non(\Iwf)$};
\node (cofI) at (4,2) {$\cof(\Iwf)$};
\node (sizX) at (4,0) {$|X|$};
\node (sizI) at (5,1) {$|\Iwf|$};

\draw (azero) edge[->] (addI);
\draw (addI) edge[->] (covI);
\draw (addI) edge[->] (nonI);
\draw (covI) edge[->] (sizX);
\draw (nonI) edge[->] (sizX);
\draw (covI) edge[->] (cofI);
\draw (nonI) edge[->] (cofI);
\draw (sizX) edge[->] (sizI);
\draw (cofI) edge[->] (sizI);
}
\end{tikzpicture}
\caption{Diagram of the cardinal characteristics associated with $\Iwf$. An arrow  $\mathfrak x\rightarrow\mathfrak y$ means that (provably in ZFC) 
    $\mathfrak x\le\mathfrak y$.}
\label{diag:idealI}
\end{figure}

Throughout this paper, we consider the Cantor space $\cantor$ as a topological group with the standard modulo $2$ coordinatewise addition.  We say that an ideal $\Iwf\subseteq\Pwf(\cantor)$ is 
\emph{translation invariant} if $A+x\in\Iwf$ for each $A\in\Iwf$ and $x\in\cantor$.

\begin{definition}
Let $\Iwf\subseteq\Pwf(\cantor)$ be an ideal. A set $X\subseteq\cantor$ is termed \emph{$\Iwf$-additive} if, for every $A\in\Iwf$, $A+X\in\Iwf$. Denote by $\IAwf$ the collection of the $\Iwf$-additive subsets of $\cantor$. Notice that $\IAwf$ is a ($\sigma$-)ideal and $\IAwf\subseteq\Iwf$ when $\Iwf$ is a translation invariant ($\sigma$-)ideal.\footnote{Notice that $\IAwf$ contains all finite sets (which we demand for all ideals) iff $\Iwf$ is translation invariant. However, translation invariance is not required to check the other properties of $\sigma$-ideal for $\IAwf$, as well as $\IAwf\subseteq\Iwf$ and \autoref{bas:NA}.}
\end{definition}

We have an easy observation:

\begin{lemma}\label{bas:NA}
For any translation invariant ideal $\Iwf$ on $\cantor$, we have:   
\begin{enumerate}[label=\rm(\arabic*)]
    \item\label{bas:NA1} $\add(\Iwf)\leq\add(\IAwf)$.

    \item\label{bas:NA2} $\cov(\Iwf)\leq\cov(\IAwf)$ and $\non(\IAwf)\leq\non(\Iwf)$.
\end{enumerate}
\end{lemma}
\begin{proof}
\noindent\ref{bas:NA1}: Let $\set{X_\alpha}{\alpha<\kappa}\subseteq\IAwf$ with $\kappa<\add(\Iwf)$. We show that $\bigcup_{\alpha<\kappa}X_\alpha\in\IAwf$. Let $B\in\Iwf$. Then $X_\alpha+B\in\Iwf$. Observe that
\[\bigg(\bigcup_{\alpha<\kappa}X_\alpha\bigg)+B=\bigcup_{\alpha<\kappa}(X_\alpha+B).\]
Since $\kappa<\add(\Iwf)$, $\bigcup_{\alpha<\kappa}(X_\alpha+B)\in\Iwf$. Therefore, $\bigcup_{\alpha<\kappa}X_\alpha\in\IAwf$. 

\noindent\ref{bas:NA2}: 
Clear because $\IAwf\subseteq\Iwf$.
\end{proof}

The cardinal $\non(\IAwf)$ has been studied in~\cite{paw85,Kra} under the different name \emph{transitive additivity of $\Iwf$}:\footnote{In~\cite{BJ} is denoted by $\add^\star(\Iwf)$.}
\[\add_t^*(\Iwf)=\min\set{|X|}{X\subseteq\cantor\text{ and }\exists A\in\Iwf\colon A+X\notin\Iwf}.\]
It is clear from the definition that $\non(\IAwf)=\add_t^*(\Iwf)$.

The ideal $\IAwf$ has received a lot of attention when $\Iwf$ is either $\Mwf$ or $\Nwf$.  Pawlikowski~\cite{paw85} characterized $\add^*_t(\Nwf)$ (i.e.\ $\non(\NAwf)$) employing slaloms.

\begin{definition}\label{defloc}
Given a sequence of non-empty sets $b = \seq{b(n)}{n\in\omega}$ and $h\colon \omega\to\omega$, define
\begin{align*}
 \prod b &:= \prod_{n\in\omega}b(n),\textrm{\ and} \\
 \Swf(b,h) &:= \prod_{n\in\omega} [b(n)]^{\leq h(n)}.
\end{align*}
For two functions $x\in\prod b$ and $\varphi\in\Swf(b,h)$ write  
\[x\,\in^*\varphi\textrm{\ iff\ }\forall^\infty n\in\omega:x(n)\in \varphi(n).\]
We set
\[\blc_{b,h}:=\min\largeset{|F|}{F\subseteq \prod b \text{ and } \neg\exists \varphi \in \Swf(b,h)\,\forall x \in F\colon x\in^* \varphi},\]
and set $\minLc:=\min\set{\blc_{b,\id_\omega}}{b\in\baire}$. Here, $\id_\omega$ denotes the identity function on $\omega$.
\end{definition}

\begin{theorem}[{\cite[Lemma~2.2]{paw85}}]\label{chNPaw}
$\non(\NAwf)=\minLc$.
\end{theorem}

Another characterization of $\minLc$ is the following.

\begin{lemma}[{\cite[Lemma~3.8]{CM}}]\label{minlcCM}
$\minLc=\min\set{\blc_{b,h}}{b\in\baire}$ when $h$ goes to infinity.  
\end{lemma}

Hence, we can infer:

\begin{corollary}
 $\non(\NAwf)=\min\set{\blc_{b,h}}{b\in\baire}$ when $h$ goes to infinity.
\end{corollary}

Yet another characterization of $\add(\Nwf)$ and $\add(\Mwf)$ in terms of the uniformity of the null-additive and meager-additive ideals was accomplished by Pawlikowski. Here, $\bfrak$ denotes the \emph{bounding number}, which is defined in~\autoref{b-d}.

\begin{theorem}[{\cite[Lem.~2.3]{paw85}}]\label{chPawmn}\
\begin{enumerate}[label=\rm(\arabic*)]
    \item $\add(\Nwf)=\min\{\bfrak,\non(\NAwf)\}$.

    \item $\add(\Mwf)=\min\{\bfrak,\non(\MAwf)\}$. 
\end{enumerate}

\end{theorem}

As a consequence of the previous, we immediately have the following:

\begin{corollary}\label{cor:addnonMA}
If $\non(\IAwf) \leq \bfrak$, then $\add(\Iwf)=\add(\IAwf)=\non(\IAwf)$, when $\Iwf$ is $\Nwf$ or $\Mwf$.    
\end{corollary}

On the other hand, 
Bartoszy\'nski and Judah~\cite{bartJudah} and Shelah ~\cite{shmn} provided important combinatorial characterizations of the null-additive and meager-additive sets, which are stated below. Shelah used them to prove that every null-additive set is meager-additive, that is, $\NAwf\subseteq\MAwf$. 

Denote by $\Ior$ the set of partitions of $\omega$ into finite non-empty intervals.

\begin{theorem}\label{ch:NA-MA} Let $X\subseteq\cantor$.
    \begin{enumerate}[label=\rm(\arabic*)]
        \item \emph{(\cite[Thm.~13]{shmn})}\label{ch:NA}
 $X\in\NAwf$ iff for all $I=\seq{I_n}{n\in\omega}\in\Ior$ there is some $\varphi\in\prod_{n\in \omega}\pts(2^{I_n})$ such that $\forall n\in \omega\colon |\varphi(n)|\leq n$ and $X\subseteq H_\varphi$,  where \[
H_\varphi:=\set{x\in\cantor}{\forall^{\infty} n\in \omega\colon x{\upharpoonright}I_n\in \varphi(n)}.
\]
        \item\emph{(\cite[Thm.~2.2]{bartJudah})}\label{ch:MA}
 $X\in\MAwf$ iff for all $I\in\Ior$ there are $J\in\Ior$ and $y\in\cantor$  such that  \[\forall x\in X\, \forall^\infty n<\omega\, \exists k<\omega\colon I_k\subseteq J_n\text{\ and\ }x{\upharpoonright}I_k=y{\upharpoonright}I_k.\] 
 Moreover, Shelah~\cite[Thm.~18]{shmn} proved that $J$ can be found coarser than $I$, i.e.\ every member of $J$ is the union of members of $I$
    \end{enumerate}
\end{theorem}

Bartoszy\'nski and Judah provided a characterization of the uniformity of the meager-additive ideal:

 \begin{theorem}[{\cite[Thm.~2.2]{bartJudah}}, see also {\cite[Thm.~2.7.14]{BJ}}]\label{addtch}\ \\ 
The cardinal $\non(\MAwf)$ is the largest cardinal $\kappa$ such that, for every bounded family $F\subseteq\baire$ of size ${<}\kappa$, \[\tag{\faPagelines}\exists r,h\in\baire\, \forall f\in F\, \exists n\in\omega\, \forall m\geq n\,  \exists k\in[r(m),r(m+1)]\colon f(k)=h(k).\label{addtchtag}\]
\end{theorem}

In~\cite{zin}, Zindulka used combinatorial properties
of meager-additive sets described by Shelah and Pawlikowski to
characterize meager-additive sets in $\cantor$ in a way that nicely parallels the definition of strong measure zero sets. This led him to establish that $\EAwf=\MAwf$. Therefore:

\begin{corollary}\label{ma=me}
 $\non(\MAwf)=\non(\EAwf)$. As a consequence, 
 $\non(\MAwf)\leq\non(\Ewf)$. 
\end{corollary}



In the present paper, we use the previously mentioned combinatorial properties
of null and meager-additive sets described by Bartoszy\'nski, Judah, Shelah, and Pawlikowski, to prove our main results. First, we show that the hypothesis $\non(\NAwf) \leq \bfrak$ is not required in \autoref{cor:addnonMA} (for the null-additive ideal) to show that: 

\begin{teorema}\label{MainThm1} In $\thzfc$ we have that $\add(\NAwf)=\non(\NAwf)$.
\end{teorema}

This is one of the main results of the paper. 
It is unclear whether ZFC proves $\add(\MAwf) = \non(\MAwf)$.

Another result of Pawlikowski concerns the relationship between $\add_t(\Nwf)$ and $\add(\SNwf)$, where $\SNwf$ denotes the $\sigma$-ideal of the strong measure zero sets (see~\autoref{def:SN}). Namely, he states $\add_t(\Nwf)\leq\add(\SNwf)$,  but this proof does not appear anywhere. We offer our own proof of this inequality in~\autoref{Sec:zfc}. Concretely, we prove:

\begin{theorem}[{\cite{paw85}}]\label{MainThm1c}  $\minLc\leq\add(\SNwf)$. 
\end{theorem}

As a noteworthy consequence, we get:

\begin{corollary}
 $\non(\NAwf)\leq\add(\SNwf)$.    
\end{corollary}

\autoref{cichonext} summarizes the inequalities among some cardinal characteristics associated with $\Ewf$, $\MAwf$,  $\NAwf$, and $\SNwf$, with the cardinals in Cicho\'n's diagram. 
Notice that $\non(\MAwf) \leq \non(\SNwf)$ follows from $\MAwf\subseteq \SNwf$ (by~Galvin's, Mycielski's, and Solovay's~\cite{GaMS} characterization of strong measure zero sets),
 and that $\add(\Mwf) \leq \non(\Ewf)$ is a consequence of $\add(\Ewf) = \add(\Mwf)$ (\cite{BS1992}).

\begin{figure}[ht!]
\centering
\begin{tikzpicture}[scale=1.2]
\small{
\node (aleph1) at (-1,2.3) {$\aleph_1$};
\node (addn) at (0.5,2.3){$\add(\Nwf)$};
\node (covn) at (0.5,7.5){$\cov(\Nwf)$};
\node (nonn) at (10.4,2.3) {$\non(\Nwf)$} ;
\node (cfn) at (10.4,7.5) {$\cof(\Nwf)$} ;
\node (addm) at (3.8,2.3) {$\add(\Mwf)$} ;
\node (covm) at (6.9,2.3) {$\cov(\Mwf)$} ;
\node (nonsn) at (8.8,2.3) {$\non(\SNwf)$} ;
\node (nonm) at (3.8,7.5) {$\non(\Mwf)$} ;
\node (cfm) at (6.9,7.5) {$\cof(\Mwf)$} ;
\node (b) at (3.8,5.1) {$\bfrak$};
\node (d) at (6.9,5.1) {$\dfrak$};
\node (c) at (11.7,7.5) {$\cfrak$};
\node (addsn) at (5.2,3.2) {$\add(\SNwf)$};
\node (addma) at (2.4,3.2) {\subiii{$\add(\MAwf)$}};
\node (nonna) at (1.3,3.9) {\subiii{$\non(\NAwf)$}};
\node (nonma) at (1.8,5.3) {\subiii{$\non(\MAwf)$}};
\node (none) at (2.5,6.8) {$\non(\Ewf)$};
\draw (aleph1) edge[->] (addn)
      (addn) edge[->] (covn)
      (covn) edge [->] (nonm)
      (nonm)edge [->] (cfm)
      (cfm)edge [->] (cfn)
      (cfn) edge[->] (c);

\draw (addm) edge [->]  (addma);
     
\draw (addma) edge [->]  (nonma);
\draw
   (addn) edge [->]  (addm)
   (addm) edge [->]  (covm)
   (covm) edge [->]  (nonsn)
   (nonsn) edge [->]  (nonn)
   (nonna) edge [->]  (nonma)
   (nonn) edge [->]  (cfn);
\draw (addm) edge [->] (b)
      (b)  edge [->] (nonm);
\draw (covm) edge [->] (d)
      (d)  edge[->] (cfm);
\draw (b) edge [->] (d);

\draw  
(addsn) edge [line width=.15cm,white,-]  (nonna)
(addsn) edge [<-]  (nonna)
(nonma) edge [line width=.15cm,white,-]  (nonsn)
(nonma) edge [->]  (nonsn)
(addsn) edge [line width=.15cm,white,-]  (nonsn)
(addsn) edge [->]  (nonsn)
(none) edge [line width=.15cm,white,-]  (nonn)
(none) edge [->]  (nonn)
(nonma) edge [->] (none)
        (none) edge [->] (nonm)
        (addn) edge [->] (nonna);
      
      



}
\end{tikzpicture}
\caption{Cicho\'n's diagram with some cardinal characteristics associated with $\Ewf$, $\MAwf$, $\SNwf$, and $\NAwf$.}
\label{cichonext}
\end{figure}
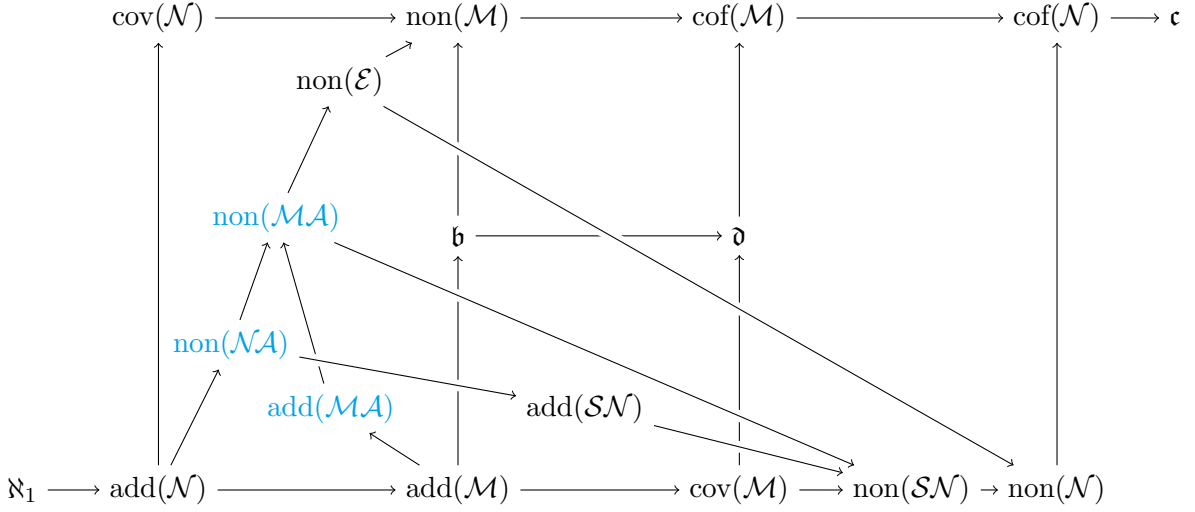

Many cardinal characteristics can be described using a \emph{relational system}, as reviewed in \autoref{Sec:zfc}. For any relational system $\Rbf$, their bounding and dominating numbers are denoted by $\bfrak(\Rbf)$ and $\dfrak(\Rbf)$, respectively. 

\autoref{addtchtag} has inspired us to develop a new (Polish) relational system $\Rbf_b$, parametrized by $b\in\omega^\omega$, which can be used to reformulate \autoref{addtch} as 
\[\non(\MAwf)=\min\set{\bfrak(\Rbf_b)}{b\in\baire}.\] 
This relational system will play an important role in this work, also for our consistency results. First, we prove (in~\autoref{Sec:zfc}) the following connections between their associated cardinal characteristics and some other classical characteristics. 

\begin{teorema}\label{MoreZFC}\ 
\begin{enumerate}[label=\rm(\arabic*)]
    \item 
    $\displaystyle \sup_{b\in D}\dfrak(\Rbf_b) \leq \cov(\MAwf) \leq \dfrak\left(\prod_{b\in D}\Rbf_b\right)$ for any dominating family $D\subseteq \baire$. 

    \item For all $b\in\omega^\omega$, $\dfrak(\Rbf_b) \leq \cof(\Mwf)$.

    \item Let $b\in\omega^\omega$.  If $\displaystyle \sum_{k<\omega}\frac{1}{b(k)}<\infty$, then $\bfrak(\Rbf_b) \leq \non(\Ewf)$ and $\cov(\Ewf)\leq \dfrak(\Rbf_b)$.
\end{enumerate}
\end{teorema}

In~\cite{CMRM}, it was established that $\cof(\Nwf)=\aleph_1$ and $\cov(\SNwf)=\cfrak=\aleph_2$ hold in Sacks' model. There,  $\cov(\MAwf)=\aleph_2$ because $\cov(\SNwf) \leq \cov(\MAwf) \leq \cov(\NAwf)$ (remember that $\NAwf \subseteq \MAwf\subseteq \SNwf$). By~\autoref{MoreZFC}, we obtain that $\sup_{b\in\omega^\omega}\dfrak(\Rbf_b) =\aleph_1$ in Sacks' model, which yields the consistency of $\sup_{b\in\omega^\omega}\dfrak(\Rbf_b)<\cov(\MAwf)$. In addition, it follows that no classical cardinal characteristic of the continuum (different from $\cfrak$) is an upper bound of $\cov(\MAwf)$. Because of the latter,  we ask whether is it consistent with ZFC that $\cov(\MAwf)<\cfrak$, or even $\cov(\NAwf)<\cfrak$. The upper bound of $\cov(\MAwf)$ from \autoref{MoreZFC} is used in~\autoref{Secmain} to show that this holds for $\MAwf$, but the case for $\NAwf$ remains open.

\begin{teorema}[\autoref{Thm:covMA}]\label{ThmCovma}
Let $\theta<\nu\leq\lambda$ be uncountable cardinals such that $\theta^{<\theta}=\theta$, $\nu^{\theta} = \nu$ and $\lambda^{\aleph_0} = \lambda$. Then there is a poset, preserving cofinalities, forcing 
\[
     \cov(\Nwf) = \aleph_1  \leq \add(\Mwf) = \cof(\Mwf) = \theta 
      \leq \cov(\MAwf)\leq \nu \leq  \non(\Nwf) =\cfrak =\lambda.
\]
In particular, it is consistent with $\thzfc$ that $\cov(\MAwf) < \non(\Nwf)$.    
\end{teorema}

Concerning more consistency results, 
Pawlikowski~\cite[Thm.~2.4]{paw85} constructed a FS (finite support) iteration of ccc posets to obtain a model where 
\[
\add(\Nwf)=\add(\Mwf)=\bfrak=\aleph_1<\non(\NAwf)=\non(\MAwf)=\cfrak=\aleph_2.
\]
On the other hand, in~\cite[Thm.~5.15]{CM} we constructed a model where 
\[
\add(\Nwf)=\add(\Mwf)=\bfrak=\mu\leq\minLc=\non(\Mwf)=\cov(\Mwf)=\nu\leq\dfrak=\cfrak=\lambda
\]
for arbitrary regular cardinals $\mu\leq\nu$ and a cardinal $\lambda\geq\nu$ such that $\lambda=\lambda^{<\mu}$.\footnote{This last requirement can be weakened to $\cof([\lambda]^{<\mu}) = \lambda = \lambda^{\aleph_0}$.} Hence, by~\autoref{chNPaw} and \autoref{bas:NA}~\ref{bas:NA2}, 
\begin{equation*}\label{cmresult}
\bfrak=\mu \leq \non(\NAwf)=\non(\MAwf)=\nu\tag{\faFilePdfO}    
\end{equation*}
holds in this model. 
Consequently, it is consistent that $\bfrak<\add(\NAwf)=\non(\NAwf)$ by~\autoref{MainThm1}. On the other hand,  $\aleph_1=\bfrak=\non(\MAwf)<\cov(\Nwf)=\aleph_2$ holds in the model obtained by a FS iteration of length $\aleph_2$ of random forcing
(see e.g~\cite[Thm.~5.4]{Car23}) because $\non(\MAwf)=\non(\EAwf)\leq\non(\Ewf)$ by \autoref{ma=me}.

Motivated by~(\ref{cmresult}), we could ask:

\begin{problem}\label{P:ma_na} Are each of the following statements consistent with $\thzfc$?
\begin{enumerate}[label=\normalfont(\alph*)]
    \item\label{P:ma_na1} $\bfrak<\non(\NAwf)<\non(\MAwf)$.
    
    \item\label{P:ma_na2} $\non(\NAwf)<\non(\MAwf)<\bfrak$.
    
    \item\label{P:ma_na3} $\non(\NAwf)<\bfrak<\non(\MAwf)$.
\end{enumerate}
\end{problem}
Concerning~\autoref{P:ma_na}~\ref{P:ma_na1},~\ref{P:ma_na2}, it is known that ``$\non(\Nwf)=\aleph_1$ and $\cov(\Nwf)=\bfrak=\aleph_2=\cfrak$" is consistent with ZFC (see e.g~\cite[Model~7.6.7]{BJ}), which implies $\non(\NAwf)=\non(\MAwf)=\aleph_1<\bfrak$. On the other hand, a model for $\non(\NAwf)=\aleph_1<\non(\MAwf)=\bfrak=\cfrak=\aleph_2$ is obtained by adding $\aleph_2$-many dominating reals by using a FS iteration of length $\aleph_2$ of Hechler forcing $\Dor$ (see~\autoref{defposet}~\ref{defposet2}) because $\add(\Mwf)=\min\{\bfrak,\non(\MAwf)\}$ by~\autoref{chPawmn} and $\non(\NAwf)=\aleph_1$  by~\cite[Lem.~4.24]{CM}.

We give a positive answer to~\autoref{P:ma_na}~\ref{P:ma_na3}, which is the main result of this work. By separating even more cardinal characteristics of the continuum, we prove:

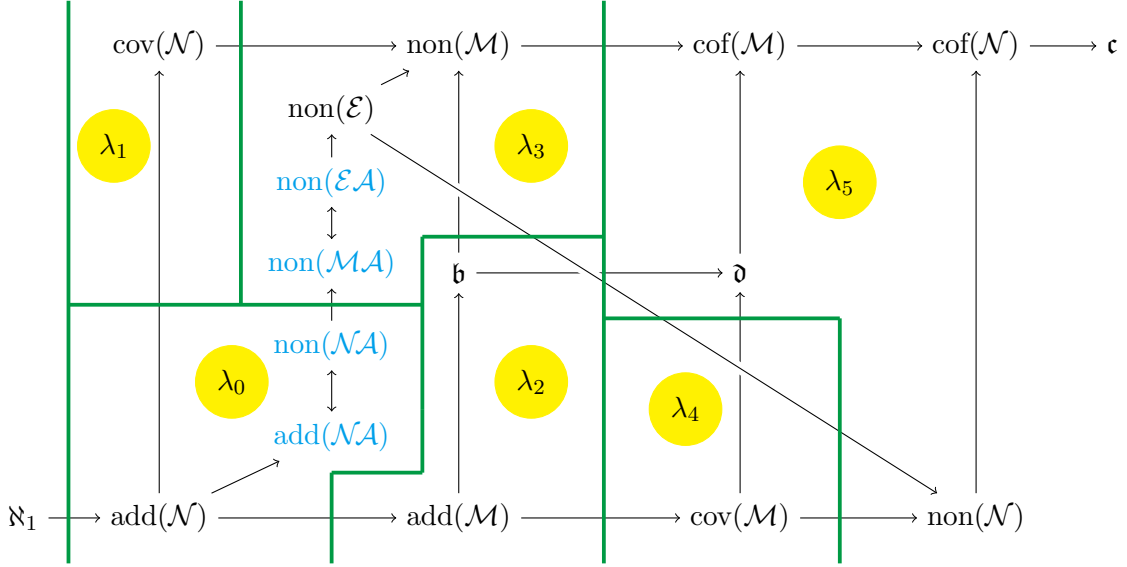
\begin{figure}[t]
\centering
\begin{tikzpicture}[scale=1.2]
\small{
\node (aleph1) at (-1,2.3) {$\aleph_1$};
\node (addn) at (0.5,2.3){$\add(\Nwf)$};
\node (covn) at (0.5,7.5){$\cov(\Nwf)$};
\node (nonn) at (9.5,2.3) {$\non(\Nwf)$} ;
\node (cfn) at (9.5,7.5) {$\cof(\Nwf)$} ;
\node (addm) at (3.8,2.3) {$\add(\Mwf)$} ;
\node (covm) at (6.9,2.3) {$\cov(\Mwf)$} ;
\node (nonm) at (3.8,7.5) {$\non(\Mwf)$} ;
\node (cfm) at (6.9,7.5) {$\cof(\Mwf)$} ;
\node (b) at (3.8,5) {$\bfrak$};
\node (d) at (6.9,5) {$\dfrak$};
\node (c) at (11,7.5) {$\cfrak$};
\node (addna) at (2.4,3.2) {\subiii{$\add(\NAwf)$}};
\node (nonna) at (2.4,4.2) {\subiii{$\non(\NAwf)$}};
\node (nonma) at (2.4,5.1) {\subiii{$\non(\MAwf)$}};
\node (nonea) at (2.4,6) {\subiii{$\non(\EAwf)$}};
\node (none) at (2.4,6.8) {$\non(\Ewf)$};
\draw (aleph1) edge[->] (addn)
      (addn) edge[->] (covn)
      (covn) edge [->] (nonm)
      (nonm)edge [->] (cfm)
      (cfm)edge [->] (cfn)
      (cfn) edge[->] (c);

\draw (nonma) edge [<->]  (nonea)
   (addn) edge [->]  (addm)
   (addm) edge [->]  (covm)
   (covm) edge [->]  (nonn)
   (addna) edge [<->]  (nonna)
   (nonna) edge [->]  (nonma)
   (nonn) edge [->]  (cfn);
\draw (addm) edge [->] (b)
      (b)  edge [->] (nonm);
\draw (covm) edge [->] (d)
      (d)  edge[->] (cfm);
\draw (b) edge [->] (d);

\draw   (nonea) edge [->] (none)
        (none) edge [->] (nonm)
        (addn) edge [->] (addna);

\draw (none) edge [line width=.15cm,white,-] (nonn)
(none) edge [->] (nonn);

\draw[color=sug,line width=.05cm] (-0.5,4.65)--(3.4,4.65);
\draw[color=sug,line width=.05cm] (3.4,5.4)--(5.4,5.4);
\draw[color=sug,line width=.05cm] (-0.5,1.8)--(-0.5,8);
\draw[color=sug,line width=.05cm] (5.4,8)--(5.4,1.8);
\draw[color=sug,line width=.05cm] (3.4,3.5)--(3.4,5.4);
\draw[color=sug,line width=.05cm] (3.4,2.8)--(3.4,3.5);
\draw[color=sug,line width=.05cm] (3.4,2.8)--(2.4,2.8);
\draw[color=sug,line width=.05cm] (2.4,1.8)--(2.4,2.8);
\draw[color=sug,line width=.05cm] (1.4,4.65)--(1.4,8);
\draw[color=sug,line width=.05cm] (8,1.8)--(8,4.5);
\draw[color=sug,line width=.05cm] (5.4,4.5)--(8,4.5); 

\draw[circle, fill=yellow,color=yellow] (1.3,3.8) circle (0.4);
\draw[circle, fill=yellow,color=yellow] (0,6.4) circle (0.4);
\draw[circle, fill=yellow,color=yellow] (4.6,3.8) circle (0.4);
\draw[circle, fill=yellow,color=yellow] (4.6,6.4) circle (0.4);
\draw[circle, fill=yellow,color=yellow] (6.3,3.5) circle (0.4);
\draw[circle, fill=yellow,color=yellow] (8,6) circle (0.4);
\node at (1.3,3.8) {$\lambda_0$};
\node at (0,6.4) {$\lambda_1$};
\node at (4.6,3.8) {$\lambda_2$};
\node at (4.6,6.4) {$\lambda_3$};
\node at (6.3,3.5) {$\lambda_4$};
\node at (8,6) {$\lambda_5$};
}
\end{tikzpicture}
\caption{Constellation forced in~\autoref{Mainthm}.}
\label{Figtthm:a}
\end{figure}

\begin{teorema}[\autoref{Mthm}]\label{Mainthm}
Let $\lambda_0\leq\lambda_1\leq\lambda_2\leq\lambda_3\leq\lambda_4$ be uncountable regular cardinals, and $\lambda_5$ a cardinal such that $\lambda_5\geq\lambda_4$ and $\cof([\lambda_5]^{<\lambda_i})=\lambda_5 = \lambda_5^{\aleph_0}$ for $i\leq 2$. 
Then there is a ccc poset forcing~\autoref{Figtthm:a}. 
\end{teorema}


We describe the method to approach \autoref{Mainthm}. 
Goldstern, Mej\'ia,  and Shelah~\cite{GMS} discovered a way to construct sequences of ultrafilters along a FS iteration to control that restrictions of the eventually different real forcing do not add dominating reals, a technique that was used to force
the consistency of the constellation of~\autoref{leftGMS}. The latter was used and improved in~\cite{BCM} to force seven values in Cichon's diagram with the left side separated (see~\autoref{BCM}). The latter method consists of building ultrafilters along a matrix iteration, which is known as an~\emph{ultrafilter extendable matrix iteration} (uf-extendable matrix iteration, see~\autoref{Defmatuf}). Recently, in~\cite{Car23} this method was used to force that the four cardinal characteristics associated with $\Ewf$ can be pairwise different, and in~\cite{BCM2} to force \emph{Cicho\'n's maximum} ($10$ different values in Cicho\'n's diagram, the maximum possible) along with pairwise different values for the cardinal characteristics associated with $\SNwf$.  

\begin{figure}[t]
\centering
\begin{tikzpicture}[scale=1.2]
\small{
\node (aleph1) at (-1,3) {$\aleph_1$};
\node (addn) at (1,3){$\add(\Nwf)$};
\node (covn) at (1,7){$\cov(\Nwf)$};
\node (nonn) at (9,3) {$\non(\Nwf)$} ;
\node (cfn) at (9,7) {$\cof(\Nwf)$} ;
\node (addm) at (3.66,3) {$\add(\Mwf)$} ;
\node (covm) at (6.33,3) {$\cov(\Mwf)$} ;
\node (nonm) at (3.66,7) {$\non(\Mwf)$} ;
\node (cfm) at (6.33,7) {$\cof(\Mwf)$} ;
\node (b) at (3.66,5) {$\bfrak$};
\node (d) at (6.33,5) {$\dfrak$};
\node (c) at (11,7) {$\cfrak$};
\draw (aleph1) edge[->] (addn)
      (addn) edge[->] (covn)
      (covn) edge [->] (nonm)
      (nonm)edge [->] (cfm)
      (cfm)edge [->] (cfn)
      (cfn) edge[->] (c);
\draw
   (addn) edge [->]  (addm)
   (addm) edge [->]  (covm)
   (covm) edge [->]  (nonn)
   (nonn) edge [->]  (cfn);
\draw (addm) edge [->] (b)
      (b)  edge [->] (nonm);
\draw (covm) edge [->] (d)
      (d)  edge[->] (cfm);
\draw (b) edge [->] (d);

\draw[color=blue,line width=.05cm] (-0.2,5.6)--(4.6,5.6);
\draw[color=blue,line width=.05cm] (-0.2,2.5)--(-0.2,7.5);
\draw[color=blue,line width=.05cm] (2.5,2.5)--(2.5,7.5);
\draw[color=blue,line width=.05cm] (4.6,2.5)--(4.6,7.5);

\draw[circle, fill=yellow,color=yellow] (1.7,4.4) circle (0.4);
\draw[circle, fill=yellow,color=yellow] (7.8,5)   circle (0.4);
\draw[circle, fill=yellow,color=yellow] (1.7,6.3) circle (0.4);
\draw[circle, fill=yellow,color=yellow] (3.1,4) circle (0.4);
\draw[circle, fill=yellow,color=yellow] (3.1,6.3) circle (0.4);
\node at (1.7,4.4) {$\theta_0$};
\node at (7.8,5) {$\theta_4$};
\node at (1.7,6.3) {$\theta_1$};
\node at (3.1,4) {$\theta_2$};
\node at (3.1,6.3) {$\theta_3$};
}
\end{tikzpicture}
\caption{Separating the cardinal characteristics on the left side of Cicho\'n's diagram. This constellation was forced in~~\cite[Main Thm.~6.1]{GMS} where $\aleph_1\leq\theta_0\leq\theta_1 \leq\theta_2 \leq \theta_3=\theta_3^{\aleph_0}$ are regular and $\theta_4$ is a cardinal such that $\theta_3 < \theta_4=\theta_4^{<\theta_3}$.}
\label{leftGMS}
\end{figure}
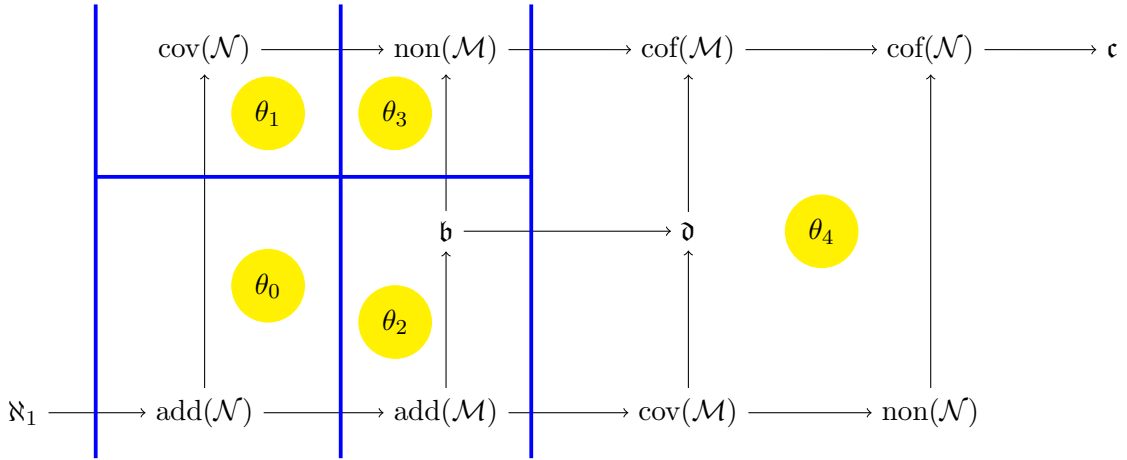  

 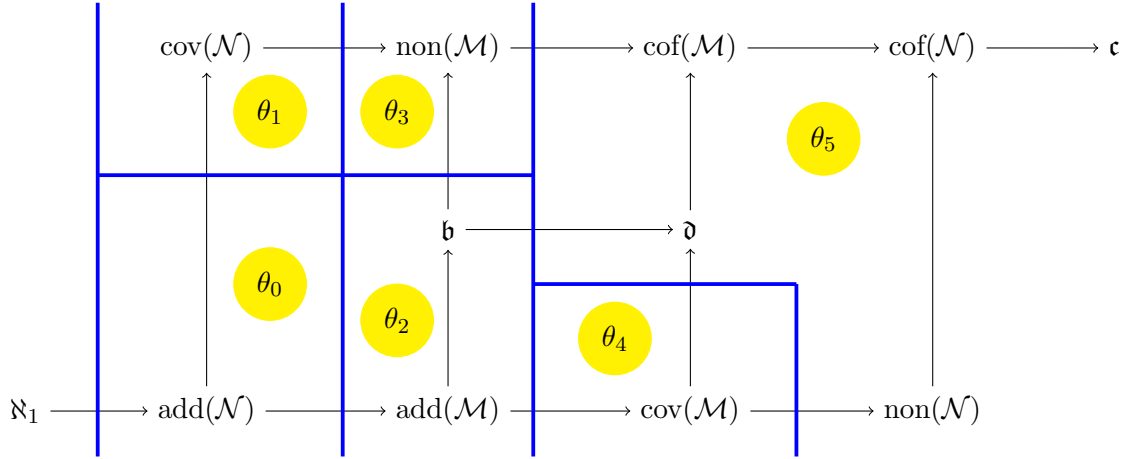
\begin{figure}[ht!]
\centering
\begin{tikzpicture}[scale=1.2]
\small{
\node (aleph1) at (-1,3) {$\aleph_1$};
\node (addn) at (1,3){$\add(\Nwf)$};
\node (covn) at (1,7){$\cov(\Nwf)$};
\node (nonn) at (9,3) {$\non(\Nwf)$} ;
\node (cfn) at (9,7) {$\cof(\Nwf)$} ;
\node (addm) at (3.66,3) {$\add(\Mwf)$} ;
\node (covm) at (6.33,3) {$\cov(\Mwf)$} ;
\node (nonm) at (3.66,7) {$\non(\Mwf)$} ;
\node (cfm) at (6.33,7) {$\cof(\Mwf)$} ;
\node (b) at (3.66,5) {$\bfrak$};
\node (d) at (6.33,5) {$\dfrak$};
\node (c) at (11,7) {$\cfrak$};
\draw (aleph1) edge[->] (addn)
      (addn) edge[->] (covn)
      (covn) edge [->] (nonm)
      (nonm)edge [->] (cfm)
      (cfm)edge [->] (cfn)
      (cfn) edge[->] (c);
\draw
   (addn) edge [->]  (addm)
   (addm) edge [->]  (covm)
   (covm) edge [->]  (nonn)
   (nonn) edge [->]  (cfn);
\draw (addm) edge [->] (b)
      (b)  edge [->] (nonm);
\draw (covm) edge [->] (d)
      (d)  edge[->] (cfm);
\draw (b) edge [->] (d);

\draw[color=blue,line width=.05cm] (-0.2,5.6)--(4.6,5.6);
\draw[color=blue,line width=.05cm] (-0.2,2.5)--(-0.2,7.5);
\draw[color=blue,line width=.05cm] (2.5,2.5)--(2.5,7.5);
\draw[color=blue,line width=.05cm] (4.6,2.5)--(4.6,7.5);
\draw[color=blue,line width=.05cm] (4.6,4.4)--(7.5,4.4);
\draw[color=blue,line width=.05cm] (7.5,2.5)--(7.5,4.4);
 

\draw[circle, fill=yellow,color=yellow] (1.7,4.4) circle (0.4);
\draw[circle, fill=yellow,color=yellow] (7.8,6)   circle (0.4);
\draw[circle, fill=yellow,color=yellow] (1.7,6.3) circle (0.4);
\draw[circle, fill=yellow,color=yellow] (3.1,4) circle (0.4);
\draw[circle, fill=yellow,color=yellow] (3.1,6.3) circle (0.4);
\draw[circle, fill=yellow,color=yellow] (5.5,3.8) circle (0.4);
\node at (1.7,4.4) {$\theta_0$};
\node at (7.8,6) {$\theta_5$};
\node at (1.7,6.3) {$\theta_1$};
\node at (3.1,4) {$\theta_2$};
\node at (5.5,3.8) {$\theta_4$};
\node at (3.1,6.3) {$\theta_3$};
}
\end{tikzpicture}
\caption{Seven values in Cicho\'n's diagram. This constellation was forced in~\cite[Thm~5.3]{BCM} where $\aleph_1\leq\theta_0\leq\theta_1 \leq\theta_2 \leq \theta_3\leq\theta_4$ are regular cardinals and $\theta_5$ is a cardinal such that $\theta_4\leq\theta_5=\theta_5^{{<}\theta_2}$.
}
\label{BCM}
\end{figure}

The proof of~\autoref{Mainthm} is settled by the construction of a  ${<}\lambda_3$-uf-extendable matrix iteration. Details are provided in~\autoref{Secmain}.


Though it is well-known the consistency with ZFC of each of $\add(\Nwf)=\bfrak<\add(\SNwf)$ and  $\add(\Nwf)<\bfrak=\add(\SNwf)$ (see~\cite[Sec.~8.4B]{BJ}), it is not known any model where the values of $\add(\Nwf)$, $\bfrak$, and $\add(\SNwf)$ are pairwise different. Intending to solve the latter, we introduce a $\sigma$-linked poset that increases $\add(\SNwf)$ and does not add dominating reals, which is used to prove our next main result:

\begin{teorema}[{\autoref{Thm:ufQsn}}]\label{Thm:faddSN}
Let $\lambda_0\leq\lambda_3\leq\lambda_4$ be uncountable regular cardinals, and $\lambda_5$ a cardinal such that $\lambda_5\geq\lambda_4$ and $\cof([\lambda_5]^{<\lambda_0})=\lambda_5 = \lambda_5^{\aleph_0}$. Then there is some ccc poset forcing~\autoref{Fig:addsn}. 
\end{teorema}

\begin{figure}[ht!]
\centering
\begin{tikzpicture}[scale=1.15]
\small{
\node (aleph1) at (-1,2.3) {$\aleph_1$};
\node (addn) at (0.5,2.3){$\add(\Nwf)$};
\node (covn) at (0.5,7.5){$\cov(\Nwf)$};
\node (nonn) at (10,2.3) {$\non(\Nwf)$} ;
\node (cfn) at (10,7.5) {$\cof(\Nwf)$} ;
\node (addm) at (3.8,2.3) {$\add(\Mwf)$} ;
\node (covm) at (6.5,2.3) {$\cov(\Mwf)$} ;
\node (nonm) at (3.8,7.5) {$\non(\Mwf)$} ;
\node (cfm) at (6.5,7.5) {$\cof(\Mwf)$} ;
\node (b) at (3.8,4.8) {$\bfrak$};
\node (d) at (6.5,4.8) {$\dfrak$};
\node (c) at (12,7.5) {$\cfrak$};
\node (covsn) at (5.25,5.75) {$\cov(\SNwf)$};
\node (addsn) at (2.5,4.8) {$\add(\SNwf)$};
\node (nonma) at (2.2,6.2) {$\non(\MAwf)$};
\node (nonsn) at (8.25,2.3) {$\non(\SNwf)$} ;

\draw (aleph1) edge[->] (addn)
      (addn) edge[->] (covn)
      (covn) edge [->] (nonm)
      (nonm)edge [->] (cfm)
      (cfm)edge [->] (cfn)
      (cfn) edge[->] (c);

\draw   (covm) edge [->]  (nonsn)
(nonsn) edge [->]  (nonn)
(addm) edge [->]  (covm)
(addn) edge [->]  (addm)
(addn) edge [->]  (addsn)
(nonn) edge [->]  (cfn)
 (addm) edge [->] (b)
      (b)  edge [->] (nonm);
\draw (covm) edge [->] (d)
      (d)  edge[->] (cfm);
\draw (b) edge [->] (d);

\draw (addn) edge [->] (nonma);
\draw (nonma) edge [->] (nonm);

\draw   (addsn) edge [line width=.15cm,white,-]  (covsn)
(addsn) edge [->]  (covsn);
\draw   (covsn) edge [line width=.15cm,white,-]  (c)
(covsn) edge [->]  (c);
\draw   (addsn) edge [line width=.15cm,white,-]  (nonsn)
(addsn) edge [->]  (nonsn);
(covsn) edge [->]  (c);
\draw   (nonma) edge [line width=.15cm,white,-]  (nonsn)
(nonma) edge [->]  (nonsn);
\draw   (covn) edge [line width=.15cm,white,-]  (covsn)
(covn) edge [->]  (covsn);
\draw[color=sug,line width=.05cm] (-0.5,1.8)--(-0.5,8);
\draw[color=sug,line width=.05cm] (-0.5,4.3)--(1.4,4.3);
\draw[color=sug,line width=.05cm] (1.4,4.3)--(3.4,4.3);
\draw[color=sug,line width=.05cm] (3.4,4.3)--(3.4,5.2);
\draw[color=sug,line width=.05cm] (3.4,5.2)--(6.1,5.2);
\draw[color=sug,line width=.05cm] (5.4,6.2)--(5.4,8);
\draw[color=sug,line width=.05cm] (5.4,6.2)--(6.1,6.2);
\draw[color=sug,line width=.05cm] (6.1,5.2)--(6.1,6.2);
\draw[color=sug,line width=.05cm] (5.4,1.8)--(5.4,5.2);
\draw[color=sug,line width=.05cm] (5.4,4)--(12,4);
\draw[circle, fill=yellow,color=yellow] (2.4,3.5) circle (0.4);
\draw[circle, fill=yellow,color=yellow] (1.2,5.5) circle (0.4);
\draw[circle, fill=yellow,color=yellow] (8.2,3.2) circle (0.4);
\draw[circle, fill=yellow,color=yellow] (8.2,5.8) circle (0.4);
\node at (2.4,3.5) {$\lambda_0$};
\node at (1.2,5.5) {$\lambda_3$};
\node at (8.2,3.2) {$\lambda_4$};
\node at (8.2,5.8) {$\lambda_5$};
}

\end{tikzpicture}
\caption{Constellation forced in~\autoref{Thm:faddSN}.}
\label{Fig:addsn}
\end{figure}
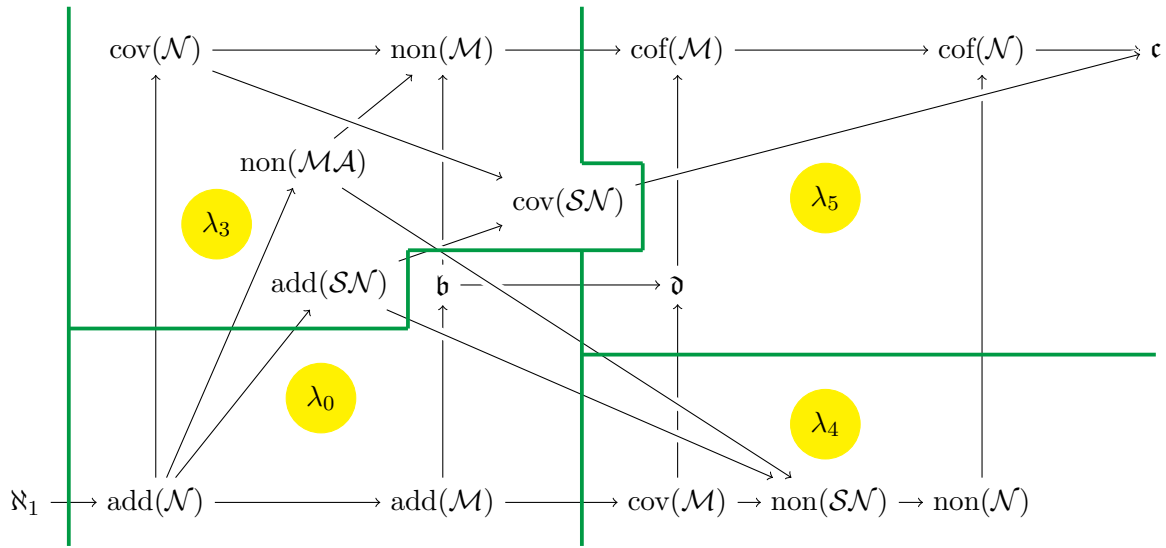

At the end of~\autoref{Secmain}, we shall show the previous theorem by using uf-extendable matrix iterations.

\section{ZFC results}\label{Sec:zfc}

We employ the combinatorial properties of the null-additive and meager-additive sets referred to in~\autoref{sec:intro} to prove~\autoref{MainThm1}, and to extract a relational system $\Rbf_b$ (\autoref{RelsAM}), parametrized with $b\in(\omega+1)^\omega$, present in the characterization of $\non(\MAwf)$ from \autoref{addtch}, which will be useful to prove consistency results.  
In addition, we present connections of $\Rbf_b$ with measure and category,  
and provide our proof of \autoref{MainThm1c}.

Recall that $\Ior$ denotes the set of partitions of $\omega$ into finite non-empty intervals. We use the following strengthening of \autoref{ch:NA-MA}~\ref{ch:NA}.

\begin{theorem}\label{ch:NA2}
    Let $h\in\baire$ be a function diverging to infinity. For $X\subseteq \cantor$, $X\in\NAwf$ iff the statement below holds.
    \begin{enumerate}[label = $(\text{\faThemeisle})_h$]
        \item For all $I \in\Ior$ there is some $\varphi\in\prod_{n\in \omega}\pts(2^{I_n})$ such that $\forall n\in \omega\colon |\varphi(n)|\leq h(n)$ and $X\subseteq H_\varphi$.
    \end{enumerate}
\end{theorem}
\begin{proof}
    See~\cite[Thm~3.2]{bartJudah} and \cite[Thm.~2.7.18]{BJ}, but we provide an argument in connection with \autoref{ch:NA-MA}~\ref{ch:NA} (which is this theorem for $h = \id_\omega$). It is enough to show that, for any $h,h'\in\baire$, if $h$ diverges to infinity, then $(\text{\faThemeisle})_{h'}$ implies $(\text{\faThemeisle})_h$. We use an argument similar to~\cite[Thm.~3.16]{CMlocalc}. Fix $I\in\Ior$. Pick an increasing function $g\in\baire$ with $g(0)=0$ such that, for $0<k<\omega$, $h(n)\geq h'(k)$ for all $n\geq g(k)$. Refine $I$ to $I'\in\Ior$ such that $I'_k:=\bigcup\set{I_n}{g(k)\leq n <g(k+1)}$. So, by $(\text{\faThemeisle})_h'$, there is some $\varphi'\in\prod_{k\in \omega}\pts(2^{I'_k})$ such that $|\varphi'(k)|\leq h'(k)$ for all $k<\omega$, and $X\subseteq H_{\varphi'}$. Set
    \[\varphi(n):=
    \begin{cases}
        \set{s\frestr I_n}{s\in\varphi'(k)} & \text{if $g(k)\leq n<g(k+1)$ for some (unique) $k\geq 1$,}\\
        \emptyset & \text{if $n<g(1)$.}
    \end{cases}\]
    It is clear that $\varphi\in\prod_{n\in \omega}\pts(2^{I_n})$ and $|\varphi(n)| = 0 \leq h(n)$ whenever $n<g(1)$; in the case that $n$ is as in the first case above, $|\varphi(n)|\leq |\varphi'(k)| \leq h'(k)\leq h(n)$. Moreover, $X\subseteq H_{\varphi'}\subseteq H_\varphi$, which finishes the proof.
\end{proof}

\begin{proof}[Proof of \autoref{MainThm1}]
It just suffices to prove that $\non(\NAwf)\leq\add(\NAwf)$. Suppose that $\set{X_\alpha}{\alpha<\kappa}\subseteq\NAwf$ and $\kappa<\non(\NAwf)$. We prove that $\bigcup_{\alpha<\kappa}X_\alpha\in\NAwf$. Let $I\in\Ior$. Then by using~\autoref{ch:NA-MA}~\ref{ch:NA}, for each $\alpha<\kappa$ choose $\varphi^I_\alpha\in\Swf(2^{I},\id_\omega)$ such that $X_\alpha\subseteq H_{\varphi^I_\alpha}$. Let $b^I:=\seq{[2^{I_n}]^{\leq n}}{n\in\omega}$. Since $\seq{\varphi^I_\alpha}{\alpha<\kappa}\subseteq\prod b^I$, by applying~\autoref{chNPaw} there is some $\psi^I\in\Swf(b^I,\id_\omega)$ such that $\forall\alpha<\kappa\colon \varphi^I_\alpha\in^*\psi^I$. Set $\varphi^I\in\Swf(2^{I},\id_\omega^2)$ by $\varphi^I(n):=\bigcup\psi^I(n)$. As a result, we get that $\forall\alpha<\kappa\,\forall^\infty n\colon \varphi^I_\alpha(n) \subseteq\varphi^I(n)$. This implies that $\bigcup_{\alpha<\kappa}X_\alpha\subseteq H_{\varphi^I}.$ Thus, $\bigcup_{\alpha<\kappa}X_\alpha\in\NAwf$ by \autoref{ch:NA2} (applied to $h(n):= n^2$). 
\end{proof}



We now recall the combinatorial description of the meager ideal from Bartoszy\'nski, Just, and Scheepers~\cite{BWS}, which was originally introduced by Talagrand~\cite{Tal98}. First, we establish some preliminary results. Instead of dealing with all meager sets, we only need to consider suitable cofinal families.

\begin{definition}
Let $I\in\Ior$ and let $x\in\cantor$. Define 
\[B_{x,I}:=\set{y\in\cantor}{\forall^\infty n\in\omega\colon y{\upharpoonright}I_n\neq x{\upharpoonright}I_n}.\]
For $n\in\omega$, define
\[B_{x,I}^n:=\set{y\in\cantor}{\forall m\geq n\colon x{\upharpoonright}I_m\neq y{\upharpoonright}I_m}.\]
Then $B_{x,I}^m\subseteq B_{x,I}^n$ whenever $m < n <\omega$. Thus, $B_{x,I}=\bigcup_{n\in\omega}B_{x,I}^n$. 

Denote by $B_{I}$ the set $B_{0,I}=\set{y\in\cantor}{\forall^\infty n\in\omega\colon y{\upharpoonright}I_n\neq0)}$.    
\end{definition}

A pair $(x,I)\in 2^\omega\times \Ior$ is known as a \emph{chopped real}, and these are used to produce a cofinal family of meager sets.
It is clear that $B_{x,I}$ is a meager subset 
of $\cantor$ (see, e.g.~\cite{blass}).

\begin{theorem}[{Talagrand~\cite{Tal98}}, see e.g.~{\cite[Prop.~13]{BWS}}]\label{ChM}
For every meager set $F\subseteq\cantor$ and $I\in\Ior$ there are $x\in\cantor$ and $I'\in\Ior$ such that $F\subseteq B_{I',x}$ and each $I'_n$ is the union of finitely many $I_k$'s.
\end{theorem}

\begin{lemma}[{\cite[Prop~9]{BWS}}]\label{baseM}
For $x, y\in\cantor$ and for $I, J\in\Ior$, the following statements are equivalent:
\begin{enumerate}[label=\normalfont(\arabic*)]
    \item $B_{I,x}\subseteq B_{J,y}$.
    \item $\forall^\infty n<\omega\,\exists k<\omega\colon I_k\subseteq J_n$ and $x{\upharpoonright}I_k=y{\upharpoonright}I_k$.
\end{enumerate}
\end{lemma}

We now review some basic notation about relational systems. A~\emph{relational system} is a triple $\Rbf=\la X, Y, \sqsubset\ra$ where $\sqsubset$ is a relation and $X$ and $Y$ are non-empty sets. Such a relational system has two cardinal characteristics associated with it:
\begin{align*}
    \bfrak(\Rbf)&:=\min\set{|F|}{F\subseteq X\text{\ and }\neg\exists y\in Y\, \forall x\in F\colon x \sqsubset y}\\
    \dfrak(\Rbf)&:=\min\set{|D|}{D\subseteq Y\text{\ and }\forall x\in X\, \exists y\in D\colon x \sqsubset y}.
\end{align*}
We also define the \emph{dual} $\Rbf^\perp := \la Y,X,\sqsubset^\perp\ra$ where $y \sqsubset^\perp x$ means $x \not\sqsubset y$. Note that $\bfrak(\Rbf^\perp) = \dfrak(\Rbf)$ and $\dfrak(\Rbf^\perp) = \bfrak(\Rbf)$.

Given another relational system $\Rbf'=\la X',Y',R'\ra$,  say that a pair $(\Psi_-,\Psi_+)\colon\Rbf\to\Rbf'$ is a \emph{Tukey connection from $\Rbf$ into $\Rbf'$} if 
 $\Psi_-\colon X\to X'$ and $\Psi_+\colon Y'\to Y$ are functions such that  $\forall\, x\in X\ \forall\, y'\in Y'\colon \Psi_-(x) \sqsubset' y' \Rightarrow x \sqsubset \Psi_+(y')$. Say that $\Rbf$ is \emph{Tukey below} $\Rbf'$, denoted by $\Rbf\leqT\Rbf'$, if there is a Tukey connection from $\Rbf$ to $\Rbf'$. 
 Say that $\Rbf$ is \emph{Tukey equivalent} to $\Rbf'$, denoted by $\Rbf\eqT\Rbf'$, if $\Rbf\leqT\Rbf'$ and $\Rbf'\leqT\Rbf$. It is well-known that $\Rbf\leqT\Rbf'$ implies $\bfrak(\Rbf')\leq\bfrak(\Rbf)$ and $\dfrak(\Rbf)\leq\dfrak(\Rbf')$. Hence, $\Rbf\eqT\Rbf'$ implies $\bfrak(\Rbf')=\bfrak(\Rbf)$ and $\dfrak(\Rbf)=\dfrak(\Rbf')$. 

 \begin{example}\label{ex:CI}
 It is well-known that, for any ideal $\Iwf$ on $X$, via the relational system $\Cbf_\Iwf:=\la X,\Iwf,\in\ra$, $\bfrak(\Cbf_\Iwf)=\non(\Iwf)$ and $\dfrak(\Cbf_\Iwf)=\cov(\Iwf)$. 
 \end{example}

\begin{example}\label{b-d}
Note that $\leq^*$ is a directed preorder on $\baire$, where $x\leq^* y$ means $\forall^\infty n<\omega\colon x(n)\leq y(n)$. We think of $\baire$ as the relational system with the relation $\leq^*$. Then $\bfrak:=\bfrak(\baire)$ and $\dfrak:=\dfrak(\baire)$ are the well known \emph{bounding number} and \emph{dominating number}, respectively.
\end{example}

\begin{example}\label{Charb-d}
Define the following relation on $\Ior$:
   \[ I \sqsubseteq J \text{ iff } \forall^\infty n<\omega\, \exists m<\omega\colon I_m\subseteq J_n.
   \]
Note that $\sqsubseteq$ is a directed preorder on $\Ior$, so we think of $\Ior$ as the relational system with the relation $\sqsubseteq$. 
In Blass~\cite{blass}, it is proved that $\Ior \eqT \omega^\omega$. Hence, $\bfrak=\bfrak(\Ior)$ and $\dfrak=\dfrak(\Ior)$. 
\end{example}

\begin{example}\label{exm:Lc}
    The cardinals in \autoref{defloc} can be defined by the relational system $\Lc(b,h):=\la\prod b,\Swf(b,h),\in^*\ra$, i.e.\ $\bfrak(\Lc(b,h))=\blc_{b,h}$ and $\dfrak(\Lc(b,h)) := \dlc_{b,h}$.
\end{example}

We now introduce the relational systems involved in the characterization of $\non(\MAwf)$ in~\autoref{addtch}.

\begin{definition}\label{RelsAM} Fix $b\in(\omega+1)^\omega$.
\begin{enumerate}[label=(\arabic*)]
\item\label{RelsAMone} For  $I\in\Ior$, $f, h\in\baire$, define 
    \[f  \sqrb (I,h)\textrm{\ iff\ }\forall^\infty n\in\omega\,\exists k\in I_n\colon f(k)=h(k).\]
    \item\label{RelsAMtwo} Define the relational system $\Rbf_b:=\la\prod b,\Ior\times\prod b,\sqrb\ra$. When $b(n)=\omega$ for all $n<\omega$, we denote this relational system by $\Rbf_\omega$.
\end{enumerate}
In the context of $\Rbf_b$, we will always consider that $b(n)>0$ for all $n$, even if we just write ``$b\in(\omega+1)^\omega$" (or $b\in\baire$).\footnote{In~\cite{cardonaRb}, the cardinals $\bfrak(\Rbf_b)$ and $\dfrak(\Rbf_b)$ are denoted by $\bfrak_{b}^\mathsf{eq}$ and $\dfrak_{b}^\mathsf{eq}$, respectively.}
\end{definition}

\begin{remark}\label{rem:Rb-nonM}
    Notice that, for fixed $(I,h)\in \Ior\times \prod b$, $\set{f\in\prod b}{f\sqrb (I,h)}$ is meager whenever $b\geq^* 2$, so $\Cbf_\Mwf\leqT\Rbf_b$, which implies $\bfrak(\Rbf_b)\leq\non(\Mwf)$ and $\cov(\Mwf)\leq \dfrak(\Rbf_b)$. On the other hand, if $b\ngeq^*2$ then we can find some $(I,h)\in \Ior\times \prod b$ such that $f\sqrb (I,h)$ for all $f\in\prod b$, so $\dfrak(\Rbf_b) = 1$ and $\bfrak(\Rbf_b)$ is undefined.
\end{remark}

\begin{fact}\label{RelsAMhree}
For $b\in(\omega+1)^\omega$, 
$\Rbf_b\eqT\la\prod b,\Ior\times\baire,\sqrb\ra$. 
As a consequence, if $b'\in(\omega+1)^\omega$ and
$b\leq^* b'$, then $\Rbf_b\leqT\Rbf_{b'}$. In particular, $\bfrak(\Rbf_{b'}) \leq \bfrak(\Rbf_b)$ and $\dfrak(\Rbf_b)\leq \dfrak(\Rbf_{b'})$.   
\end{fact}

We now aim to prove the following reformulation of \autoref{addtch}:

\begin{theorem}[{\cite[Thm.~2.2]{bartJudah}}]\label{th:BJ}
$\non(\MAwf)=\min\set{\bfrak(\Rbf_b)}{b\in\baire}$.  
\end{theorem}

We follow the proof of the cited reference under our notation using $\Rbf_b$. 
The following lemma establishes one of the inequalities.

\begin{lemma}\label{RsAMone}
Let $b\in\baire$. Then $\Rbf_{b}\leqT \Cbf_\MAwf$. 
 In particular, \[\non(\MAwf)\leq\min\set{\bfrak(\Rbf_b)}{b\in\baire} \text{ and } \sup\set{\dfrak(\Rbf_b)}{b\in\baire}\leq\cov(\MAwf).\]   
\end{lemma}
\begin{proof}
Given $b\in\baire$, thanks to \autoref{RelsAMhree} we may assume that there is some $I^b\in\Ior$ such that $b(n)=2^{|I_n^b|}$. 
Then, 
we can identify numbers ${<}b(n)$ with $0$-$1$ sequences of length $|I_n^b|$. 
We have to find maps $\Psi_-\colon\prod b\to \cantor$  and  $\Psi_+\colon \MAwf\to \Ior\times\prod b$ such that, for any $ f\in\prod b$ and for any $X\in\MAwf$, $\Psi_-(f)\in X$ implies $f\sqrb \Psi_+(X)$.

For $f\in\prod b$ define $x_f^{I^b}\in\cantor$ by 
\[x_f^{I^b}={\underbrace{f(0)}_{\text{length }|I_0^b|}}^{\frown} \cdots\cdots{}^{\frown}{\underbrace{f(n)}_{\text{length } |I_n^b|}}^{\frown}\cdots,\]  
so put $\Psi_-(f):= x^{I^b}_{f}$. 

For $X\in \MAwf$, $X+B_{I^b}\in\Mwf$. Note that \[X+B_{I^b}=\bigcup_{x\in X}B_{x,I^b}.\]
Then, by~\autoref{ChM}, there are $y\in\cantor$ and $J\in\Ior$ such that 
\[\bigcup_{x\in X}B_{x,I^b}\subseteq B_{y, J}.\]
Let $h\in\prod b$ such that $y=x_{h}^{I^b}$ (recall that $b(n)=2^{|I^b_n|}$), so put $\Psi_+(X):=(J',h)$ where
\[k\in J'_n \text{ iff } \min J_n<\max I^b_k \leq \max J_n.\]

It remains to prove that, for any $ f\in\prod b$ and for any $X\in\MAwf$, $\Psi_-(f)\in X$ implies $ f\sqrb\Psi_+(X)$. Suppose that $x_{f}^{I_b}\in X$ and $\Psi_+(X) = (J',h)$. Then $B_{x_{f}^{I_b},I^b}\subseteq B_{x_{h}^{I^b},J}$. Hence, by using~\autoref{baseM}, 
\[\forall^\infty n\,\exists k\colon I_k^b\subseteq J_n\text{ and }x_f^{I^b}{\upharpoonright}I_k^b=x_{h}^{I^b}{\upharpoonright}I_k^b.\]
Since $I^b_k\subseteq J_n$ implies $k\in J'_n$, the equation above implies that $f\sqrb(J',h)$. 
\end{proof}

To prove the converse inequality of \autoref{th:BJ}, we employ products of relational systems.

\begin{definition}\label{def:prodrel}
Let $\overline{\Rbf} = \seq{\Rbf_i}{i\in K}$ be a sequence of relational systems $\Rbf_i = \la X_i,Y_i,\sqsubset_i\ra$. Define $\prod \overline{\Rbf} = \prod_{i\in K} \Rbf_i:=\left\la \prod_{i\in K}X_i, \prod_{i\in K}Y_i, \sqsubset^\times \right\ra$ where $x \sqsubset^\times y$ iff $x_i\sqsubset_i y_i$ for all $i\in K$.

For two relational systems $\Rbf$ and $\Rbf'$, write $\Rbf\times \Rbf'$ to denote their product, and when $\Rbf_i = \Rbf$ for all $i\in K$, we write $\Rbf^K := \prod \overline{\Rbf}$.
\end{definition}

\begin{fact}[{\cite{CarMej23}}]\label{products}
Let $\overline{\Rbf}$ be as in \autoref{def:prodrel}. Then $\sup_{i\in K}\dfrak(\Rbf_i)\leq\dfrak(\prod \overline{\Rbf})\leq\prod_{i\in K}\dfrak(\Rbf_i)$ and $\bfrak(\prod \overline{\Rbf})=\min_{i\in K}\bfrak(\Rbf_i)$.
\end{fact}

In the following result, not only do we complete the proof of \autoref{th:BJ}, but we find an upper bound of $\cov(\MAwf)$ that will be useful to show the consistency with ZFC of $\cov(\MAwf)<\non(\Nwf)$ in \autoref{Secmain}.

\begin{lemma}\label{prodRb}
    For any dominating family $D\subseteq\omega^\omega$, 
    $\Cbf_{\MAwf}\leqT \prod_{b\in D}\Rbf_b$. In particular, $\min_{b\in D}\bfrak(\Rbf_b) \leq\non(\MAwf)$ and $\cov(\MAwf)\leq \dfrak\left(\prod_{b\in D}\Rbf_b\right)$.
\end{lemma}
\begin{proof}
    Without loss of generality, we may assume that there is some $\Ior$-dominating family $D_0$, i.e.\ $\forall I\in \Ior\, \exists J\in D_0\colon I\sqsubseteq J$, such that for each $b\in D$ there is some $I\in D_0$ such that $b = 2^I$, i.e.\ $b(n) = 2^{I_n}$ for all $n<\omega$. 
    
    Define $\Psi_-\colon 2^\omega\to \prod_{I\in D_0} 2^I$ by $\Psi_-(x)(I):=\seq{x\frestr I_n}{n<\omega}$. And define $\Psi_+\colon \prod_{I\in D_0}\Ior\times 2^I\to \MAwf$ such that, for $z=\seq{(J^I,z^I)}{I\in\Ior}$,
    \[\Psi_+(z) := \set{x\in 2^\omega}{\forall I\in D_0\, \forall^\infty n<\omega\, \exists k\in J^I_k\colon x\frestr I_k = z^I(k)}.\]
    For each $I\in D_0$ let $I'_n:=\bigcup_{k\in J^I_n}I_k$ and $y^I\in 2^\omega$ the concatenation of all the $z^I(k)\in 2^{I_k}$ for $k<\omega$, i.e., $y^I\frestr I_k = z^I(k)$. Then $I':=\seq{I'_n}{n<\omega}\in \Ior$, $I\sqsubseteq I'$ and
    \[\forall x \in \Psi_+(z)\, \forall^\infty n<\omega\, \exists k<\omega\colon I_k\subseteq I'_n \text{ and }x\frestr I_k = y^I\frestr I_k.\]
    Therefore, by \autoref{ch:NA-MA}, $\Psi_+(z)\in \MAwf$.

    It is clear that $(\Psi_-,\Psi_+)$ is the required Tukey connection.
\end{proof}






As we mentioned in~\autoref{sec:intro}, $\cof(\Nwf)=\aleph_1$ and $\cov(\SNwf)=\cfrak=\aleph_2$ holds in Sacks model (see~\cite[Thm.~4.7]{CMRM}). There,  $\cov(\MAwf)=\aleph_2$ because $\cov(\SNwf) \leq \cov(\MAwf) \leq \cov(\NAwf)$.

We now focus on proving that $\sup\set{\dfrak(\Rbf_b)}{b\in\omega^\omega} \leq \cof(\Mwf)$. Therefore, $\sup\set{\dfrak(\Rbf_b)}{b\in\omega^\omega} =\aleph_1$ in Sacks model, so we cannot dualize \autoref{th:BJ}, i.e.\ $\thzfc$ (if consistent) cannot prove that $\cov(\MAwf)$ equals $\sup\set{\dfrak(\Rbf_b)}{b\in \omega^\omega}$.
Another consequence is that no classical cardinal characteristics of the continuum (different from  $\cfrak$) is an upper bound of $\cov(\MAwf )$.

We use the composition of relational systems to prove our claim.

\begin{definition}[{\cite[Sec.~4]{blass}}]\label{compTK}
Let $\Rbf_e=\la X_e,Y_e,\sqsubset_e\ra$ be a relational system for $e\in\{0,1\}$. 
The \emph{composition of $\Rbf_0$ with $\Rbf_1$} is defined by $(\Rbf_0;\Rbf_1):=\la X_0\times X_1^{Y_0}, Y_0\times Y_1, \sqsubset_* \ra$ where
\[(x,f)\sqsubset_*(y,b)\text{ iff }x \sqsubset_0 y\text{ and }f(y) \sqsubset_1 b.\]
\end{definition}

\begin{fact}\label{ex:compleqT}
Let $\Rbf_i$ be a relational system for $i<3$. If $\Rbf_0\leqT\Rbf_1$, then $\Rbf_0\leqT \Rbf_1\times \Rbf_2 \leqT (\Rbf_1;\Rbf_2)$ and 
$\Rbf_1\times \Rbf_2
\eqT \Rbf_2\times \Rbf_1$.
\end{fact}

The following theorem describes the effect of the composition on cardinal characteristics.

\begin{theorem}[{\cite[Thm.~4.10]{blass}}]\label{blascomp}
Let $\Rbf_e$ be a relational system for $e\in\{0,1\}$. Then $\bfrak(\Rbf_0;\Rbf_1) =\min\{\bfrak(\Rbf_0),\bfrak(\Rbf_1)\}$ and $\dfrak(\Rbf_0;\Rbf_1)=\dfrak(\Rbf_0)\cdot\dfrak(\Rbf_1)$. 
\end{theorem}

We introduce the following relational system for combinatorial purposes.

\begin{definition}\label{def:Ed}
  Let $b:=\seq{ b(n)}{ n<\omega}$ be a sequence of non-empty sets. Define the relational system $\Ed_b:=\la\prod b,\prod b,\neq^\infty\ra$ where $x=^\infty y$ means $x(n)=y(n)$ for infinitely many $n$. The relation $x\neq^\infty y$ means that \emph{$x$ and $y$ are eventually different}. Denote $\balc_{b,1}:=\bfrak(\Ed_b)$ and $\dalc_{b,1}:=\dfrak(\Ed_b)$.
  When $b(n)=\omega$ for all $n<\omega$, denote the relational system by $\Ed$ and its associated cardinal characteristics by $\balc_{\omega,1}$ and $\dalc_{\omega,1}$.
\end{definition}

Recall the following characterization of the cardinal characteristics associated with $\Mwf$. It is well-known that $\balc_{\omega,1}=\non(\Mwf)$ and $\dalc_{\omega,1}=\cov(\Mwf)$ (Bartoszy\'nski and Miller, see e.g.\ \cite[Thm.~5.1]{CMlocalc}). 
The one for $\add(\Mwf)$ below is due to Miller~\cite{Miller}.

\begin{theorem}[{\cite[Sec.~3.3]{CM}}]\label{charM2}
    \[\add(\Mwf) = \min(\{\bfrak\}\cup\set{\dalc_{b,1}}{b\in \omega^\omega}) \text{ and } \cof(\Mwf) = \sup(\{\dfrak\}\cup\set{\balc_{b,1}}{b\in \omega^\omega})\]
\end{theorem}

Therefore, 
to settle our claim, it suffices to prove:

\begin{theorem}\label{dRbup}
    For $b\in(\omega+1)^\omega$, $\Ed_b^\perp\leqT\Rbf_b\leqT (\Ed_b^\perp;\Ior)$. In particular, $\balc_{b,1}\leq \dfrak(\Rbf_b)\leq \max\{\balc_{b,1},\dfrak\}$ and $\min\{\dalc_{b,1},\bfrak\}\leq \bfrak(\Rbf_b)\leq \dalc_{b,1}$.
\end{theorem}
\begin{proof}
    The Tukey-inequality $\Ed_b^\perp\leqT\Rbf_b$ is immediate from the definitions, so we focus on the second one. 
    Define $\Psi_-\colon \prod b \to \prod b \times \Ior^{\prod b}$ by $\Psi_-(x):= (x,F_x)$ where, for $y\in \prod b$, if $y =^\infty x$ then $F_x(y):= I^y_x \in \Ior$ is chosen such that $\forall k<\omega\, \exists i\in I^y_{x,k}\colon y(i) = x(i)$; otherwise, $F_x(y)$ can be anything (in $\Ior$).

    Define $\Psi_+\colon \prod b\times\Ior\to \Ior\times \prod b$ by $\Psi_+(y,J) = (J,y)$. We check that $(\Psi_-,\Psi_+)$ is a Tukey connection. Assume that $x,y\in\prod b$, $J\in\Ior$ and that $\Psi_-(x)\sqsubset_* (y,J)$, i.e.\ $x =^\infty y$ and $I^y_x \sqsubseteq J$. Since each $I^y_{x,k}$ contains a point where $x$ and $y$ coincide, $I^y_x \sqsubseteq J$ implies that, for all but finitely many $n<\omega$, $J_n$ contains a point where $x$ and $y$ coincide, which means that $x\sqrb (J,y) = \Psi_+(y,J)$.
\end{proof}

\begin{corollary}\label{supRbM}
    For all $b\in(\omega+1)^\omega$, $\add(\Mwf)\leq\bfrak(\Rbf_b)$ and $\dfrak(\Rbf_b) \leq \cof(\Mwf)$.
\end{corollary}

Note that $\add(\Mwf) \leq \min\set{\bfrak(\Rbf_b)}{b\in\omega^\omega}$ already follows from \autoref{bas:NA} and \autoref{th:BJ}.

\begin{remark}\label{FSRb}
    For $b\in\omega^\omega$, $\balc_{b,1}\leq \non(\Mwf)$ and $\cov(\Mwf)\leq \dalc_{b,1}$. On the other hand, after a FS (finite support) iteration of uncountable cofinality lentgh of ccc non-trivial posets, $\non(\Mwf) \leq \cov(\Mwf)$, which implies by \autoref{dRbup} that $\bfrak \leq \bfrak(\Rbf_b)$ and $\dfrak(\Rbf_b)\leq \dfrak$. Hence, the consistency of $\bfrak(\Rbf_b)<\bfrak$ (and $\dfrak<\dfrak(\Rbf_b)$) cannot be obtained by FS iterations. The same applies to $\non(\MAwf) <\bfrak$.
\end{remark}

Concerning $\Rbf_\omega$, \autoref{dRbup} indicates that $\non(\Mwf) \leq \dfrak(\Rbf_\omega)\leq \max\{\non(\Mwf),\dfrak\} = \cof(\Mwf)$ and $\add(\Mwf)=\min\{\cov(\Mwf),\bfrak\}\leq \bfrak(\Rbf_\omega)\leq \cov(\Mwf)$. But more can be concluded.

\begin{lemma}\label{Romega-d}
   $\baire\leqT\Rbf_\omega$. 
\end{lemma}
\begin{proof}
    Let $\Psi_-\colon \baire\to \baire$ that sends $x\in\baire$ to some increasing $x'\in \baire$ above $x$ (everywhere). 
    Define $\Psi_+\colon \Ior\times\baire\to\baire$ such that, for $(I,h)\in\Ior\times\baire$, $\Psi_+(I,h)$ is the map in $\omega^\omega$ that sends each point in the interval $I_n$ to $\max_{k\in I_{n+1}}h(k)$. Then $(\Psi_-,\Psi_+)$ is a Tukey connection: if $x\in\baire$, $(I,h)\in \Ior\times\baire$ and $x'\sqsubset^\bullet (I,h)$, i.e.\ $\exists\, k_n\in I_n\colon x'(k_n)=h(k_n)$ for all but finitely many $n<\omega$, then $x(j)\leq x'(j) < x'(k_{n+1}) \leq \max_{k\in I_{n+1}}h(k)$ for all $j\in I_n$, i.e.\ $x\leq^* \Psi_+(I,h)$.
\end{proof}

\begin{theorem}[{\cite[Thm.~2.2.12]{BJ}}]\label{Romega}
   $\bfrak(\Rbf_\omega) = \add(\Mwf)$, and $\dfrak(\Rbf_\omega)=\cof(\Mwf)$.
\end{theorem}
\begin{proof}
    Immediately from \autoref{dRbup} and \autoref{Romega-d}.
\end{proof}

We also present further connections between $\Rbf_b$ and measure zero.

\begin{lemma}\label{RbE}
Let $b\in\omega^\omega$.  
    If $\sum_{k<\omega}\frac{1}{b(k)}<\infty$ then $\Cbf_\Ewf \leqT \Rbf_b$. In particular, $\bfrak(\Rbf_b) \leq \non(\Ewf)$ and $\cov(\Ewf)\leq \dfrak(\Rbf_b)$.
\end{lemma}
\begin{proof}
    For $0<m<\omega$, consider the \emph{uniform measure} $\mu_m$ on $m$, which assigns measure $\frac{1}{m}$ to each singleton. Consider the measure $\Lb_b$ on (the completion of) the Borel $\sigma$-algebra of $\prod b$ obtained as the product measure of the uniform measures of each $b(i)$. We can define $\Ewf(\prod b)$ on $\prod b$ similarly, and thanks to the map
    \[x\in \prod b \mapsto \sum_{n<\omega}\frac{x(i)}{\prod_{k\leq n}b(k)},\]
    we have that $\Ewf(\prod b)\eqT \Ewf$ and $\Cbf_{\Ewf(\prod b)}\eqT \Cbf_\Ewf$, see details in~\cite[Sec.~7.1 (arXiv version)]{GaMej}.\footnote{In this reference, $\Nwf_{\Fin}$ is the null-ideal, while $\Nwf^*_{\Fin}$ is $\Ewf$.} So, for this proof, we can work with $\Ewf = \Ewf(\prod b)$.

    Let $F\colon \prod b\to \prod b$ the identity function. For $(J,h)\in \Ior\times\prod b$, define
    \[G(J,h) := \largeset{x\in \prod b}{x \sqrb (J,h)}.\]
    It is enough to show that $G(J,h)\in \Ewf$ to conclude that $(F,G)$ is the desired Tukey connection. It is clear that $G(J,h)$ is an $F_\sigma$-set, since
    \[G(J,h) = \bigcup_{m<\omega}\bigcap_{n\geq m}\bigcup_{k\in J_n}A^{h(k)}_k, \text{ where }A^{\ell}_k := \largeset{x\in \prod b}{x(k) = \ell} \text{ for }\ell<b(k),\]
    and each $A^\ell_k$ is clopen. Since $\Lb_b(A^\ell_k) = \frac{1}{b(k)}$, we obtain
    \[\Lb_b(G(J,h))) \leq \lim_{m\to \infty}\prod_{n\geq m}\sum_{k\in J_n}\frac{1}{b(k)}.\]
    This limit above is $0$ because $\sum_{k<\omega}\frac{1}{b(k)}<\infty$.
\end{proof}

\begin{remark}
    In the case $\sum_{k<\omega}\frac{1}{b(k)} = \infty$, it is possible to find a $J\in \Ior$ such that $\Lb_b(G(J,h)) = 1$ for all $h\in\prod b$. Indeed,
    \[\prod b\menos G(J,h) = \bigcap_{m<\omega}\bigcup_{n\geq m}\bigcap_{k\in J_n}\left(\prod b\menos A^{h(k)}_k\right),\]
    so
    \[\Lb_b\left(\prod b\menos G(J,h)\right) \leq \lim_{m\to\infty}\sum_{n\geq m}\prod_{k\in J_n}\left(1 - \frac{1}{b(k)}\right) \leq \lim_{m\to\infty}\sum_{n\geq m} e^{-\sum_{k\in J_n}\frac{1}{b(k)}}.\]
    Then, it is enough to find a $J\in\Ior$ such that $\sum_{k\in J_n}\frac{1}{b(k)} \geq n$ for all $n<\omega$.
\end{remark}

The previous lemma and \autoref{th:BJ} imply that $\non(\MAwf)\leq \non(\Ewf)$ and $\cov(\Ewf)\leq \cov(\MAwf)$. Note that this also follows from Zindulka's result $\EAwf = \MAwf$.

We conclude this section by showing 
Pawlikowski's claim 
$\minLc \leq \add(\SNwf)$ (\autoref{MainThm1c}). 

We begin with some notation:

\begin{itemize}
    \item For $s\in 2^{<\omega}$, denote $[s]:=\set{x\in 2^\omega}{ s\subseteq x}$.

    \item For $\sigma \in (2^{<\omega})^\omega$, define $\hgt_\sigma\colon \omega\to\omega$ by $\hgt_\sigma(n):=|\sigma(n)|$ for all $n<\omega$, which we call the \emph{height of $\sigma$}. Also, define
        \[[\sigma]_\infty:=\set{x\in 2^\omega}{\exists^\infty n\colon \sigma(n)\subseteq x}.\]
\end{itemize}

\begin{definition}\label{def:SN}
A set \emph{$X\subseteq 2^\omega$ has strong measure zero} if
\[\forall f\in\omega^\omega\, \exists \sigma\in(2^{<\omega})^\omega\colon f\leq \hgt_\sigma \text{ and }X\subseteq\bigcup_{i<\omega}[\sigma(i)].\]
Denote by $\SNwf$ the collection of strong measure zero subsets of $2^\omega$.
\end{definition}

The following characterization of $\SNwf$ is quite practical.

\begin{lemma}\label{charSN}
    Let $X\subseteq2^\omega$ and let $D\subseteq\omega^\omega$ be a dominating family. Then $X\subseteq2^\omega$ has strong measure zero in $2^\omega$ iff  
    \[\forall f\in D\, \exists\sigma\in (2^{<\omega})^\omega\colon \hgt_\sigma=f\text{\ and\ } X\subseteq[\sigma]_\infty.\]
\end{lemma}

\begin{proof}[Proof of \autoref{MainThm1c}]
We show that, for any fixed $h\in\baire$ diverging to infinity and any dominating family $D\subseteq \baire$, $\SNwf \leqT \prod_{b\in D}\Lc(b,h)$, so we also have $\cof(\SNwf)\leq \dfrak\left(\prod_{b\in D}\Lc(b,h)\right)$.\footnote{However, this upper bound of $\cof(\SNwf)$ is not better than those explored in~\cite{CarMej23}.} 
Let $b_-(n):=\lfloor \log_2 b(n) \rfloor$ and fix an injection $s^n\colon 2^{b_-(n)}\to b(n)$ with left inverse $t^n$. 
First define $\Psi_-\colon \SNwf\to \prod_{b\in D}\prod b$. Let $X\in\SNwf$.  For $b\in D$, by \autoref{charSN} pick $\sigma^b_{X}\in(2^{<\omega})^\omega$ such that $\hgt_{\sigma^b_X}=b_-$ and $X\subseteq[\sigma^b_X]_\infty$, and set $\Psi_-(X):=\seq{\seq{s^n(\sigma^b_X(n))}{n<\omega}}{b\in D}$.

Now define $\Psi_+\colon \prod_{b\in D} \Swf(b,h)\to \SNwf$. 
Partition $\omega$ into intervals $I_n$ of length $h(n)$. Let $\bar\varphi = \seq{\varphi^b}{b\in D}\in \prod_{b\in D} \Swf(b,h)$. For each $b\in D$ pick $\psi^b\in\Swf(b,h)$ such that $\varphi^b(n) \subseteq \psi^b(n) = \set{\ell^b_i}{i\in I_n}$ (i.e.\ $|\psi^b(n)|=h(n)$) for all $n<\omega$. Set $\tau^b:=\seq{t^n(\ell^b_i)}{i<\omega}$ and $\Psi_+(\bar\varphi):=\bigcap_{b\in D}[\tau^b]_\infty$. We get that $\set{\hgt_{\tau^b}}{b\in D}$ forms a dominating family, so $\Psi_+(\bar\varphi)\in\SNwf$. Indeed, for $x\in\baire$, define $x'(n):= 2^{\max_{i\in I_n}x(i)}$, so $x'\leq^* b$ for some $b\in D$, i.e.\ $\max_{i\in I_n}x(i) \leq b_-(n)$ for all but finitely many $n$, thus $x(i) \leq b_-(n) = \hgt_{\tau^b}(i)$ for $i\in I_n$.

The pair $(\Psi_-,\Psi_+)$ is a Tukey-connection.
For $X\in\SNwf$ and $\bar\varphi\in\prod_{b\in D}\Lc(b,h)$, assume $\Psi_-(X)\sqsubset^\times \bar\varphi$, i.e.\ for all $b\in D$, for all but finitely many $n<\omega$, $s^n(\sigma^b_X(n)) \in \varphi^b(n)$, which means that $s^n(\sigma^b_X(n))=\ell^b_{i_n}$ for some $i_n\in I_n$. This implies $\sigma^b_X(n)=t^n(\ell^b_{i_n})=\tau^b(i_n)$. Therefore, $X\subseteq [\sigma^b_X]_\infty \subseteq [\tau^b]_\infty$ for any $b\in D$, so $X\subseteq \Psi_+(\bar\varphi)$.
\qedhere




\end{proof}

\section{Preservation and uf-linkedness}\label{Prelink}

For the reader’s convenience, we first recall the preservation properties that were developed for FS iterations of ccc posets by Judah and Shelah~\cite{JS} and Brendle~\cite{Br}, which were generalized in~\cite[Sect.~4]{CM}. We also review some recent tools from~\cite{CarMej23,BCM2} to control the cardinal characteristics associated with $\SNwf$ in forcing iterations,
and the notion of ultrafilter limits for forcing notions from~\cite{GMS,BCM}. 
These properties will be applied in the proof of our consistency results in \autoref{Secmain}. Furthermore, we introduce forcing notions to increase $\bfrak(\Rbf_b)$ and $\add(\SNwf)$, which have ultrafilter limits.

\begin{definition}\label{DefRelSys}
Let $\Rbf=\la X,Y,\sqsubset\ra$ be a relational system and let $\theta$ be a cardinal.  
\begin{enumerate}[label=\rm(\arabic*)]
  \item For a set $M$,
  \begin{enumerate}[label=\rm(\roman*)]
      \item An object $y\in Y$ is \textit{$\Rbf$-dominating over $M$} if $x\sqsubset y$ for all $x\in X\cap M$. 
      
      \item An object $x\in X$ is \textit{$\Rbf$-unbounded over $M$} if it $\Rbf^\perp$-dominating over $M$, that is, $x\not\sqsubset y$ for all $y\in Y\cap M$. 
  \end{enumerate}
  
  
  \item A family $\set{x_i}{i\in I}\subseteq X$ is \emph{strongly $\theta$-$\Rbf$-unbounded} if 
  $|I|\geq\theta$ and, for any $y\in Y$, $|\set{i\in I }{x_i\sqsubset y}|<\theta$.
\end{enumerate}
\end{definition}

The existence of strongly unbounded families is equivalent to a Tukey connection.

\begin{lemma}[{\cite[Lem.~1.16]{CM22}}]\label{unbT}
    Let $\Rbf=\la X,Y,\sqsubset\ra$ be a relational system, $\theta$ be an infinite cardinal, and $I$ be a set of size ${\geq}\theta$.
    \begin{enumerate}[label = \normalfont (\alph*)]
        \item $\Cbf_{[I]^{<\theta}}\leqT \Rbf$ iff there exists a strongly $\theta$-$\Rbf$-unbounded family $\set{x_i}{i\in I}$.

        \item $\bfrak(\Rbf)\geq\theta$ iff $\Rbf\leqT \Cbf_{[X]^{<\theta}}$.
    \end{enumerate}
\end{lemma}

We look at the following type of well-defined relational systems.

\begin{definition}\label{def:Prs}
Say that $\Rbf=\langle X,Y,\sqsubset\rangle$ is a \textit{Polish relational system (Prs)} if
\begin{enumerate}[label=\rm(\arabic*)]
\item\label{def:Prs1} $X$ is a Perfect Polish space,
\item\label{def:Prs2} $Y$ is a non-empty analytic subspace of some Polish $Z$, and
\item\label{def:Prs3} $\sqsubset=\bigcup_{n<\omega}\sqsubset_{n}$ where $\seq{{\sqsubset_{n}}}{n\in\omega}$ is some increasing sequence of closed subsets of $X\times Z$ such that, for any $n<\omega$ and for any $y\in Y$,
$({\sqsubset_{n}})^{y}=\set{x\in X}{x\sqsubset_{n}y }$ is closed nowhere dense.
\end{enumerate}
\end{definition}

\begin{remark}\label{Prsremark}
By~\autoref{def:Prs}~\ref{def:Prs3}, $\la X,\Mwf(X),\in\ra$ is Tukey below $\Rbf$ where $\Mwf(X)$ denotes the $\sigma$-ideal of meager subsets of $X$. Therefore, $\bfrak(\Rbf)\leq \non(\Mwf)$ and $\cov(\Mwf)\leq\dfrak(\Rbf)$.
\end{remark}

For the rest of this section, fix a Prs $\Rbf=\langle X,Y,\sqsubset\rangle$ and an infinite cardinal $\theta$. 

\begin{definition}[Judah and Shelah {\cite{JS}}, Brendle~{\cite{Br}}]\label{def:good}
A poset $\Por$ is \textit{$\theta$-$\Rbf$-good} if, for any $\Por$-name $\dot{h}$ for a member of $Y$, there is a non-empty set $H\subseteq Y$ (in the ground model) of size ${<}\theta$ such that, for any $x\in X$, if $x$ is $\Rbf$-unbounded over  $H$ then $\Vdash x\not\sqsubset \dot{h}$.

We say that $\Por$ is \textit{$\Rbf$-good} if it is $\aleph_1$-$\Rbf$-good.
\end{definition}

The previous is a standard property associated with preserving $\bfrak(\Rbf)$ small and $\dfrak(\Rbf)$ large after forcing extensions.

\begin{remark}
Notice that $\theta<\theta_0$
implies that any $\theta$-$\Rbf$-good poset is $\theta_0$-$\Rbf$-good. Also, if $\Por \lessdot\Qor$ and $\Qor$ is $\theta$-$\Rbf$-good, then $\Por$ is $\theta$-$\Rbf$-good.
\end{remark}

\begin{lemma}[{\cite[Lemma~2.7]{CM}}]\label{smallgds}
Assume that $\theta$ is a regular cardinal. Then any poset of size ${<}\theta$
is $\theta$-$\Rbf$-good. In particular, Cohen forcing $\Cor$ is $\Rbf$-good.
\end{lemma}

We now present the instances of Prs and the corresponding good posets that we use in our applications.

\begin{example}\label{ExamPresPro}
\mbox{} 
\begin{enumerate}[label=\normalfont(\arabic*)]

    \item\label{Pres(null-cov)}  Define $\Omega_n:=\set{a\in [2^{<\omega}]^{<\aleph_0}}{\Lb(\bigcup_{s\in a}[s])\leq 2^{-n}}$ (endowed with the discrete topology) and put $\Omega:=\prod_{n<\omega}\Omega_n$ with the product topology, which is a perfect Polish space. For every $x\in \Omega$ denote 
     \[N_{x}:=\bigcap_{n<\omega}\bigcup_{m\geq n}\bigcup_{s\in x(m)}[s],\] which is clearly a Borel null set in $2^{\omega}$.
       
    Define the Prs $\Cn:=\la \Omega, 2^\omega, \sqsubset^{\rm n}\ra$ where $x\sqsubset^{\rm n} z$ iff $z\notin N_{x}$. Recall that any null set in $2^\omega$ is a subset of $N_{x}$ for some $x\in \Omega$, so $\Cn$ and $\Cbf_\Nwf^\perp$ are Tukey-Galois equivalent. Hence, $\bfrak(\Cn)=\cov(\Nwf)$ and $\dfrak(\Cn)=\non(\Nwf)$.
    
 Any $\mu$-centered poset is $\mu^+$-$\Cn$-good (\cite{Br}). In particular, $\sigma$-centered posets are $\Cn$-good.
    
    \item\label{Pres(non-mea)}
    The relational system $\Ed_b$ is Polish when $b=\seq{b(n)}{n<\omega}$ is a sequence of non-empty countable sets such that $|b(n)|\geq 2$ for infinitely many $n$.
    Consider $\Ed:=\la\baire,\baire,\neq^\infty\ra$. 
    By~\cite[Thm.~2.4.1 \& Thm.~2.4.7]{BJ} (see also~\cite[Thm.~5.3]{CMlocalc}),  $\bfrak(\Ed)=\non(\Mwf)$ and $\dfrak(\Ed)=\cov(\Mwf)$.
    
    \item\label{Pres(unb)} The relational system $\baire = \la\baire,\baire,\leq^{*}\ra$ is Polish.
     Any $\mu$-$\Fr$-linked poset (see~\autoref{Def:Fr}) is $\mu^+$-$\baire$-good (see~\autoref{mej:uf}).
    
    \item\label{Pres(uni-null)}  For each $k<\omega$, let $\id^k:\omega\to\omega$ such that $\id^k(i)=i^k$ for all $i<\omega$ and $\Hcal:=\largeset{\id^{k+1}}{k<\omega}$. Let $\Lc^*:=\la\baire, \Scal(\omega, \Hcal), \in^*\ra$ be the Polish relational system where \[\Swf(\omega, \Hcal):=\set{\varphi\colon \omega\to[\omega]^{<\aleph_0}}{\exists{h\in\Hcal}\, \forall{i<\omega}\colon|\varphi(i)|\leq h(i)},\]
     and recall that $x\in^*\varphi$ iff $\forall^\infty n\colon x(n)\in\varphi(n)$. As a consequence of~\cite[Thm.~2.3.9]{BJ} (see also~\cite[Thm.~4.2]{CMlocalc}), $\bfrak(\Lc^*)=\add(\Nwf)$ and $\dfrak(\Lc^*)=\cof(\Nwf)$.

     Any $\mu$-centered poset is $\mu^+$-$\Lc^*$-good (see~\cite{Br,JS}) so, in particular, $\sigma$-centered posets are $\Lc^*$-good. Besides,  Kamburelis~\cite{Ka} showed that any Boolean algebra with a strictly positive finitely additive measure is $\Lc^*$-good (in particular, any subalgebra of random forcing).

     \item For $b\in\omega^\omega$, $\Rbf_b$ is a Polish relational system when $b\geq^*2$ (cf.~\autoref{rem:Rb-nonM}).

     \item\label{(M-cohen)} Let $\Mbf := \la 2^\omega,\Ior\times 2^\omega, \sqsubm \ra$ where
     \[x \sqsubm (I,y) \text{ iff }\forall^\infty n\colon x\frestr I_n \neq y\frestr I_n.\]
     This is a Polish relational system and $\Mbf\eqT \Cbf_\Mwf$ (by \autoref{ChM}).

     Note that, whenever $M$ is a transitive model of $\thzfc$, $c\in 2^\omega$ is a Cohen real over $M$ iff $c$ is $\Mbf$-unbounded over $M$.

     \item\label{PrsSN} In~\cite[Sec.~5]{BCM2}, we present a Polish relation system $\Rbf^f_\Gwf$, parametrized by a countable set $\{f\}\cup\Gwf$ of increasing functions in $\omega^\omega$, which is useful to control $\add(\SNwf)$ and $\cof(\SNwf)$ in FS iterations (see \autoref{mainpresaddSN}). We do not need to review the definition of this relational system, but it is enough to indicate that any (poset forcing equivalent to a) Boolean algebra with a striclty positive finitely additive measure, and any $\sigma$-centered poset, are $\Rbf^f_\Gwf$-good (\cite[Thm.~5.8 \& Cor.~5.9]{BCM2}, cf.~\ref{Pres(uni-null)}).
\end{enumerate}
\end{example}

We now turn to FS (finite support) iterations. To fix some notation, for two posets $\Por$ and $\Qor$, we write $\Por\subsetdot\Qor$ when $\Por$ is a complete suborder of $\Qor$, i.e.\ the inclusion map from $\Por$ into $\Qor$ is a complete embedding.

\begin{definition}[Direct limit]\label{def:limdir}
We say that $\la\Por_i:\, i\in S\ra$ is a \emph{directed system of posets} if $S$ is a directed preorder and, for any $j\in S$, $\Por_j$ is a poset and $\Por_i\subsetdot\Por_j$ for all $i\leq_S j$.

For such a system, we define its \emph{direct limit} $\limdir_{i\in S}\Por_i:=\bigcup_{i\in S}\Por_i$ ordered by
\[q\leq p \sii \exists\, i\in S\colon p,q\in\Por_i\text{ and }q\leq_{\Por_i} p.\]
\end{definition}

Good posets are preserved along FS iterations as follows.

\begin{theorem}[{\cite[Sec.~4]{BCM2}}]\label{Comgood}
Let $\seq{ \Por_\xi,\Qnm_\xi}{\xi<\pi}$ be a FS iteration such that, for $\xi<\pi$, $\Por_\xi$ forces that $\Qnm_\xi$ is a non-trivial $\theta$-cc $\theta$-$\Rbf$-good poset. 
Let $\set{\gamma_\alpha}{\alpha<\delta}$ be an increasing enumeration of $0$ and all limit ordinals smaller than $\pi$ (note that $\gamma_\alpha=\omega\alpha$), and for $\alpha<\delta$ let $\dot c_\alpha$ be a $\Por_{\gamma_{\alpha+1}}$-name of a Cohen real in $X$ over $V_{\gamma_\alpha}$. 

Then $\Por_\pi$ is $\theta$-$\Rbf$-good. Moreover, 
if $\pi\geq\theta$ then $\Cbf_{[\pi]^{<\theta}}\leqT\Rbf$, $\bfrak(\Rbf)\leq\theta$ and $|\pi|\leq\dfrak(\Rbf)$.
\end{theorem}

We even have nice theorems for $\SNwf$.

\begin{theorem}[{\cite[Thm.~5.10]{BCM2}}]\label{mainpresaddSN}
Let $\theta_0\leq \theta$ be uncountable regular cardinals, $\lambda=\lambda^{<\theta_0}$ a cardinal and let $\pi=\lambda\delta$ (ordinal product) for some ordinal $0<\delta<\lambda^+$. Assume $\theta \leq \lambda$ and $\cf(\pi)\geq \theta_0$. If $\Por$ is a FS iteration of length $\pi$ of non-trivial $\theta_0$-cc $\theta$-$\Rbf^f_\Gwf$-good posets of size ${\leq}\lambda$, 
then $\Por$ forces $\Cbf_{[\lambda]^{<\theta}}\leqT\SNwf$, in particular, 
$\add(\SNwf)\leq\theta$ and $\lambda\leq\cof(\SNwf)$.
\end{theorem}

We now present two preservation results for the covering of $\SNwf$, originally introduced by Pawlikowski~\cite{P90} and generalized and improved in~\cite{CarMej23}. Here, we use the notion of the \emph{segment cofinality} of an ordinal $\pi$:
\[\scf(\pi):=\min\set{|c|}{c\subseteq \pi \text{ is a non-empty final segment of }\pi}.\]

\begin{theorem}[{\cite{P90},~\cite[Thm.~5.4~(c)]{CarMej23}}]\label{thm:cfcovSN}
    Let $\seq{ \Por_\xi}{\xi\leq\pi}$ be a $\subsetdot$-increasing sequence of posets such that $\Por_\pi = \limdir_{\xi<\pi}\Por_\xi$. Assume that $\cf(\pi)>\omega$, $\Por_\pi$ has the $\cf(\pi)$-cc and $\Por_{\xi+1}$ adds a Cohen real over the $\Por_\xi$-generic extension for all $\xi<\pi$. Then $\pi\leqT \Cbf_{\SNwf}^\perp$, in particular $\cov(\SNwf) \leq \cf(\pi) \leq \non(\SNwf)$.
\end{theorem}

\begin{theorem}[{\cite{P90},~\cite[Cor.~5.9]{CarMej23}}]\label{thm:precaliber}
Assume that $\theta\geq\aleph_1$ is regular. Let
$\Por_\pi=\seq{\Por_\xi,\Qnm_\xi}{ \xi<\pi}$ be a FS iteration of non-trivial precaliber $\theta$ posets such that $\cf(\pi)>\omega$ and $\Por_\pi$ has $\cf(\pi)$-cc,
and let $\lambda:=\scf(\pi)$. Then $\Por_\pi$ forces $\Cbf_{[\lambda]^{<\theta}}\leqT\Cbf_{\SNwf}^{\perp}$. In particular, whenever $\scf(\pi)\geq\theta$, $\Por_\pi$ forces $\cov(\SNwf)\leq\theta$ and $\scf(\pi)\leq\non(\SNwf)$.
\end{theorem}

To force a lower bound of $\bfrak(\Rbf)$, we use:

\begin{theorem}[{\cite[Thm.~2.12]{CM22}}]\label{itsmallsets}
Let $\Rbf=\la X,Y,\sqsubset\ra$ be a Polish relational system, $\theta$ an uncountable regular cardinal, and let  $\Por_\pi=\seq{\Por_\xi,\Qnm_\xi}{\xi<\pi}$ be a FS iteration of $\theta$-cc posets with $\cf(\pi)\geq\theta$. Assume that, for all $\xi<\pi$ and any $A\in[X]^{<\theta}\cap V_\xi$, there is some $\eta\geq\xi$ such that $\Qnm_\eta$ adds an $\Rbf$-dominating real over $A$. Then $\Por_\pi$ forces $\theta\leq\bfrak(\Rbf)$, i.e.\ $\Rbf\leqT\Cbf_{[X]^{<\theta}}$.
\end{theorem}

\begin{lemma}[{\cite[Lemma~4.5]{CM}}]\label{lem:strongCohen}
Assume that $\theta$ has uncountable cofinality. Let $\seq{\Por_{\alpha}}{\alpha<\theta}$ be a $\subsetdot$-increasing sequence of $\cf(\theta)$-cc posets such that  $\Por_\theta=\limdir_{\alpha<\theta}\Por_{\alpha}$. If $\Por_{\alpha+1}$ adds a Cohen real $\dot{c}_\alpha\in X$ over $V^{\Por_\alpha}$ for any $\alpha<\theta$, then $\Por_{\theta}$ forces that $\set{\dot{c}_\alpha}{\alpha<\theta}$ is a strongly $\theta$-$\Rbf$-unbounded family, i.e.\ $\theta\leqT \Rbf$.
\end{lemma}

From now on, we restrict our attention to the notion of \emph{ultrafilter-limits} introduced in \cite{GMS}, and to the notion of \emph{filter-linkedness} by the second author~\cite{mejiavert}. More about the latter can be found in~\cite[Section 3]{BCM}.

 Given a poset $\Por$, the $\Por$-name $\dot{G}$ usually denotes the canonical name of the $\Por$-generic set. If $\bar{p}=\seq {p_n}{n<\omega}$ is a sequence in $\Por$, denote by $\dot{W}_\Por(\bar{p})$ the $\Por$-name of $\set{n<\omega}{p_n\in\dot{G}}$. When the forcing is understood from the context, we just write $\dot{W}(\bar{p})$.

\begin{definition}\label{Def:GMS}
Let $\Por$ be a poset, $D\subseteq \pts(\omega)$ a non-principal ultrafilter, and $\mu$ an infinite cardinal.
\begin{enumerate}[label=\rm(\arabic*)]
\item\label{GMS1} A set $Q\subseteq \Por$ has \emph{$D$-limits} if there is a function $\lim^{D}\colon Q^\omega\to \Por$ and a $\Por$-name $\dot D'$ of an ultrafilter extending $D$ such that, for any $\bar q = \seq{q_i}{i<\omega}\in Q^\omega$,
\[{\lim}^{D}\, \bar q \Vdash \dot{W}(\bar{q})\in \dot D'.\]

\item A set $Q\subseteq \Por$ has \emph{uf-limits} if it has $D$-limits for any ultrafilter $D$.

\item The poset $\Por$ is \emph{$\mu$-$D$-$\lim$-linked} if $\Por = \bigcup_{\alpha<\mu}Q_\alpha$ where each $Q_\alpha$ has $D$-limits. We say that $\Por$ is \emph{uniformly $\mu$-$D$-$\lim$-linked} if, additionally, the $\Por$-name $\dot D'$ from~\ref{GMS1} only depends on $D$ (and not on $Q_\alpha$, although we have different limits for each $Q_\alpha$).

\item The poset $\Por$ is \emph{$\mu$-uf-$\lim$-linked} if $\Por = \bigcup_{\alpha<\mu}Q_\alpha$ where each $Q_\alpha$ has uf-limits. We say that $\Por$ is \emph{uniformly $\mu$-uf-$\lim$-linked} if, additionally, for any ultrafilter $D$ on $\omega$, the $\Por$-name $\dot D'$ from~\ref{GMS1} only depends on $D$.
\end{enumerate}    
\end{definition}

For not adding dominating reals, we have the following weaker notion.

\begin{definition}[{\cite{mejiavert}}]\label{Def:Fr}
    Let $\Por$ be a poset and $F$ a filter on $\omega$. A set $Q\subseteq \Por$ is \emph{$F$-linked} if, for any $\bar p=\seq{p_n}{n<\omega} \in Q^\omega$, there is some $q\in \Por$ forcing that $F\cup \{\dot{W}(\bar p)\}$ generates a filter on $\omega$.
    We say that $Q$ is \emph{uf-linked (ultrafilter-linked)} if it is $F$-linked for any filter $F$ on $\omega$ containing the \emph{Frechet filter} $\Fr:=\set{\omega\menos a}{a\in[\omega]^{<\aleph_0}}$.
    
    For an infinite cardinal $\mu$, $\Por$ is \emph{$\mu$-$F$-linked} if $\Por = \bigcup_{\alpha<\mu}Q_\alpha$ for some $F$-linked $Q_\alpha$ ($\alpha<\mu$). When these $Q_\alpha$ are uf-linked, we say that $\Por$ is \emph{$\mu$-uf-linked}.
\end{definition}

For instance, random forcing is $\sigma$-uf-linked~\cite{mejiavert}, but it may not be $\sigma$-uf-$\lim$-linked (cf.~\cite[Rem.~3.10]{BCM}). It is clear that any uf-$\lim$-linked set $Q\subseteq \Por$ is uf-linked, which implies $\Fr$-linked.

\begin{theorem}[{\cite{mejiavert}}]\label{mej:uf}
Any $\mu$-$\Fr$-linked poset is $\mu^+$-$\omega^\omega$-good.
\end{theorem}

\begin{example}\label{exm:ufl}
The following are the instances of $\mu$-uf-lim-linked posets that we use in our applications.
\begin{enumerate}[label=\rm (\arabic*)]

    \item Any poset of size $\mu$ is uniformly $\mu$-uf-lim-linked (because singletons are uf-lim-linked). In particular, Cohen
forcing is uniformly $\sigma$-uf-lim-linked.
    
    \item \cite{GMS, BCM} The standard eventually different real forcing is uniformly $\sigma$-uf-lim-linked.
\end{enumerate}    
\end{example}

We now introduce a forcing notion $\Por_b$ that increases $\bfrak(\Rbf_b)$ (see~\autoref{RelsAM}) and prove that $\Por_b$ is uniformly $\sigma$-uf-$\lim$-linked. 

\begin{definition}\label{def_Pb}
Given $b\in\baire$, the poset $\Por_b$ is defined as follows:
A condition $p=(s,t,F)\in\Por_b$ if it fulfills the following:
\begin{itemize}
    \item $s\in\omega^{<\omega}$ is increasing with $s(0)>0$ (when $|s|>0$), 
    \item $t\in\Seq_{<\omega}(b):=\bigcup_{n<\omega}\prod_{i<n}b(i)$, and 
    \item $F\in[\prod b]^{<\aleph_0}$.
\end{itemize}
We order $\Por_b$ by setting $(s',t',F')\leq (s,t,F)$ iff $s\subseteq s'$, $t\subseteq t'$,
$F\subseteq F'$ and,
\[\forall f\in F\,\forall n \in|s'|\smallsetminus|s|\, \exists k\in[s'(n-1),s'(n))\colon f(k)=t'(k). \text{ (Here $s'(-1):=0$.)}\]

The poset $\Por_b$ is $\sigma$-centered, since for $s\in\omega^{<\omega}$ increasing, and for $t\in\Seq_{<\omega}(b)$, the set 
\[P_{s,t}:=\set{(s',t',F)\in\Por_b}{s'=s\textrm{\ and\ }  t'=t}\]
is centered and $\bigcup_{s\in\omega^{<\omega},\, t\in\Seq_{<\omega}(b)}P_{s,t}=\Por_b$.

Let $G$ be a $\Por_b$-generic filter over $V$. In $V[G]$, define 
\[r_{\gen}:=\bigcup\set{s}{\exists t,F\colon (s,t,F) \in G} \text{ and } 
h_{\gen}:=\bigcup\set{t}{\exists s,F\colon (s,t,F) \in G}.\]
Then $(r_{\gen},h_{\gen})\in\baire\times\prod b$ and, for every $f\in \prod b\cap V$, and for all but finitely many $n\in\omega$ there is some $k\in[r_{\gen}(n),r_{\gen}(n+1)]$ such that $f(k)=h_{\gen}(k)$. We can identify the generic real with $(J_{\gen},h_{\gen})\in\Ior\times\prod b$ where $J_{\gen,n} :=[r_\gen(n-1),r_\gen(n))$, which satisfies that, for every $f\in \prod b\cap V$, $f  \sqrb (J_{\gen},h_{\gen})$.

We will show that the sets 
\[P_{s,t,m}:=P_b(s,t,m)=\set{(s',t',F)\in\Por_b}{s'=s,\  t'=t \text{ and }|F|\leq m}\]
for $s\in\omega^{<\omega}$, $t\in\Seq_{<\omega}(b)$ and $m<\omega$, witness that $\Por_b$ is uniformly $\sigma$-uf-$\lim$-linked. 
For an ultrafilter $D$ on $\omega$,  and $\bar{p}=\seq{p_n}{n\in\omega}\in P_{s,t,m}$, we show how to define $\lim^D\bar p$. Let $p_n=(s,t,F_n) \in P_{s,t,m}$. Considering the lexicographic order $\lhd$ of $\prod b$, and let $\set{x_{n,k}}{k<m_n}$ be a $\lhd$-increasing enumeration of $F_n$ where $m_n \leq m$. Next find an unique  $m_*\leq m$ such that $A:=\set{n\in\omega}{m_n=m_*}\in D$. For each $k<m_*$, define $x_k:=\lim_n^D x_{n,k}$ in $\prod b$ where $x_k(i)$ is the unique member of $b(i)$ such that $\set{n\in A}{x_{n,k}(i) = x_k(i)} \in D$ (this coincides with the topological $D$-limit). Therefore, we can think of $F:=\set{x_k}{k<m_*}$ as the $D$-limit of $\seq{F_n}{n<\omega}$, so we define $\lim^D \bar p:=(s,t,F)$. Note that $\lim^D \bar p \in P_{s,t,m}$. 
\end{definition}


\begin{theorem}\label{uf:Pb}
The poset $\Por_b$ is uniformly $\sigma$-uf-$\lim$-linked:
For any ultrafilter $D$ on $\omega$, there is a $\Por_b$-name of an ultrafilter $\dot D'$ on $\omega$ extending $D$ such that,
for any $s\in \omega^{<\omega}$, $t\in \Seq_{<\omega}(b)$ , $m<\omega$ and $\bar p\in P_{s,t,m}^{\omega}$, $\lim^D \bar p \Vdash \dot W(\bar p)\in \dot D'$. 
\end{theorem}

To prove the former theorem, it suffices to show the following:

\begin{clm}
      Assume $M<\omega$, $\set{(s_k,t_k,m_k)}{k<M}\subseteq\omega^{<\omega}\times\Seq_{<\omega}(b)\times \omega$, $\set{\bar{p}^k}{k<M}$ such that each $\bar{p}^k=\seq{ p_{k,n}}{n<\omega}$ is a sequence in $P_{s_k,t_k,m_k}$, $q_k$ is the $D$-limit of $\bar{p}^k$ for each $k<M$, and $q\in\Por_b$ is stronger than every $q_k$. Then, for any $a\in D$, there are $n\in a$ and $q'\leq q$ stronger than $p_{k,n}$ for all $k<M$ (i.e.\ $q'$ forces $a\cap\bigcap_{k<M}\dot{W}(\bar{p}^k)\neq\emptyset$).
\end{clm}
\begin{proof}
Write $p_{k,n}=(s_k,t_k,F_{k,n})$, $q_k=(s_k,t_k,F_{k})$ where each $F_k = \set{x^k_j}{j<m_{*,k}}$ is the $D$-limit of $F_{k,n} = \set{x^{k,n}_j}{j<m_{*,k}}$ (increasing $\lhd$-enumeration) with $m_{*,k}\leq m_k$. Assume that $q = (s,t,F)\leq q_k$ in $\Por_b$ for all $k<M$. Let 
\begin{multline*}
 U_k:=\big\{\seq{x_j}{j<m_{*,k}} \st\, \\ 
 \forall j<m_{*,k}\, \forall \ell\in|s|\smallsetminus|s_k|\, \exists m\in[s(\ell-1),s(\ell))\colon x_j(m)=t(m)\big\},  
\end{multline*}
which is an open neighborhood of $\seq{x^k_j}{j<m_{*,k}}$ in $(\prod b)^{m_{*,k}}$. 
Then 
\[b_k := \big\{n<\omega \st\,  \forall j<m_{*,k}\, \forall \ell\in|s|\smallsetminus|s_k|\, \exists m\in[s(\ell-1),s(\ell))\colon x_{j}^{k,n}(m)=t(m) \big\}\in D.\]
Hence, $a\cap\bigcap_{k<M}b_k\neq\emptyset$, so choose $n\in a\cap\bigcap_{k<M}b_k$ and put $q'=(s,t,F')$ where $F':=F\cup \bigcup_{k<M}F_{k,n}$. This is a condition in $\Por_b$ because $|F'|\leq |F|+\sum_{k<M}m_{*,k}$. Furthermore, $q'$ is stronger than $q$ and $p_{n,k}$ for any $k<M$. 
\end{proof}

Now we define a forcing to increase the additivity of the strong measure zero ideal. This is a weakening of a forcing of Yorioka~\cite{Yorioka}.

\begin{definition}\label{foraddSN}
Let $f$ be an increasing function in $\baire$.  
Define $\Qor_{f}$ as the poset whose conditions are triples $(\sigma, N, F)$ such that $\sigma\in (2^{<\omega})^{<\omega}$, $N<\omega$ and $F\subseteq (2^{<\omega})^\omega$, satisfying the following requirements:
\begin{itemize}
    \item $|\sigma(i)| = f(i)$ for all $i<|\sigma|$,

    \item $|F|\leq N$ and $|\sigma|\leq N^2$, and 
    
    \item  $\forall\tau\in F\,  \forall n<\omega\colon |\tau(n)| = f((n+1)^2)$.
\end{itemize}
We order $\Qor_{f}$ by $(\sigma', N', F')\leq(\sigma, N, F)$ iff $\sigma\subseteq\sigma'$, $N\leq N'$, $F\subseteq F'$ and \[\forall\tau\in F\, \forall i\in N'\smallsetminus N\, \exists n<|\sigma'|\colon \sigma'(n)\subseteq\tau(i).\] 
\end{definition}

\begin{lemma}\label{denseQf}
Let $f\in\omega^\omega$ be increasing.
\begin{enumerate}[label = \normalfont (\arabic*)]
    \item\label{dense1} For $n<\omega$, the set $\set{(\sigma, N, F)\in \Qor_{f}}{n<N}$ is dense. Even more, if $(\sigma,N,F)\in \Qor_{f}$ and $N'\geq N$ in $\omega$, then there is some $\sigma'$ such that $(\sigma',N',F)\leq (\sigma,N,F)$ in $\Qor_{f}$.

    \item For $\tau\in (2^{<\omega})^\omega$, if $\forall i<\omega\colon |\tau(i)| = f((i+1)^2)$, then the set $\set{(\sigma, N, F)\in \Qor_{f}}{\tau\in F}$ is dense.

    \item For $n<\omega$, the set $\set{(\sigma, N, F)\in \Qor_{f}}{n<|\sigma|}$ is dense.
\end{enumerate}
\end{lemma}
\begin{proof}
    We show~\ref{dense1} (the other properties follow immediately from this). Let $(\sigma,N,F)\in\Qor_{f}$ and $N'\geq N$ in $\omega$. We need to extend $\sigma$ to $\sigma'$ to ensure that, for any $i\in N'\menos N$ and $\tau\in F$, $\tau(i)$ extends some $\sigma'(n)$. For this purpose, we aim for $|\sigma'| = |\sigma| + |F|(N'-N)$. 
    Enumerate $F=\set{\tau_i}{i<|F|}$. 
    For $d<N'-N$ and $i<|F|$, we have
    \[f(|\sigma|+d |F|+i) \leq f(N^2 + d N + i)<f(N(N+d+1)) < f((N+d+1)^2)= |\tau_i(N+d)|,\]
    so it is enough to define $\sigma'(|\sigma|+d |F|+i):= \tau_i(N+d)\frestr f(|\sigma|+d |F|+i)$.
\end{proof}

The poset $\Qor_{f}$ is ccc, even $\sigma$-$k$-linked for any $k<\omega$, since the set
\[Q^k_{\sigma,N}:=\set{(\sigma',N',F)\in\Qor_{f}}{\sigma'=\sigma,\ N'=N \text{ and }k|F|\leq N}\]
is $k$-linked and $\bigcup_{\sigma\in (2^{<\omega})^\omega, N<\omega} Q^k_{\sigma,N}$ is dense in $\Qor_{f}$ by \autoref{denseQf}~\ref{dense1}.

Let $G$ be a $\Qor_{f}$-generic filter over $V$. In $V[G]$, define \[\sigma_{\gen}:=\bigcup\set{\sigma}{\exists (N, F)\colon (\sigma, N, F) \in G}.\] Then $\sigma_{\gen}\in(2^{<\omega})^\omega$, $\hgt_{\sigma_\gen} = f$ and, for every $\tau\in (2^{<\omega})^\omega\cap V$, if $|\tau(i)|\geq f((i+1)^2)$ for all but finitely many $i<\omega$, then $[\tau]_\infty\subseteq\bigcup_{n<\omega}[\dot\sigma_\gen(n)]$.

We aim to show that $\Qor_{f}$ is uniformly $\sigma$-uf-$\lim$-linked, witnessed by
\[Q_{\sigma, N}:=Q_{f}(\sigma, N)=\set{(\tau, N', F)\in \Qor_{f}}{\tau=\sigma,\ N=N'}.\]
for $\sigma\in(2^{<\omega})^{<\omega}$ and $N<\omega$.\footnote{This set may be empty for some $\sigma$, but this does not hurt the arguments.} 
Let $D$ be an ultrafilter on $\omega$, and $\bar{p}=\seq{p_n}{n\in\omega}$ be a 
sequence in $Q_{\sigma, N}$ with $p_n = (\sigma,N,F_n)$. Since $|F_{n}|\leq N$, we can find $a_0\in D$ and $N_0<\omega$ such that $F_{n}=\set{\tau_{n,k}}{k\in N_0}$ (increasing enumeration using the lexicographic order of $(2^{<\omega})^\omega$ with respect to some canonical well-order of $2^{<\omega}$) for all $n\in a_0$. For each $k<N_0$, define $\tau_k=\lim_n^D \tau_{n,k}$ by
\[\tau_k(i)=s\textrm{\ iff\ }\set{n\in a_0}{\tau_{n,k}(i)=s}\in D,\]
which matches the topological $D$-limit in $\prod_{i<\omega}2^{f(i+1)^2}$. 
Then, the $D$-limit of $F_n$ can be defined as $F:=\set{\tau_k}{k<N_0}$ and $\lim^D \bar p := (\sigma,N,F)$. It is clear that this limit is in $Q_{\sigma,N}$.


\begin{theorem}\label{ufQf}
The poset $\Qor_f$ is uniformly $\sigma$-uf-$\lim$-linked: 
If $D$ is an ultrafilter on $\omega$, 
then there is a  $\Qor_{f}$-name of an ultrafilter $\dot D'$ on $\omega$ extending $D$ such that, for any $\sigma \in (2^{<\omega})^\omega$, $N<\omega$ and $\bar p\in Q_{\sigma,N}^\omega$, $\lim^D \bar p \Vdash W(\bar p)\in \dot D'$.
\end{theorem}

Just as in~\autoref{uf:Pb}, to prove the foregoing theorem, it suffices to see the following:

\begin{clm}
      Assume $M<\omega$, $\set{(\sigma_k,N_k)}{k<M}\subseteq(2^{<\omega})^{<\omega}\times\omega$, $\set{\bar{p}^k}{k<M}$ such that each $\bar{p}^k=\seq{ p_{k,n}}{n<\omega}$ is a sequence in $Q_{\sigma_k,N_k}$, $q_k$ is the $D$-limit of $\bar{p}^k$ for each $k<M$, and $q\in\Qor_{f}$ is stronger than every $q_k$. If $a\in D$ then there are some $n\in a$ and $q'\leq q$ stronger than $p_{k,n}$ for all $k<M$.
\end{clm}
\begin{proof}
Write the forcing conditions as $p_{k,n}=(\sigma_k,N_k, F_{k,n})$ where $F_{k,n}=\set{\tau_i^{k,n}}{i<N_0^k}$ (increasing enumeration) with $N_0^k\leq N_k$, for all $n\in a$ (wlog), and let $q_k=(\sigma_k,N_k,F_{k})$ be such that each $F_k=\set{\tau^k_i}{i<N_0^k}$ is the $D$-limit of $\seq{ F_{k,n}}{n<\omega}$,
that is, $\tau_i^k=\lim_n^D \tau_i^{k,n}$ for $i< N_0^k$. 

Assume that $q=(\sigma,N,F) \leq q_k$ for all $k<M$. 
By strengthening $q$ if necessary, we assume that $|F|+\sum_{k<M}N_k\leq N$.
Then 
\[\forall i<N_0^k\, \forall j\in N\smallsetminus N_k\,\exists \ell<|\sigma|\colon\sigma(\ell)\subseteq\tau_i^k(j),\] 
so $b_k:=\set{n<\omega}{\forall i<N_0^k\,\forall j\in N\smallsetminus N_k\,\exists \ell<|\sigma|:\sigma(\ell)\subseteq\tau_i^{k,n}(j)}\in D$. Hence $a\cap\bigcap_{k<M}b_k$ is not empty. Choose an $n$ in that set and put $q':=(\sigma,N,F')$ where $F':=F\cup \bigcup_{k<M} F_{n,k}$. This is a condition in $\Qor_{f}$ because $|F'|\leq |F|+\sum_{k<M}N_k \leq N$. Thus $q'$ is stronger than $q$ and $p_{n,k}$ for $k<M$. 
   \end{proof}

\section{Consistency results}\label{Secmain}

In this section, we prove our main consistency results about the cardinal characteristics associated with $\NAwf$ and $\MAwf$. Concretely, we prove~\autoref{ThmCovma}--\ref{Thm:faddSN}. 

We start with a review of some other posets we will use in the proof of our consistency results.

\begin{definition}\label{defposet} Define the following forcing notions 
\begin{enumerate}[label=\normalfont(\arabic*)]
    \item \emph{Localization forcing} is the poset defined by $\Loc:=\set{(n,\varphi)\in\omega\times \Swf(\omega,\id_\omega)}{\exists m<\omega\,\forall i<\omega\colon |\varphi(i)|\leq m}$ ordered by $(n',\varphi')\leq(n,\varphi)$ iff $n\leq n'$, $\varphi'\frestr n = \varphi\frestr n$ and $\varphi(i)\subseteq\varphi'(i)$ for every $i<\omega$. This forcing is used to increase $\add(\Nwf)$. Recall that $\Loc$ is $\sigma$-linked, hence ccc.
    
    \item\label{defposet2} \emph{Hechler forcing} is defined by  $\Dor=\omega^{<\omega}\times\baire$,
ordered by $(t,g)\leq(s,f)$ if $s\subseteq t$, $f\leq g$ and $f(i)\leq t(i)$ for all $i\in |t|\menos|s|$. This forcing is used to increase $\bfrak$. Recall that $\Dor$ is $\sigma$-centered.

  \item For an infinite cardinal  $\theta$, $\Fn_{<\theta}(A,B)$ denotes the poset of partial functions from $A$ into $B$ of size ${<}\theta$, ordered by $\supseteq$.

  \item $\Cor_\lambda:=\Fn_{<\aleph_0}(\lambda\times\omega,2)$ is the poset adding $\lambda$-many Cohen reals, and denote random forcing by $\Bor$.
\end{enumerate}
\end{definition}

We now begin proving our main consistent results. In particular, we establish~\autoref{ThmCovma}:

\begin{theorem}\label{Thm:covMA}
 Let $\theta<\nu\leq\lambda$ be uncountable cardinals such that $\theta^{<\theta}=\theta$, $\nu^{\theta} = \nu$ and $\lambda^{\aleph_0} = \lambda$. Then there is a poset, preserving cofinalities, forcing 
\[
     \cov(\Nwf) = \aleph_1  \leq \add(\Mwf) = \cof(\Mwf) = \theta 
      \leq \cov(\MAwf)\leq \nu \leq  \non(\Nwf) =\cfrak =\lambda.
\]
In particular, it is consistent with $\thzfc$ that $\cov(\MAwf) < \non(\Nwf)$.
\end{theorem}
\begin{proof}
First force with $\Fn_{<\theta}(\nu,\theta)$ to obtain $\dfrak_\theta = 2^\theta = \nu$ in its generic extension, where $\dfrak_\theta := \dfrak(\la \theta,\theta,\leq\ra^\theta)$ (as a relational system), which coincides with the canonical dominating number of $\theta^\theta$. Notice that cardinalities (and cofinalities) are preserved, as well as the cardinal arithmetic hypothesis.

Aftwerwards, perform a FS iteration $\Por = \seq{\Por_\xi,\Qnm_\xi}{\xi<\lambda\theta}$ where $\Qor_\xi$ is a $\Por_\xi$-name of $\Dor\ast \Por_{\dot d_\xi}$ where $\dot d_\xi$ is the name of the dominating real over $V_\xi := V^{\Por_\xi}$ added by $\Dor$. The iterands of this iteration are $\Cn$-good (see \autoref{ExamPresPro}~\ref{Pres(null-cov)}), so $\Por$ forces $\Cbf_{[\lambda]^{<\aleph_1}} \leqT \Cn$ by \autoref{Comgood}. On the other hand, $\Por$ forces $\cfrak=\lambda$, so it follows that $\cov(\Nwf) = \aleph_1$ and $\non(\Nwf) = \cfrak$. On the other hand, by the cofinaly-many Cohen and dominating reals $\seq{\dot d_{\lambda\rho}}{\rho<\theta}$ added along the iteration, since $\cf(\lambda\theta)=\theta$, we obtain $\bfrak = \non(\Mwf) = \cov(\Mwf) = \dfrak = \theta$. which implies $\add(\Mwf) = \cof(\Mwf)=\theta$. Even more, we obtain $\omega^\omega \eqT \Cbf_\Mwf \eqT \theta$.

In the final generic extension $V_{\lambda\theta}$, it is clear that $D:=\set{d_{\lambda\rho}}{\rho<\theta}$ is $\leq^*$-increasing and dominating in $\baire$. Denote $d'_\rho := d_{\lambda\rho}$ for $\rho<\theta$. We show that $\Rbf_{d'_\rho} \eqT \theta$. On the one hand, $\theta \eqT \Cbf_\Mwf \leqT \Rbf_{d'_\rho}$. For the converse, define $F\colon \prod d'_\rho \to \theta$ such that, for $x\in \prod d'_\rho$, $F(x)$ is some ordinal $\eta>\rho$ such that $x\in V_{\lambda\eta}$; and define $F'\colon \theta\to \Ior\times\baire$ such that $F'(\varrho)$ is the $\Por_{d'_\varrho}$-generic real added by $\Qor_{\lambda\varrho}$ when $\varrho\geq\rho$, otherwise $F'(\varrho):= F'(\rho)$. It is clear that $(F,F')$ is the desired Tukey connection. 

Since $\Por$ is ccc, the equality $\dfrak_\theta = \nu$ is preserved (see e.g.~\cite[Lem.~6.6]{CarMej23}). Now, by \autoref{prodRb},
\[\Cbf_\MAwf \leqT \prod_{b\in D}\Rbf_b \eqT \la\theta,\leq \ra^\theta,\]
so $\cov(\MAwf) \leq \dfrak_\theta = \nu$.
\end{proof}

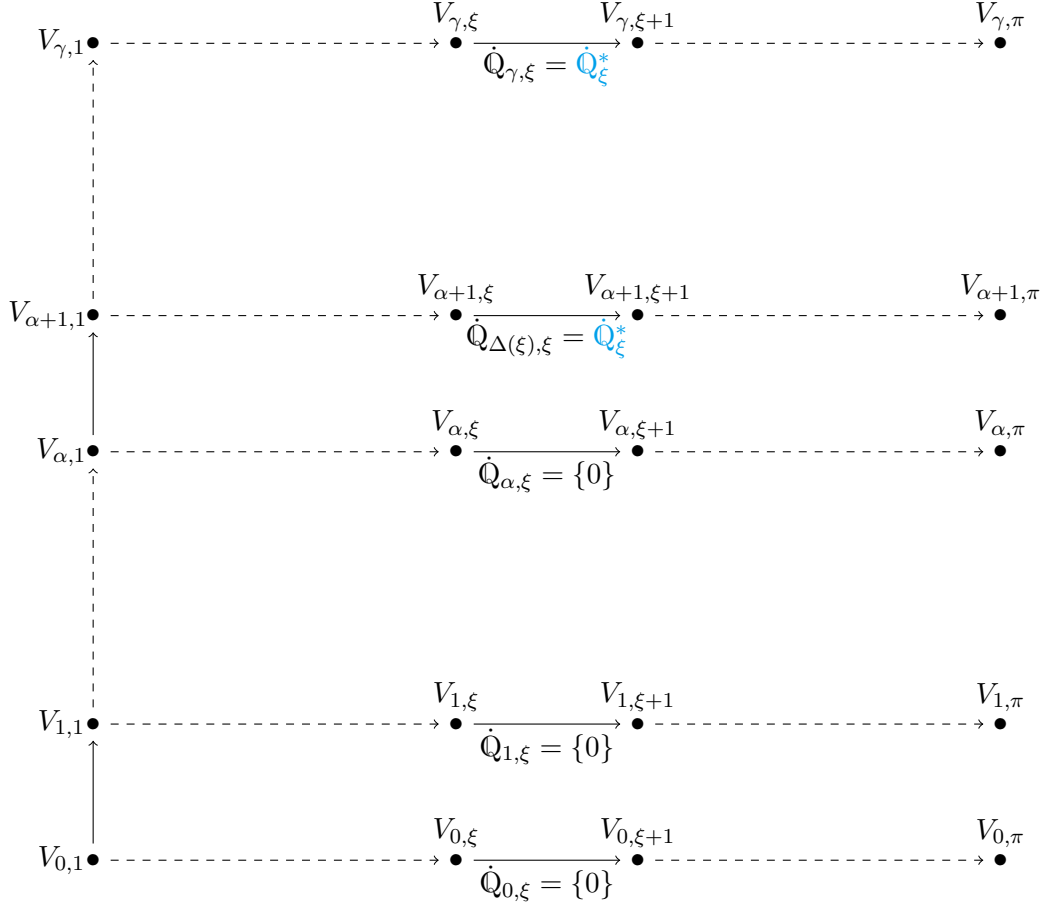
\begin{figure}[ht]
\centering
\begin{tikzpicture}[scale=1.2]
\small{

\node (f00) at (0,0){$\bullet$};
\node (f01) at (0,1.5){$\bullet$};
\node (f02) at (0,4.5) {$\bullet$} ;
\node (f03) at (0,6) {$\bullet$} ;
\node (f04) at (0,9) {$\bullet$} ;
\node (fxi0) at (4,0){$\bullet$};
\node (fxi1) at (4,1.5){$\bullet$};
\node (fxi2) at (4,4.5) {$\bullet$};
\node (fxi3) at (4,6) {$\bullet$} ;
\node (fxi4) at (4,9) {$\bullet$} ;
\node (fxi+10) at (6,0){$\bullet$};
\node (fxi+11) at (6,1.5){$\bullet$};
\node (fxi+12) at (6,4.5) {$\bullet$};
\node (fxi+13) at (6,6) {$\bullet$} ;
\node (fxi+14) at (6,9) {$\bullet$} ;
\node (fpi0) at (10,0){$\bullet$};
\node (fpi1) at (10,1.5){$\bullet$};
\node (fpi2) at (10,4.5) {$\bullet$};
\node (fpi3) at (10,6) {$\bullet$} ;
\node (fpi4) at (10,9) {$\bullet$} ;

\draw   (f00) edge [->] (f01);
\draw[dashed]    (f01) edge [->] (f02);
\draw       (f02) edge [->] (f03);
\draw[dashed]       (f03) edge [->] (f04);
        
\draw[dashed]   (f00) edge [->] (fxi0);
\draw      (fxi0) edge [->] (fxi+10);
\draw[dashed]       (fxi+10) edge [->] (fpi0);
 
\draw[dashed]   (f01) edge [->] (fxi1);
\draw        (fxi1) edge [->] (fxi+11);
\draw[dashed]        (fxi+11) edge [->] (fpi1); 
        
\draw[dashed]  (f02) edge [->] (fxi2);
\draw        (fxi2) edge [->] (fxi+12);
\draw[dashed]        (fxi+12) edge [->] (fpi2);  
        
\draw[dashed]  (f03) edge [->] (fxi3);
\draw       (fxi3) edge [->] (fxi+13);
\draw[dashed]      (fxi+13) edge [->] (fpi3);  
        
\draw[dashed]  (f04) edge [->] (fxi4);
\draw        (fxi4) edge [->] (fxi+14);
\draw[dashed]        (fxi+14) edge [->] (fpi4);         

\node at (-0.35,0) {$V_{0,1}$};
\node at (-0.35,1.5) {$V_{1,1}$};
\node at (-0.35,4.5) {$V_{\alpha,1}$};
\node at (-0.5,6) {$V_{\alpha+1,1}$};
\node at (-0.35,9) {$V_{\gamma,1}$};

\node at (4,0.3) {$V_{0,\xi}$};
\node at (6,0.3) {$V_{0,\xi+1}$};

\node at (4,1.8) {$V_{1,\xi}$};
\node at (6,1.8) {$V_{1,\xi+1}$};

\node at (4,4.8) {$V_{\alpha,\xi}$};
\node at (6,4.8) {$V_{\alpha,\xi+1}$};

\node at (4,6.3) {$V_{\alpha+1,\xi}$};
\node at (6,6.3) {$V_{\alpha+1,\xi+1}$};

\node at (4,9.3) {$V_{\gamma,\xi}$};
\node at (6,9.3) {$V_{\gamma,\xi+1}$};

\node at (10,0.3) {$V_{0,\pi}$};
\node at (10,1.8) {$V_{1,\pi}$};
\node at (10,4.8) {$V_{\alpha,\pi}$};
\node at (10,6.3) {$V_{\alpha+1,\pi}$};
\node at (10,9.3) {$V_{\gamma,\pi}$};

\node at (5,-0.25) {$\Qnm_{0,\xi}=\{0\}$};
\node at (5,1.25) {$\Qnm_{1,\xi}=\{0\}$};
\node at (5,4.25) {$\Qnm_{\alpha,\xi}=\{0\}$};
\node at (5,5.75) {$\Qnm_{\Delta(\xi),\xi}=\subiii{\Qnm^*_\xi}$};
\node at (5,8.75) {$\Qnm_{\gamma,\xi}=\subiii{\Qnm^*_\xi}$};




}
\end{tikzpicture}
\caption{A simple matrix iteration}
\label{matrixuf}
\end{figure}

We use the matrix iterations with ultrafilters method from~\cite{BCM}, which we examine below, to force many simultaneous values in Cicho\'n's diagram.

\begin{definition}[{\cite[Def.~2.10]{BCM}}]\label{Defmatsimp}
A~\emph{simple matrix iteration} of ccc posets (see~\autoref{matrixuf}) is composed of the following objects: 
\begin{enumerate}[label=\rm (\Roman*)]
    \item ordinals $\gamma$ (height) and $\pi$ (length);
    \item a function $\Delta\colon \pi\to\gamma$; 
    \item a sequence of posets $\seq{\Por_{\alpha,\xi}}{\alpha\leq \gamma,\ \xi\leq \pi}$ where $\Por_{\alpha,0}$ is the trivial poset for any $\alpha\leq \gamma$;
    \item for each $\xi<\pi$, $\dot{\Qor}^*_\xi$ is a 
     $\Por_{\Delta(\xi),\xi}$-name of a poset such that $\Por_{\gamma,\,\xi}$ forces it to be ccc;
    \item $\Por_{\alpha,\xi+1}=\Por_{\alpha,\,\xi}\ast\Qnm_{\alpha,\xi}$, where  
\[\dot{\mathbb{Q}}_{\alpha,\xi}:=
\begin{cases}
    \Qnm^*_\xi & \textrm{if  $\alpha\geq\Delta(\xi)$,}\\
    \{0\}         & \textrm{otherwise;}
\end{cases}\]
    \item for $\xi$ limit, $\Por_{\alpha,\xi}:=\limdir_{\eta<\xi}\Por_{\alpha,\eta}$. 
\end{enumerate}

It is known that $\alpha\leq\beta\leq\gamma$ and $\xi\leq\eta\leq\pi$ imply $\Por_{\alpha,\xi}\subsetdot\Por_{\beta,\eta}$, see e.g.~\cite{B1S} and \cite[Cor.~4.31]{CM}.  If $G$ is $\Por_{\gamma,\pi}$-generic
over $V$, we denote $V_{\alpha,\xi}= [G\cap\Por_{\alpha,\xi}]$ for all $\alpha\leq\gamma$ and $\xi\leq\pi$. 
\end{definition}

\begin{lemma}[{\cite[Lemma~5]{BrF}, see also~\cite[Cor.~2.6]{mejiavert}}]\label{realint}
 Assume that $\Por_{\gamma, \pi}$ is a simple matrix iteration as in~\autoref{Defmatsimp} with $\cf(\gamma)>\omega$. 
Then, for any $\xi\leq\pi$,
\begin{enumerate}[label=\rm (\alph*)]
    \item  $\Por_{\gamma,\xi}$ is the direct limit of $\seq{\Por_{\alpha,\xi}}{\alpha<\gamma}$, and
    \item if $\eta<\cf(\gamma)$ and $\dot{f}$ is a $\Por_{\gamma,\xi}$-name of a function from $\eta$ into $\bigcup_{\alpha<\gamma}V_{\alpha,\xi}$ then $\dot{f}$ is forced to be equal to a $\Por_{\alpha,\xi}$-name for some $\alpha<\gamma$.  In particular, the reals in $V_{\gamma,\xi}$ are precisely the reals in $\bigcup_{\alpha<\gamma}V_{\alpha,\xi}$.
\end{enumerate}  
\end{lemma}

Using a Polish relational system that is Tukey-equivalent with $\Cbf_\Mwf$ (see \autoref{ExamPresPro}~\ref{(M-cohen)}) we have the following result.

\begin{theorem}[{\cite[Thm.~5.4]{CM}}]\label{matsizebd}
Let $\Por_{\gamma, \pi}$ be a simple matrix iteration as in~\autoref{Defmatsimp}. Assume that, for any $\alpha<\gamma$, there is some $\xi_\alpha<\pi$ such that $\Por_{\alpha+1,\xi_\alpha}$ adds a Cohen real $\dot{c}_\alpha\in X$ over $V_{\alpha,\xi_\alpha}$. 
   Then, for any $\alpha<\gamma$, $\Por_{\alpha+1,\pi}$ forces that $\dot{c}_{\alpha}$ is Cohen over $V_{\alpha,\pi}$. 

   In addition, if $\cf(\gamma)>\omega_1$ and $f\colon \cf(\gamma)\to\gamma$ is increasing and cofinal, then  $\Por_{\gamma,\pi}$ forces that $\set{\dot{c}_{f(\zeta)}}{\zeta<\cf(\gamma)}$ is a strongly $\cf(\gamma)$-$\Cbf_\Mwf$-unbounded family. In particular, $\Por_{\gamma,\pi}$ forces $\gamma\leqT \Cbf_\Mwf$ and $\non(\Mwf)\leq\cf(\gamma)\leq\cov(\Mwf)$. 
\end{theorem}

\begin{definition}[{\cite[Def.~4.2]{BCM}}]\label{Defmatuf}
Let $\theta\geq\aleph_1$ and let $\Por_{\gamma, \pi}$ be a simple matrix iteration as in~\autoref{Defmatsimp}. Say that $\Por_{\gamma, \pi}$ is a~\emph{${<}\theta$-uf-extendable matrix iteration} if for each $\xi<\pi$, $\Por_{\Delta(\xi),\xi}$ forces that $\Qnm_\xi$ is a $\theta_\xi$-uf-linked poset for some cardinal $\theta_\xi<\theta$ (decided in the ground model).
\end{definition}

The next result shows the effect of uf-extendable matrix iterations on $\la\baire,\leq^*\ra$.
 
\begin{theorem}[{\cite[Thm.~4.4]{BCM}}]\label{mainpres}
Assume that $\theta\leq\mu$ are uncountable cardinals with $\theta$ regular. Let $\Por_{\gamma,\pi}$ be a ${<}\theta$-uf-extendable matrix iteration as in~\autoref{Defmatuf} such that
\begin{enumerate}[label = \rm (\roman*)]
    \item $\gamma\geq\mu$ and $\pi\geq\mu$,
    \item for each $\alpha<\mu$, $\Delta(\alpha)=\alpha+1$ and $\Qnm^*_\alpha$ is Cohen forcing, and
    \item $\dot c_\alpha$ is a $\Por_{\alpha+1,\alpha+1}$-name of the Cohen real in $\omega^\omega$ added by $\Qnm^*_\alpha$.
\end{enumerate}
Then $\Por_{\alpha,\pi}$ forces that $\set{\dot c_\alpha}{\alpha<\mu}$ is strongly $\theta$-$\omega^\omega$-unbounded, in particular, $\Cbf_{[\mu]^{<\theta}}\leqT \omega^\omega$.\footnote{Although the conclusion in the cited reference is different, the same proof works.}
\end{theorem}

Now, we have developed enough machinery to prove~\autoref{Mainthm} and~\ref{Thm:faddSN}. Here,
we denote the relational systems (some introduced in~\autoref{ExamPresPro}) $\Rbf_0:=\Lc^*$, $\Rbf_1:=\Cn$, and $\Rbf_2:=\baire$.

\begin{theorem}\label{Mthm}
Let $\lambda_0\leq\lambda_1\leq\lambda_2\leq\lambda_3\leq\lambda_4$ be uncountable regular cardinals, and $\lambda_5$ a cardinal such that $\lambda_5\geq\lambda_4$ and $\cof([\lambda_5]^{<\lambda_i})=\lambda_5 = \lambda_5^{\aleph_0}$ for $i\leq 2$. Then there is a ccc poset forcing:

\begin{enumerate}[label=\rm(\arabic*)]
    \item $\cfrak=\lambda_5$;
    
    \item $\Rbf_i\eqT\Cbf_{[\lambda_5]^{<{\lambda_i}}}$ for $0\leq i\leq2$;

    \item $\Cbf_{[\lambda_5]^{<\lambda_0}} \leqT \SNwf$ and $\Cbf_\SNwf^\perp \eqT \Cbf_{[\lambda_5]^{<\lambda_1}}$;

    \item $\lambda_3\leqT\Cbf_\Mwf$ and $\lambda_4\leqT \Cbf_\Mwf$; and

    \item $\Rbf_b\leqT\lambda_4\times\lambda_3$ for each $b\in\baire$.
\end{enumerate}
In particular, it is forced that:
\begin{align*}
\add(\Nwf) & =\non(\NAwf)=\add(\SNwf)=\lambda_0\leq\cov(\Nwf)=\cov(\SNwf)=\lambda_1\leq\add(\Mwf)=\bfrak=\lambda_2\\
&\leq\non(\MAwf)=\non(\Mwf)=\lambda_3\leq\cov(\Mwf)=\sup\set{\dfrak(\Rbf_b)}{b\in\omega^\omega} = \lambda_4\\
&\leq\dfrak=\non(\SNwf) = \non(\Nwf)=\cfrak=\lambda_5.    
\end{align*}
\end{theorem}
\begin{proof}
For each $\rho<\lambda_4\lambda_3$ denote $\lambda_\rho:=\lambda_4+\lambda_5\rho$. Fix a bijection $g=(g_0, g_1,g_2):\lambda_5\to\{0, 1,2\}\times\lambda_4\times\lambda_5$ and a function $t\colon\lambda_4\lambda_3\to\lambda_4$ such that, for any $\alpha<\lambda_4$, $t^{-1}\llbracket\{\alpha\}\rrbracket$ is cofinal in $\lambda_4\lambda_3$.
 
We are going to build a ccc poset of the form $\Cor_{\lambda_5}\ast\Por$ where $\Por$  is constructed as follows: 

Let $V_{0,0}:=V^{\Cor_{\lambda_5}}$.  We construct $\Por:=\Por_{\gamma,\,\pi}$ from a ${<}\lambda_2$-uf-extendable matrix iteration with $\gamma=\lambda_4$ and $\pi=\lambda_4+\lambda_5\lambda_4\lambda_3$, starting with:
\begin{enumerate}[label=\rm (C\arabic*)]
    \item\label{(C1)} $\Delta(\alpha):=\alpha+1$ and $\Qnm^*_{\alpha}=\Cor_{\alpha}$ for $\alpha\leq\lambda_4$. 
\end{enumerate}

Let us define the matrix iteration at each $\xi=\lambda_\rho+\varepsilon$ for $\rho<\lambda_4\lambda_3$ and $\varepsilon<\lambda_5$ as follows. Denote
\begin{align*}
\Qor^+_0 & := \Loc, & \Qor^+_1 &:= \Bor, & \Qor^+_2 &:=\Dor,\\
X_0 & := \omega^\omega, & X_1 & := \Omega, &  X_2 & := \omega^\omega. 
\end{align*}
For $j<3$, $\rho<\lambda_4\lambda_3$ and $\alpha<\lambda_4$, choose 
\begin{enumerate}[label=(E$j$)]
   \item\label{Ej} a collection $\set{\Qnm_{j,\alpha,\zeta}^\rho}{\zeta<\lambda_5}$ of nice $\Por_{\alpha,\lambda_\rho}$-names for posets of the form $(\Qor^+_j)^N$ for some transitive model $N$ of ZFC with $|N|<\lambda_{j}$
   such that, for any $\Por_{\alpha,\lambda_\rho}$-name $\dot F$ of a subset of $X_j$ of size ${<}\lambda_{j}$, there is some $\zeta<\lambda_5$ such that, in $V_{\alpha,\lambda_\rho}$, $\Qnm^\rho_{j,\alpha,\zeta} = (\Qor^+_j)^N$ for some $N$ containing $\dot F$,\footnote{This is possible by the assumption $\cof([\lambda_5]^{<\lambda_j}) = \lambda_5$, which is preserved after any ccc forcing extension.} and 
\end{enumerate}

\begin{enumerate}[label=(E${}^\rho$)]
   \item\label{Erho} an enumeration $\set{\dot{b}_{\zeta}^\rho}{\zeta<\lambda_5}$ of all the nice $\Por_{t(\rho),\lambda_\rho}$-names for all the members of $(\omega\menos\{0\})^\omega$,
  
\end{enumerate} 
and set: 
\begin{enumerate}[label=\rm (C\arabic*)]
\setcounter{enumi}{1}
    \item if $\xi=\lambda_\rho+2\varepsilon$ for some $\varepsilon<\lambda_5$, put $\Delta(\xi):=t(\rho)$ and 
    $\Qnm^*_{\xi}=\Por_{\dot b^\rho_{\varepsilon}}^{V_{\Delta(\xi),\xi}}$; and
    
    \item  if $\xi=\lambda_\rho+2\varepsilon+1$ for some $\varepsilon<\lambda_5$, put $\Delta(\xi):=g_1(\varepsilon)$ and $\Qnm^*_{\xi}=\Qnm^{\rho}_{g(\varepsilon)}$.  
\end{enumerate}
According to~\autoref{Defmatsimp}, the above settles the construction of $\Por$ as a ${<}\lambda_2$-uf-extendable matrix iteration by \autoref{exm:ufl} and \autoref{uf:Pb}. First, observe that $\Por$ is ccc. It is also clear that $\Por$ forces $\cfrak = \lambda_5$ by the assumption $\lambda_5 = \lambda_5^{\aleph_0}$. We now prove that $\Por$ forces what we want:

\begin{enumerate}[label=\rm($\boxplus_\arabic*$)]
    \item\label{p1} $\Por$ forces  $\Rbf_0\eqT\Cbf_{[\lambda_5]^{<{\lambda_0}}}$: $\Cbf_{[\lambda_5]^{<{\lambda_0}}}\leqT\Rbf_0$ is forced by~\autoref{Comgood} because, for each $\xi<\pi$, $\Por_{\gamma, \xi}$ forces that $\Qnm_{\gamma, \xi}$ is $\lambda_0$-$\Rbf_0$-good. Indeed, the case $\xi=\lambda_\rho+2\varepsilon$ for some $\rho<\lambda_4\lambda_3$ and $\varepsilon<\lambda_5$ follows by~\autoref{ExamPresPro}~\ref{Pres(uni-null)}; when $\xi=\lambda_\rho+2\varepsilon+1$ for some $\rho<\lambda_4\lambda_3$ and $\varepsilon<\lambda_5$, we distinguish three subcases: the subcase $g_0(\varepsilon)=0$ is clear by~\autoref{smallgds}; the subcases $g_0(\varepsilon)=1$ and $g_0(\varepsilon)=2$ follow by~\autoref{ExamPresPro}~\ref{Pres(uni-null)}. 
    
    On the other hand, let $\dot A$ be a $\Por$-name for a subset of $\omega^\omega$ of size ${<}\lambda_0$. By employing~\autoref{realint} we can find $\alpha<\lambda_4$ and $\rho<\lambda_4\lambda_3$ such
that $\dot A$ is $\Por_{\alpha,\lambda_\rho}$-name. By~(E0), we can find a $\zeta<\theta_6$ and a $\Por_{\alpha,\lambda_\rho}$-name $\dot N$ of a transitive model
of ZFC of size ${<}\lambda_0$ such that $\Por_{\alpha,\lambda_\rho}$ forces that $\dot N$ contains $\dot A$ as a subset and $\Loc^{\dot N}=\Qnm_{0,\alpha,\zeta}^\rho$, so the
generic slalom added by $\Qnm^*_{\xi}=\Qnm_{g(\varepsilon)}^\rho$ localizes all the reals in $\dot A$ where $\varepsilon:=g^{-1}(0,\alpha,\zeta)$ and $\xi=\lambda_\rho+2\varepsilon+1$. Hence, by utilizing~\autoref{itsmallsets}, $\Por$ forces that $\Rbf_0\leqT\Cbf_{[\lambda_5]^{<{\lambda_0}}}$ because $|\omega^\omega|=|\pi|=\lambda_5$.

\item\label{p2} $\Por$ forces that $\Rbf_i\eqT\Cbf_{[\lambda_5]^{<{\lambda_i}}}$ for $i\in\{1,2\}$: For $i=1$, since $\Por$ can be obtained by the FS iteration $\seq{\Por_{\lambda_4,\xi},\Qnm_{\lambda_4,\xi}}{\xi<\pi}$ and 
all its iterands are $\lambda_1$-$\Rbf_1$-good (see~\autoref{ExamPresPro}~\ref{Pres(null-cov)}),  $\Por$ forces $\Cbf_{[\lambda_5]^{<{\lambda_1}}}\leqT \Rbf_1$ by applying~\autoref{Comgood}; and for $i=2$, since the matrix iteration is ${<}\lambda_2$-uf-extendable, by~\autoref{mainpres}, $\Por$ forces $\Cbf_{[\lambda_5]^{<{\lambda_2}}}\leqT\Rbf_2$.

On the other hand, $\Por$ forces that $\Rbf_i\leqT\Cbf_{[\lambda_5]^{<{\lambda_i}}}$ for $i\in\{1,2\}$ by a similar argument as in~\ref{p1} (using~\ref{Ej} for $j\in\{1,2\}$).

\item\label{p2.1} $\Por$ forces that $\Cbf_{[\lambda_5]^{<\lambda_0}}\leqT \SNwf$: Immediate from~\autoref{mainpresaddSN} because all iterands are $\lambda_0$-$\Rbf^f_\Gwf$-good (see \autoref{ExamPresPro}~\ref{PrsSN}).

\item\label{p2.2} $\Por$ forces that $\Cbf_\SNwf^\perp \eqT \Cbf_{[\lambda_5]^{<\lambda_1}}$: Since $\Por$ is obtained by a FS iteration of precaliber $\lambda_1$ posets, by~\autoref{thm:precaliber} $\Por$ forces $\Cbf_{[\lambda_5]^{<{\lambda_1}}}\leqT\Cbf_\SNwf^\perp$ , and in this way $\Cbf_{[\lambda_5]^{<{\lambda_1}}}\eqT\Cbf_\SNwf^\perp$ because
$\Cbf^\perp_{\SNwf}\leqT\Cbf^\perp_{\Nwf}$ (in ZFC).

\item\label{p3} $\Por$ forces that $\lambda_3\leqT\Cbf_\Mwf$ and $\lambda_4\leqT\Cbf_\Mwf$:  Since $\cf(\pi)=\lambda_3$, the first one follow by applying~\autoref{lem:strongCohen} whereas the latter follow by~\autoref{matsizebd}.

\item\label{p4} $\Por$ forces that $ \Rbf_b\leqT\lambda_4\times\lambda_3$ for each $b\in\baire$: 
Since $\lambda_4\lambda_3 \eqT \lambda_3$ 
it suffices to prove that, in $V_{\gamma,\pi}$, there are maps $\Psi_-\colon\prod b\to\lambda_4\times\lambda_4\lambda_3$ and $\Psi_+\colon\lambda_4\times\lambda_4\lambda_3\to\Ior\times\prod b$ such that, for any $x\in\prod b$ and any $(\alpha,\rho)\in\lambda_4\times\lambda_4\lambda_3$, if $\Psi_-(x)\leq (\alpha,\rho)$, then $x\sqsubset\Psi_+(\alpha,\rho)$. To this end, denote by $(  J_{\xi}, h_{\xi})$ the $\Rbf_b$-dominating real over $V_{t(\rho),\xi}$ added by $\Qnm_{t(\rho),\xi}$ when $\xi=\lambda_\rho+2\varepsilon$ for some $\rho<\lambda_4\lambda_3$ and  $\varepsilon<\lambda_5$.
 
By~\autoref{realint}, there exists an $\alpha_{ b}<\lambda_4$ such that $b \in V_{\alpha_{ b},\pi}$. Moreover, since $\pi$ has cofinality $\lambda_3$, we can find $\rho_{ b}<\lambda_4\lambda_3$ such that $ b\in V_{\alpha_{ b},\lambda_{\rho_{ b}}}$. Now, for $x\in\prod b\cap V_{\lambda_4,\pi}$, we can find $\alpha_b\leq \alpha_x<\lambda_4$ and $\rho_b\leq \rho_x<\lambda_4\lambda_3$ such that $x\in V_{\alpha_x,\lambda_{\rho_x}}$, so put $\Psi_-(x):=(\alpha_x,\rho_x)$. 

For $(\alpha,\rho)\in\lambda_4\times\lambda_4\lambda_3$, find some $\rho'\geq \rho$ in $\lambda_4\lambda_3$ such that $t(\rho')=\alpha$. When $(\alpha,\rho)\geq (\alpha_b,\rho_b)$, since $b\in V_{\alpha,\lambda_{\rho'}}$, by~\ref{Erho} there is an $\varepsilon<\lambda_5$ such that $ b=b_{\varepsilon}^{\rho'}$, so define $\Psi_+(\alpha,\rho):=(J_{\xi}, h_{\xi})$ where $\xi=\lambda_{\rho'}+2\varepsilon$; otherwise, $\Psi_+(\alpha,\rho)$ can be anything. It is clear that $(\Psi_-,\Psi_+)$ is the required Tukey connection.

\item $\Por$ forces $\add(\SNwf) = \non(\NAwf)=\lambda_0$: 
Since $\add(\Nwf) \leq \non(\NAwf) \leq \add(\SNwf)$, it is enough to show that $\Por$ forces $\add(\SNwf)\leq \lambda_0$. But this is immediate from~\ref{p2.1}.


\item $\Por$ forces $\non(\MAwf) = \lambda_3$ and $\sup\set{\dfrak(\Rbf_b)}{b\in\omega^\omega} = \lambda_4$: By~\ref{p3} and~\ref{p4}, since $\Cbf_\Mwf \leqT \Rbf_b$ whenever $b\geq^*2$, $\bfrak(\Rbf_b) = \lambda_3$ and $\dfrak(\Rbf_b) = \lambda_4$. Hence, $\non(\MAwf) = \lambda_3$ by \autoref{th:BJ}.
\end{enumerate}

This finishes the proof of the theorem.
\end{proof}

We now proceed to show~\autoref{Thm:faddSN}.

\begin{theorem}\label{Thm:ufQsn}
Under the same hypothesis as in \autoref{Mthm}, 
there is a ccc poset forcing:
\begin{enumerate}[label=\rm(\arabic*)]
    \item $\cfrak=\lambda_5$;
    
    \item $\Lc^*\eqT \omega^\omega \eqT \Cbf_{[\lambda_5]^{<{\lambda_0}}}$;

    \item $\lambda_3\leqT\Cbf_\Mwf$ and $\lambda_4\leqT \Cbf_\Mwf$; 

    \item\label{SN4} $\lambda_3\leqT\Cbf^\perp_\SNwf$ and $\lambda_4\leqT \Cbf^\perp_\SNwf$;

    \item $\Rbf_b\leqT \lambda_4\times \lambda_3$ for all $b\in\omega^\omega$;

    \item\label{SN6} $\SNwf \leqT (\lambda_4\times \lambda_3)^{\lambda_5}$; and

    \item\label{SN7} $\Cbf^\perp_\Nwf \leqT \lambda_4\times\lambda_3$.
\end{enumerate}
In particular, it is forced that:
\begin{align*}
\add(\Nwf) & =\bfrak=\lambda_0\leq
\add(\SNwf)= \cov(\SNwf)= \non(\MAwf)= \cov(\Nwf) = \non(\Mwf)=\lambda_3\\
 & \leq \cov(\Mwf)=\sup_{b\in \omega^\omega}\dfrak(\Rbf_b) = \non(\SNwf) = \non(\Nwf) = \lambda_4\leq\dfrak=\cfrak=\lambda_5.
\end{align*}
\end{theorem}
\begin{proof}
    We proceed as in \autoref{Mthm}. Set $\lambda_\rho$ ($\rho<\lambda_4\lambda_3$) and $t$ as in there, and fix a bijection $g\colon \lambda_5 \to \lambda_4\times\lambda_5$. First add $\lambda_5$-many Cohen reals, and afterwards construct a  ${<}\lambda_0$-uf-extendable matrix iteration $\Por = \Por_{\gamma,\pi}$ with $\gamma = \lambda_4$ and $\pi = \lambda_4 + \lambda_5\lambda_4\lambda_3$, defining the first $\lambda_4$-many steps as in \ref{(C1)}.

    For $\rho<\lambda_4\lambda_3$ and $\alpha<\lambda_4$, choose 
\begin{enumerate}[label=(F0)]
   \item\label{F0} a collection $\set{\Qnm_{\alpha,\zeta}^\rho}{\zeta<\lambda_5}$ of nice $\Por_{\alpha,\lambda_\rho}$-names for posets of the form $\Loc^N$ for some transitive model $N$ of ZFC with $|N|<\lambda_{0}$
   such that, for any $\Por_{\alpha,\lambda_\rho}$-name $\dot F$ of a subset of $\baire$ of size ${<}\lambda_{0}$, there is some $\zeta<\lambda_5$ such that, in $V_{\alpha,\lambda_\rho}$, $\Qnm^\rho_{\alpha,\zeta} = \Loc^N$ for some $N$ containing $\dot F$, and 
\end{enumerate}

\begin{enumerate}[label=(F${}^\rho$)]
   \item\label{Frho} enumerations $\set{\dot{b}_{\zeta}^\rho}{\zeta<\lambda_5}$ and $\set{\dot{f}_{\zeta}^\rho}{\zeta<\lambda_5}$ of all the nice $\Por_{t(\rho),\lambda_\rho}$-names for all the members of $(\omega\menos\{0\})^\omega$, and for all the increasing functions in $\omega^\omega$, respectively, 
\end{enumerate}
and set: 
\begin{enumerate}[label=\rm (C\arabic*)]
\setcounter{enumi}{1}
    \item if $\xi=\lambda_\rho+4\varepsilon$ for some $\varepsilon<\lambda_5$, put $\Delta(\xi):=t(\rho)$ and 
    $\Qnm^*_{\xi}=\Por_{\dot b^\rho_{\varepsilon}}^{V_{\Delta(\xi),\xi}}$;

    \item if $\xi=\lambda_\rho+4\varepsilon+1$ for some $\varepsilon<\lambda_5$, put $\Delta(\xi):=t(\rho)$ and 
    $\Qnm^*_{\xi}=\Qor_{\dot f^\rho_{\varepsilon}}^{V_{\Delta(\xi),\xi}}$;

    \item if $\xi=\lambda_\rho+4\varepsilon+2$ for some $\varepsilon<\lambda_5$, put $\Delta(\xi):=t(\rho)$ and 
    $\Qnm^*_{\xi}=\Bor^{V_{\Delta(\xi),\xi}}$; and 
    
    \item if $\xi=\lambda_\rho+4\varepsilon+3$ for some $\varepsilon<\lambda_5$, put $\Delta(\xi):=g_1(\varepsilon)$ and $\Qnm^*_{\xi}=\Qnm^{\rho}_{g(\varepsilon)}$. 
\end{enumerate}
The construction is indeed a ${<}\lambda_0$-uf-extendable iteration. We prove the claims related to $\SNwf$, as the rest can be proved as in \autoref{Mthm}.

\ref{SN4} $\Por$ forces $\lambda_3\leqT\Cbf^\perp_\SNwf$ and $\lambda_4\leqT \Cbf^\perp_\SNwf$:  Immediately by \autoref{thm:cfcovSN} applied to $\seq{\Por_{\lambda_4,\xi}}{ \xi\leq \pi}$ and $\seq{\Por_{\alpha,\pi}}{ \alpha\leq \lambda_4}$, respectively.

\ref{SN6} Work in $V_{\gamma,\pi}$. Let $D\subseteq \omega^\omega$ be the set of all increasing functions. For each $f\in D$ let $f'\in\baire$ be defined by $f'(i) := f((i+1)^2)$. Since $\lambda_4\lambda_3\eqT \lambda_3$, we construct a Tukey connection $\Phi_-\colon \SNwf\to (\lambda_4\times\lambda_4\lambda_3)^D$, $\Phi_+\colon (\lambda_4\times\lambda_4\lambda_3)^D \to \SNwf$.

For $A\in\SNwf$, we can find $\la \tau^A_f\st\, f\in D\ra \subseteq (2^{<\omega})^\omega$ such that $\hgt_{\tau^A_f} = f'$ and $A\subseteq \bigcap_{f\in D}[\tau^A_f]_\infty$. By \autoref{realint}, for each $f\in D$ find $(\alpha^A_f,\rho^A_f)\in \lambda_4\times\lambda_4\lambda_3$ such that $f,\tau^A_f\in V_{\alpha^A_f,\lambda_{\rho^A_f}}$. So set $\Phi_-(A):= \seq{(\alpha^A_f,\rho^A_f)}{f\in D}$.

Whenever $\xi = \lambda_\rho + 4\varepsilon +1$ for some $\rho<\lambda_4\lambda_3$ and $\varepsilon<\lambda_5$, let $\sigma^*_{\xi} \in 2^{f^\rho_\varepsilon}$ be the $\Qor_{f^\rho_\varepsilon}$-generic real over $V_{\Delta(\xi),\xi}$ added in $V_{\Delta(\xi),\xi+1}$. 
Let $z=\seq{(\beta_f,\varrho_f)}{f\in D}$ in $(\lambda_4\times \lambda_4\lambda_3)^D$. For each $f\in D$, find $\varrho'_f\geq \varrho_f$ in $\lambda_4\lambda_3$ such that $t(\varrho'_f) = \beta_f$. When $f\in V_{\beta_f,\lambda_{\varrho'_f}}$, find $\varepsilon_f<\lambda_5$ such that $f = f^{\varrho'_f}_{\varepsilon_f}$, and let $\sigma_f := \sigma^*_{\xi_f}$ where $\xi_f:= \lambda_{\varrho'_f}+3\varepsilon_f+1$, otherwise let $\sigma_f$ be anything in $2^f$. Set $\Phi_+(z) := \bigcap_{f\in D}\bigcup_{n<\omega}[\sigma_f(n)]$, which is clearly in $\SNwf$.

It remains to show, by using the notation above, that $\Phi_-(A)\leq z$ implies $A\subseteq \Phi_+(z)$. If $\Phi_-(A)\leq z$, i.e.\ $\alpha^A_f\leq \beta_f$ and $\rho^A_f\leq \varrho_f$ for all $f\in D$, then $f,\tau^A_f\in V_{\beta_f,\varrho'_f}$, so $\sigma_f = \sigma^*_{\xi_f}$ and $[\tau^A_f]_\infty \subseteq \bigcup_{n<\omega}[\sigma_f(n)]$. Therefore, $A\subseteq \Phi_+(z)$. 
\end{proof}

\begin{remark}
Under further assumptions in the ground model, we could force some value to $\cof(\SNwf)$ by using techniques from~\cite{CarMej23} as in~\cite[Sec.~6]{BCM2}.    
\end{remark}

\section{Open problems}

We were able to determine in \autoref{MainThm1} that $\add(\NAwf)= \non(\NAwf)$, but the case of $\MAwf$ is unknown.

\begin{problem}\label{QanMA}
    Does $\thzfc$ prove $\add(\MAwf) = \non(\MAwf)$?
\end{problem}

Recall that $\add(\MAwf) = \non(\MAwf)$ follows from $\non(\MAwf) \leq \bfrak$ (see \autoref{cor:addnonMA}). In the case $\bfrak<\non(\MAwf)$, we obtain by \autoref{chPawmn} that $\bfrak=\add(\Mwf) \leq \add(\MAwf) \leq \non(\MAwf)$.

Another possible equality is considered in the following.

\begin{problem}\label{QaddMAM}
    Does $\thzfc$ prove $\add(\MAwf) = \add(\Mwf)$?
\end{problem}

Both problems cannot have positive answers simultaneously because it is consistent with ZFC that $\bfrak < \non(\MAwf)$, which follows from \autoref{Mthm}.

Notice that $\add(\MAwf) = \add(\Mwf)$ is equivalent to $\add(\MAwf)\leq \bfrak$ by \autoref{chPawmn}. So, in contrast, we may ask:

\begin{problem}
Is it consistent with $\thzfc$ that $\bfrak<\add(\MAwf)$?
\end{problem}

Since $\add(\NAwf) = \non(\NAwf)$, we know the consistency of $\bfrak<\add(\NAwf)$ with ZFC (see~(\ref{cmresult}) in~\autoref{sec:intro}).

In \autoref{Sec:zfc} we mentioned that $\cov(\SNwf)=\cov(\MAwf) = \cov(\NAwf)=\cfrak$ in Sacks model, so these covering numbers do not have ``reasonable" upper bounds in ZFC other than $\cfrak$. The consistency of $\cov(\SNwf)<\add(\Mwf)$ with ZFC is known~\cite{P90}, and we proved the consistency of $\cov(\MAwf) < \non(\Nwf)$ in \autoref{Thm:covMA}. However, we do not know the answer to the following.

\begin{problem}
    Is it consistent with $\thzfc$ that $\cov(\NAwf) < \cfrak$?
\end{problem}

We now discuss about the cofinality numbers. Yorioka and the authors have investigated the cofinality of $\SNwf$. Yorioka~\cite{Yorioka} proved that it is consistent with ZFC that $\cof(\SNwf)<\cfrak$. Building in his work, we~\cite{cardona, CarMej23} have obtained nice lower and upper bounds for $ \cof(\SNwf)$, which led us to considerably improve Yorioka's results. As to the cofinality of $\NAwf$ and $\MAwf$, we do not know anything about their behavior. For instance, we may ask: 

\begin{problem}\label{prob:cof}
    Does $\thzfc$ prove some inequality among $\cof(\NAwf)$, $\cof(\MAwf)$, $\cof(\SNwf)$ and $\cfrak$?
\end{problem}

Notice that $\MAwf$ does not have a Borel base because $\MAwf\subseteq\SNwf$, and no perfect subset of $\cantor$ is in $\SNwf$. The same applies to $\NAwf$.

In this work,
we have solved~\autoref{P:ma_na}~\ref{P:ma_na3}, i.e.\ the consistency of $\non(\NAwf) < \bfrak < \non(\MAwf)$ with ZFC, but the answer to the remaining questions are unknown: Are each of the following statements consistent with ZFC?
\begin{enumerate}[label=\normalfont(\alph*)]
    \item\label{O1} $\bfrak<\non(\NAwf)<\non(\MAwf)$.
    
    \item\label{O2} $\non(\NAwf)<\non(\MAwf)<\bfrak$.
\end{enumerate}
We know that 
\[\add(\Mwf)\leq\non(\MAwf)\leq \non(\Ewf) \leq \min\{\non(\Mwf), \non(\Nwf)\}.\]
Therefore, any FS iterations of ccc posets (with length of uncountable cofinality) forces that 
$\bfrak\leq\non(\MAwf)\leq\non(\Ewf)$
because any such iterations forces that $\non(\Mwf)\leq\cov(\Mwf)$. So we can conclude that FS iterations do not work to solve~\ref{O2}. Hence, alternative
methods are required.

One of our original intentions to introduce the poset $\Qor_f$ from \autoref{foraddSN}, which works to increase $\add(\SNwf)$, was to solve the following.

\begin{problem}[{\cite[Q.~8.1]{BCM2}}]
    Are each of the following statements consistent with $\thzfc$?
    \begin{enumerate}[label = \normalfont (\arabic*)]
        \item $\add(\Nwf) < \add(\SNwf) < \bfrak$.
        \item $\add(\Nwf) < \bfrak < \add(\SNwf)$.
    \end{enumerate}
\end{problem}

Since $\minLc \leq \add(\SNwf)$ and $\add(\Nwf) = \min\{\bfrak,\minLc\}$ (see~\autoref{chNPaw} and~\ref{chPawmn}), a necessary condition of the above is that $\add(\Nwf) = \minLc$. It is unclear to us why $\Qor_f$ should not increase $\minLc$. On the other hand, the bounding number $\bfrak$ is not a problem because $\Qor_f$ is uniformly $\sigma$-uf-$\lim$-linked, so it can be controlled.



The second author \cite{mejiamatrix} has constructed a forcing model where the four cardinal
characteristics associated with $\Nwf$ are pairwise different, the first author ~\cite{Car23} has produced a similar model for $\Ewf$, and the first model for $\Mwf$ (without using large cardinals) appears in~\cite{BCM}. 
In this context, we ask:

\begin{problem}
Are the following statements consistent?
\begin{enumerate}[label=\rm (\arabic*)]  

    \item $\non(\NAwf)<\cov(\NAwf)<\cof(\NAwf)$.
    
    \item $\add(\MAwf)<\cov(\MAwf)<\non(\MAwf)<\cof(\MAwf)$. 
    
    \item $\add(\MAwf)<\non(\MAwf)<\cov(\MAwf)<\cof(\MAwf)$.
\end{enumerate}  
\end{problem}

Although we only considered one transitive additivity of a translation invariant ideal $\Iwf$ on $2^\omega$ to show its relationship with the uniformity of $\IAwf$, there are more transitive versions of the cardinal characteristics associated with $\Iwf$ as below.\footnote{In~\cite{BJ} they are denoted by $\cov^\star(\Iwf)$, $\non^\star(\Iwf)$ and $\cof^\star(\Iwf)$, while $\add^\star(\Iwf)$ is $\add^*_t(\Iwf)$ and $\add^{\star\star}(\Iwf)$ is $\add_t(\Iwf)$.} For $A,B\subseteq 2^\omega$, write $A\subseteq_+ B$ when $A\subseteq y+ B$ for some $y\in 2^\omega$. 
\begin{align*}
    \text{Transitive additivity of $\Iwf$: } && \add_t(\Iwf) & :=\bfrak(\Iwf,\Iwf,\subseteq_+),\\
    \text{Transitive covering of $\Iwf$: } && \cov_t(\Iwf) & :=\min\set{|X|}{X\subseteq\cantor\text{ and }\exists A\in\Iwf\colon A+X=\cantor},\\
    \text{Transitive uniformity of $\Iwf$: } && \non_t(\Iwf) & := \non(\Iwf),\\
    \text{Transitive cofinality of $\Iwf$: } && \cof_t(\Jwf) & := \dfrak(\Iwf,\Iwf,\subseteq_+).
\end{align*} 
Let us notice that $\add^*_t(\Jwf)$ and $\cof_t(\Jwf)$, as well as $\cov_t(\Jwf)$ and $\non_t(\Jwf)$, are not dual pairs of cardinal characteristics. 

\autoref{Figaddetic} illustrates the relationship between the transitive cardinal characteristics and the cardinal characteristics associated with $\Iwf$. See details in~\cite{Kra}.

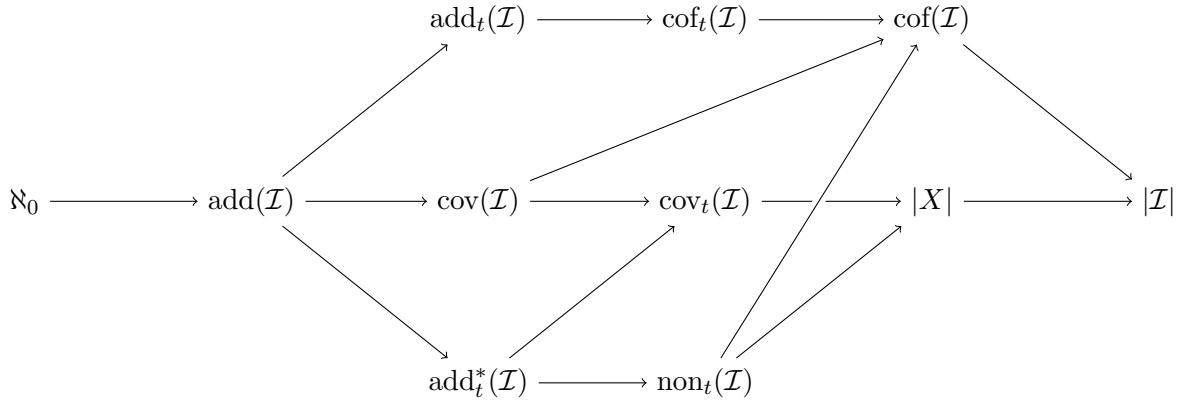
\begin{figure}[ht]
\centering
\begin{tikzpicture}[scale=1.2]
\small{
\node (azero) at (-2,1) {$\aleph_0$};
\node (addI) at (0.5,1) {$\add(\Iwf)$};
\node (covI) at (3,1) {$\cov(\Iwf)$};
\node (addtI) at (3,-1) {$\add_t^*(\Iwf)$};
\node (coftdI) at (3,3) {$\add_t(\Iwf)$}; 
\node (coftI) at (5.5,3) {$\cof_t(\Iwf)$};
\node (cofI) at (8,3) {$\cof(\Iwf)$};
\node (covtI) at (5.5,1) {$\cov_t(\Iwf)$};
\node (nontI) at (5.5,-1) {$\non_t(\Iwf)$};
\node (sizX) at (8,1) {$|X|$};
\node (sizI) at (10.5,1) {$|\Iwf|$};

\draw (azero) edge[->] (addI);
\draw (addI) edge[->] (covI);
\draw (addI) edge[->] (addtI);
\draw (covI) edge[->] (covtI);
\draw (covI) edge[->] (cofI);
\draw (covtI) edge[->] (sizX);
\draw (nontI) edge[->] (sizX);
\draw (addtI) edge[->] (nontI);
\draw (addI) edge[->] (coftdI);
\draw (addtI) edge[->] (covtI);
\draw (coftdI) edge[->] (coftI);
\draw (coftI) edge[->] (cofI);
\draw (sizX) edge[->] (sizI);
\draw (cofI) edge[->] (sizI);
\draw (nontI) edge[line width=.15cm,white,-] (cofI)
      (nontI) edge[->] (cofI);
}
\end{tikzpicture}
\caption{Hasse diagram of inequalities of the transitive cardinal characteristics associated with a translation invariant ideal $\Iwf$ on $2^\omega$.}
\label{Figaddetic}
\end{figure}

The transitive covering number was the first transitive cardinal that was studied. It appeared implicitly in 1938 in the famous Rothberger Theorem~\cite{Ro}, which states that $\cov_t(\Nwf)\leq\non(\Mwf)$ and $\cov_t(\Mwf)\leq\non(\Nwf)$. Later, Pawlikowski in~\cite{paw85} accomplished a complete description of the transitive additivity and cofinality of the null and the meager ideal. He proved in ZFC the following statements.

\begin{theorem}[{\cite{paw85}}]\label{Paw:gen}
\mbox{}
\begin{enumerate}[label=\rm(\alph*)]
    \item\label{Paw:gena} $\cof_t(\Mwf)=\dfrak$ and $\add_t(\Mwf)=\bfrak$.

    \item\label{Paw:genb} $\cof_t(\Nwf)=\cof(\Nwf)$ and $\add_t(\Nwf)=\add(\Nwf)$.

    \item\label{Paw:genc} $\add(\Nwf)=\min\{\bfrak,\add_t^*(\Nwf)\}$ and $\add(\Mwf)=\min\{\bfrak,\add_t^*(\Mwf)\}$.
\end{enumerate}    
\end{theorem}

Recall that $\add_t^*(\Nwf)\leq\add_t^*(\Mwf)$ follows from $\NAwf\subseteq \MAwf$. In addition, it is proved in~\cite[Thm.~2.7.14]{BJ} that $\cov_t(\Mwf) = \min\set{\dfrak(\Ed_b)}{b\in\omega^\omega}$. In fact, $\cov_t(\Mwf) = \non(\SNwf)$ (by Galvin's, Mycielski's and Solovay's characterization of $\SNwf$) and $\cov_t(\Nwf) = \non(\SMwf)$, where $\SMwf$ denotes the collection of strong meager subsets of $2^\omega$.

On the other hand, Kraszewski~\cite{KraS2} studied the transitive cardinals of the $\sigma$-ideal $\Swf_2$, the least nontrivial productive 
$\sigma$-ideal of subsets of the Cantor space $\cantor$. Concretely, he proved that $\add_t^*(\Swf_2)=\non(\Swf_2)=\sfrak_\omega$ (the last equality was proven by Cich\'on and Kraszewski~\cite{CicKra} where $\sfrak_\omega$ is a variaton of the splitting number), $\add_t(\Swf_2) = \aleph_1$ and  $\cof_t(\Swf_2)=\cov_t(\Swf_2)=\cfrak$.

From the rest of this section, we say that the cardinal characteristics in~\autoref{Figaddetic}, except $|X|$, $|\Iwf|$ and $\aleph_0$, are \emph{the $8$ cardinal characteristics associated with $\Iwf$}. We could ask the following:

\begin{problem}\label{prob:4tc}
For each of the ideals $\Mwf$, $\Nwf$ and $\Ewf$: Is it consistent with $\thzfc$ that their associated cardinal characteristics are pairwise different?
    
    
    

    
\end{problem}

We may have repetitions in some cases, e.g.\ $\add_t(\Nwf) = \add(\Nwf)$ and $\cof_t(\Nwf) = \cof(\Nwf)$, and also dependence, like $\add(\Mwf) = \min\{\add_t(\Mwf),\cov_t(\Mwf)\}$. More generaly, $\add(\Iwf) = \min\{\add_t(\Iwf),\add^*_t(\Iwf)\}$ (see~\cite{Kra}).

Regarding $\Nwf$, the constellation of~\autoref{Sep4Nt} holds in the second author's matrix iteration construction from~\cite[Thm.~13]{mejiamatrix}. 
On the other hand, Brendle~\cite{Brshatt} developed a sophisticated technique, called~\emph{Shattered iterations}, to obtain a model of ZFC satisfying the constellation of~\autoref{Sep4Nt2}.

\begin{figure}[ht]
\centering
\begin{tikzpicture}[scale=1]
\small{
\node (aleph1) at (-3,3) {$\aleph_1$};
\node (addn) at (0,3){$\subiii{\add(\Nwf)=\add_t(\Nwf)}$};
\node (covn) at (0,8){$\subiii{\cov(\Nwf)}$};
\node (covtn) at (2,8){$\subiii{\cov_t(\Nwf)}$};
\node (nonn) at (10,3) {$\subiii{\non(\Nwf)}$} ;
\node (cfn) at (10,8) {$\subiii{\cof(\Nwf)=\cof_t(\Nwf)}$} ;
\node (addm) at (4,3) {$\add(\Mwf)$} ;
\node (covm) at (6,3) {$\cov(\Mwf)$} ;
\node (nonm) at (4,8) {$\non(\Mwf)$} ;
\node (cfm) at (6,8) {$\cof(\Mwf)$} ;
\node (b) at (4,5.5) {$\bfrak$};
\node (nonna) at (2,4.8) {\subiii{$\add_t^*(\Nwf)$}};
\node (d) at (6,5.5) {$\dfrak$};
\node (c) at (12,8) {$\cfrak$};
\draw (aleph1) edge[->] (addn)
      (addn) edge[->] (covn)
      (covtn) edge [->] (nonm)
    (covn) edge [->] (covtn)
     (nonm)edge [->] (cfm)
      (cfm)edge [->] (cfn)
      (cfn) edge[->] (c);

\draw 
(addn) edge [->]  (nonna)
   (addn) edge [->]  (addm)
   (addm) edge [->]  (covm)
      (covm) edge [->]  (nonn)
   (nonn) edge [->]  (cfn);
\draw (addm) edge [->] (b)
      (b)  edge [->] (nonm);
\draw (covm) edge [->] (d)
      (d)  edge[->] (cfm);
\draw (b) edge [->] (d);

\draw (nonna) edge [line width=.15cm,white,-] (nonn)
    (nonna) edge [->] (nonn)
(nonna) edge [->] (covtn);

\draw[color=sug,line width=.05cm] (-2.3,2.5)--(-2.3,8.5);     
 
\draw[color=sug,line width=.05cm] (-2.3,6.5)--(3,6.5);
\draw[color=sug,line width=.05cm] (3,2.5)--(3,6.5);
\draw[color=sug,line width=.05cm] (8,2.5)--(8,8.5);
\draw[color=sug,line width=.05cm] (8,5.5)--(12,5.5);

\draw[circle, fill=yellow,color=yellow] (1,5.8) circle (0.4);
\draw[circle, fill=yellow,color=yellow] (5,6.5)   circle (0.4);
\draw[circle, fill=yellow,color=yellow] (9,6.5) circle (0.4);
\draw[circle, fill=yellow,color=yellow] (9,4.5) circle (0.4);
\node at (1,5.8) {$\theta_1$};
\node at (5,6.5) {$\theta_2$};
\node at (9,6.5) {$\theta_4$};
\node at (9,4.5) {$\theta_3$};
}
\end{tikzpicture}
\caption{Separation of the cardinals associated with $\Nwf$ where $\aleph_1\leq\theta_1\leq \theta_2\leq \theta_3$ are regular cardinals and $\theta_4\geq\theta_3$ is a cardinal such that $\theta_4^{<\theta_1}=\theta_4$. This constellation was forced in~\cite[Thm.~13]{mejiamatrix}.}
\label{Sep4Nt}
\end{figure}

\begin{figure}[ht]
\centering
\begin{tikzpicture}[scale=1]
\small{
\node (aleph1) at (-3,3) {$\aleph_1$};
\node (addn) at (0,3){$\subiii{\add(\Nwf)=\add_t(\Nwf)}$};
\node (covn) at (0,8){$\subiii{\cov(\Nwf)}$};
\node (covtn) at (2,8){$\subiii{\cov_t(\Nwf)}$};
\node (nonn) at (10,3) {$\subiii{\non(\Nwf)}$} ;
\node (cfn) at (10,8) {$\subiii{\cof(\Nwf)=\cof_t(\Nwf)}$} ;
\node (addm) at (4,3) {$\add(\Mwf)$} ;
\node (covm) at (6,3) {$\cov(\Mwf)$} ;
\node (nonm) at (4,8) {$\non(\Mwf)$} ;
\node (cfm) at (6,8) {$\cof(\Mwf)$} ;
\node (b) at (4,5.5) {$\bfrak$};
\node (nonna) at (2,4.8) {\subiii{$\add_t^*(\Nwf)$}};
\node (d) at (6,5.5) {$\dfrak$};
\node (c) at (12,8) {$\cfrak$};
\draw (aleph1) edge[->] (addn)
      (addn) edge[->] (covn)
      (covtn) edge [->] (nonm)
    (covn) edge [->] (covtn)
     (nonm)edge [->] (cfm)
      (cfm)edge [->] (cfn)
      (cfn) edge[->] (c);

\draw 
(addn) edge [->]  (nonna)
   (addn) edge [->]  (addm)
   (addm) edge [->]  (covm)
      (covm) edge [->]  (nonn)
   (nonn) edge [->]  (cfn);
\draw (addm) edge [->] (b)
      (b)  edge [->] (nonm);
\draw (covm) edge [->] (d)
      (d)  edge[->] (cfm);
\draw (b) edge [->] (d);

\draw (nonna) edge [line width=.15cm,white,-] (nonn)
    (nonna) edge [->] (nonn)
(nonna) edge [->] (covtn);

\draw[color=sug,line width=.05cm] (-2.3,6.5)--(5,6.5);
\draw[color=sug,line width=.05cm] (3,2.5)--(3,6.5);
\draw[color=sug,line width=.05cm] (5,6.5)--(5,4.4);
\draw[color=sug,line width=.05cm] (7.2,8.5)--(7.2,4.4);
\draw[color=sug,line width=.05cm] (5,4.4)--(11,4.4); 

\draw[circle, fill=yellow,color=yellow] (5,7.2) circle (0.4);
\draw[circle, fill=yellow,color=yellow] (5,3.6) circle (0.4);
\draw[circle, fill=yellow,color=yellow] (8.5,5.8) circle (0.4);
\node at (5,3.6) {$\theta_1$};
\node at (5,7.2) {$\theta_2$};
\node at (8.5,5.8) {$\theta_3$};

}
\end{tikzpicture}
\caption{Separation of the cardinals associated with $\Nwf$ with a different order where $\aleph_1\leq\theta_1\leq \theta_2\leq \theta_3$ are regular cardinals. This constellation is forced in~\cite[Cor.~30]{Brshatt}.}
\label{Sep4Nt2}
\end{figure}
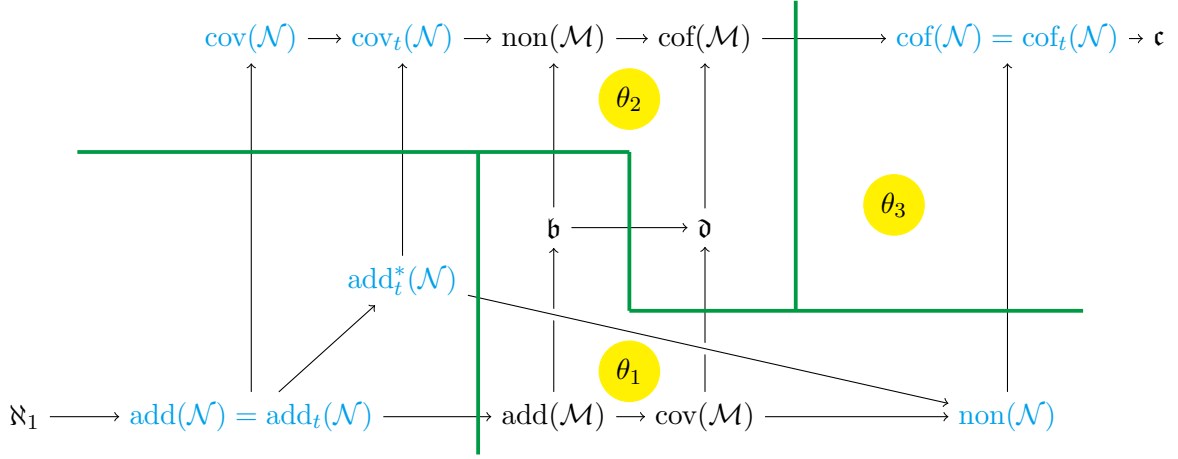

So far, shattered iterations have been used only once to deal with models in which many cardinal characteristics in Cicho\'n's diagram assume simultaneously distinct values with the order $\cov(\Mwf)<\non(\Mwf)$, so this approach may help to solve several instances of~\autoref{prob:4tc}. 

Concerning $\Mwf$, the constellation of~\autoref{Sep4M4N} holds in the forcing model from~\cite[Thm.~7.1]{BCM2}, but there the value of $\add_t^*(\Mwf)$ is unclear. It is even a challenge to separate $\add_t^*(\Mwf)$ from $\non(\Mwf)$.

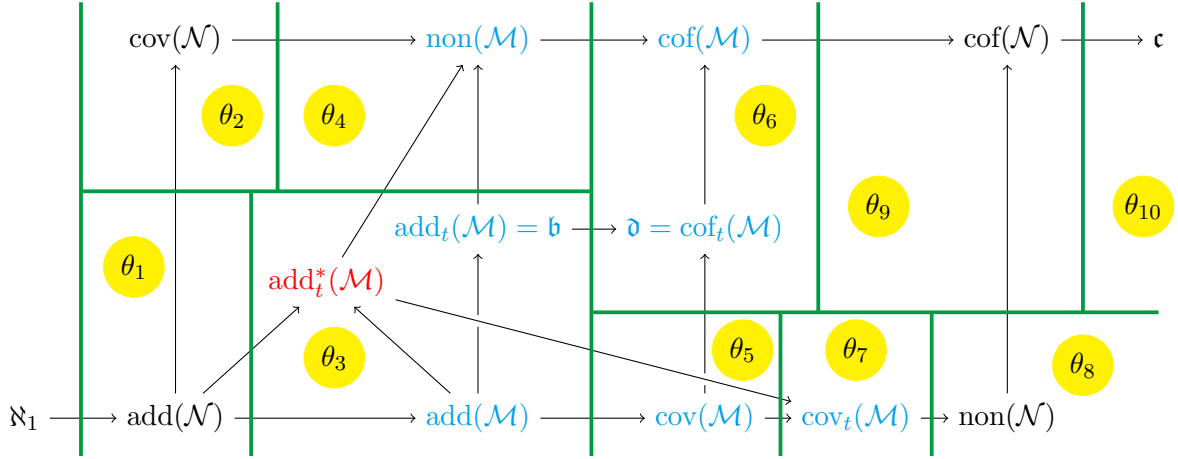
\begin{figure}[ht]
\centering
\begin{tikzpicture}[scale=1]
\small{
\node (aleph1) at (-3,3) {$\aleph_1$};
\node (addn) at (-1,3){$\add(\Nwf)$};
\node (covn) at (-1,8){$\cov(\Nwf)$};
\node (nonn) at (10,3) {$\non(\Nwf)$} ;
\node (cfn) at (10,8) {$\cof(\Nwf)$} ;
\node (addm) at (3,3) {$\subiii{\add(\Mwf)}$} ;
\node (covm) at (6,3) {$\subiii{\cov(\Mwf)}$} ;
\node (covtm) at (8,3) {$\subiii{\cov_t(\Mwf)}$} ;
\node (nonm) at (3,8) {$\subiii{\non(\Mwf)}$} ;
\node (cfm) at (6,8) {$\subiii{\cof(\Mwf)}$} ;
\node (b) at (3,5.5) {$\subiii{\add_t(\Mwf)=\bfrak}$};
\node (nonma) at (1,4.8) {\red{$\add_t^*(\Mwf)$}};
\node (d) at (6,5.5) {$\subiii{\dfrak=\cof_t(\Mwf)}$};
\node (c) at (12,8) {$\cfrak$};
\draw (aleph1) edge[->] (addn)
      (addn) edge[->] (covn)
    (covn) edge [->] (nonm)
     (nonm)edge [->] (cfm)
      (addn) edge [->]  (nonma)
      (cfm)edge [->] (cfn)
      (cfn) edge[->] (c);

\draw 
(addn) edge [->]  (addm)
   (addm) edge [->]  (covm)
      (covm) edge [->]  (covtm)
   (covtm) edge [->]  (nonn)
   (nonn) edge [->]  (cfn);
\draw (addm) edge [->] (b)
      (b)  edge [->] (nonm);
\draw (covm) edge [->] (d)
      (d)  edge[->] (cfm);
\draw (b) edge [->] (d);

\draw (nonma) edge [line width=.15cm,white,-] (covtm)
      (nonma) edge [->] (covtm);

\draw (addm) edge [->] (nonma)
      (nonma) edge [line width=.15cm,white,-] (nonm)
      (nonma) edge [->] (nonm);

\draw[color=sug,line width=.05cm] (-2.25,6)--(3,6);  
\draw[color=sug,line width=.05cm] (-2.25,2.5)--(-2.25,8.5);
\draw[color=sug,line width=.05cm] (0,2.5)--(0,6);
\draw[color=sug,line width=.05cm] (0.35,6)--(0.35,8.5);
\draw[color=sug,line width=.05cm] (3,6)--(4.5,6);
\draw[color=sug,line width=.05cm] (4.5,2.5)--(4.5,8.5);
\draw[color=sug,line width=.05cm] (7.5,4.4)--(7.5,8.5);
\draw[color=sug,line width=.05cm] (7,2.5)--(7,4.4);
\draw[color=sug,line width=.05cm] (9,2.5)--(9,4.4);
\draw[color=sug,line width=.05cm] (11,4.4)--(11,8.5);
\draw[color=sug,line width=.05cm] (4.5,4.4)--(12,4.4);

\draw[circle, fill=yellow,color=yellow] (1.1,7) circle (0.4);
\draw[circle, fill=yellow,color=yellow] (-0.25,7) circle (0.4);
\draw[circle, fill=yellow,color=yellow] (1.1,3.8) circle (0.4);
\draw[circle, fill=yellow,color=yellow] (-1.55,5) circle (0.4);
\draw[circle, fill=yellow,color=yellow]  (6.5,3.9)  circle (0.4);
\draw[circle, fill=yellow,color=yellow]  (6.8,7)  circle (0.4);
\draw[circle, fill=yellow,color=yellow]  (8,3.9)  circle (0.4);
\draw[circle, fill=yellow,color=yellow]  (11,3.7)  circle (0.4);
\draw[circle, fill=yellow,color=yellow] (8.3,5.8) circle (0.4);
\draw[circle, fill=yellow,color=yellow] (11.8,5.8) circle (0.4);
\node at (-1.55,5) {$\theta_1$};
\node at (1.1,3.8) {$\theta_3$};
\node at (-0.25,7) {$\theta_2$};
\node at (1.1,7) {$\theta_4$};
\node at (6.8,7) {$\theta_6$};
\node at (8,3.9) {$\theta_7$};
\node at (11,3.7) {$\theta_8$};
\node at (6.5,3.9) {$\theta_5$};
\node at (8.3,5.8) {$\theta_9$};
\node at (11.8,5.8) {$\theta_{10}$};
}
\end{tikzpicture}
\caption{Cicho\'n's maximum with $\cov_t(\Mwf)$ 
where for $i\leq 9$, $\theta_i$ is an uncountable regular cardinal such that $\theta_i\leq\theta_j$ for any $i\leq j$, and $\theta_{10}\geq \theta_9$ is a cardinal such that $\theta_{10}=\theta_{10}^{\aleph_0}$. 
This constellation was proved in~\cite[Thm.~7.1]{BCM2}. The value of $\add_t^*(\Mwf)$ is unclear.}
\label{Sep4M4N}
\end{figure}

\begin{problem}
Is the constellation in~\autoref{conj4Mt} consistent with $\thzfc$?
\end{problem}
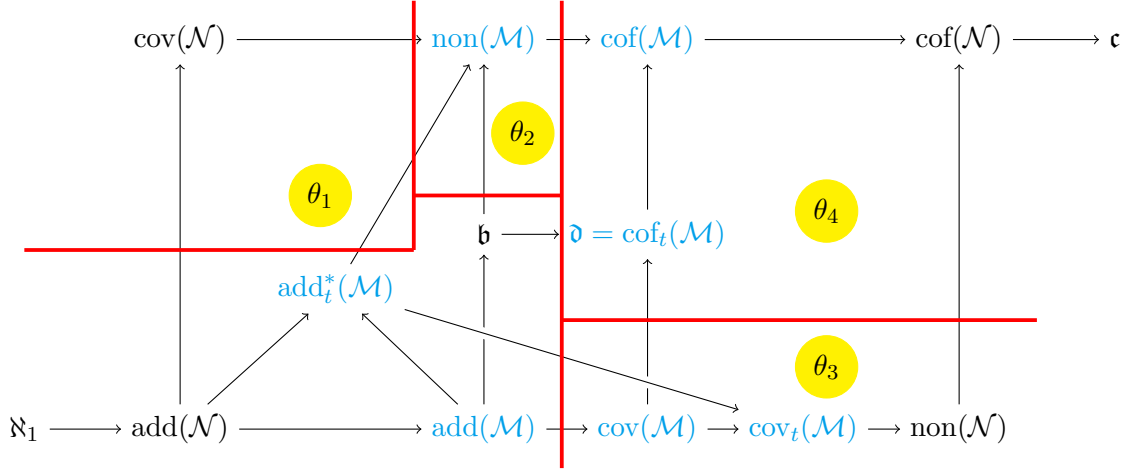
\begin{figure}[ht]
\centering
\begin{tikzpicture}[scale=1.03]
\small{
\node (aleph1) at (-2,3) {$\aleph_1$};
\node (addn) at (0,3){$\add(\Nwf)$};
\node (covn) at (0,8){$\cov(\Nwf)$};
\node (nonn) at (10,3) {$\non(\Nwf)$} ;
\node (cfn) at (10,8) {$\cof(\Nwf)$} ;
\node (addm) at (3.9,3) {$\subiii{\add(\Mwf)}$} ;
\node (covm) at (6,3) {$\subiii{\cov(\Mwf)}$} ;
\node (covtm) at (8,3) {$\subiii{\cov_t(\Mwf)}$} ;
\node (nonm) at (3.9,8) {$\subiii{\non(\Mwf)}$} ;
\node (cfm) at (6,8) {$\subiii{\cof(\Mwf)}$} ;
\node (b) at (3.9,5.5) {$\bfrak$};
\node (nonma) at (2,4.8) {\subiii{$\add_t^*(\Mwf)$}};
\node (d) at (6,5.5) {$\subiii{\dfrak=\cof_t(\Mwf)}$};
\node (c) at (12,8) {$\cfrak$};
\draw (aleph1) edge[->] (addn)
      (addn) edge[->] (covn)
    (covn) edge [->] (nonm)
     (nonm)edge [->] (cfm)
      (addn) edge [->]  (nonma)
      (cfm)edge [->] (cfn)
      (cfn) edge[->] (c);

\draw 
(addn) edge [->]  (addm)
   (addm) edge [->]  (covm)
      (covm) edge [->]  (covtm)
   (covtm) edge [->]  (nonn)
   (nonn) edge [->]  (cfn);
\draw (addm) edge [->] (b)
      (b)  edge [->] (nonm);
\draw (covm) edge [->] (d)
      (d)  edge[->] (cfm);
\draw (b) edge [->] (d);

\draw (nonma) edge [line width=.15cm,white,-] (covtm)
      (nonma) edge [->] (covtm);

\draw (addm) edge [->] (nonma)
      (nonma) edge [line width=.15cm,white,-] (nonm)
      (nonma) edge [->] (nonm);

\draw[color=red,line width=.05cm] (-2,5.3)--(3,5.3);     
\draw[color=red,line width=.05cm] (3,5.3)--(3,8.5);
\draw[color=red,line width=.05cm] (3,6)--(4.9,6);
\draw[color=red,line width=.05cm] (4.9,8.5)--(4.9,2.5);
\draw[color=red,line width=.05cm] (4.9,4.4)--(11,4.4);

\draw[circle, fill=yellow,color=yellow] (4.4,6.8) circle (0.4);
\draw[circle, fill=yellow,color=yellow] (1.8,6) circle (0.4);
\draw[circle, fill=yellow,color=yellow]  (8.3,3.8)  circle (0.4);
\draw[circle, fill=yellow,color=yellow] (8.3,5.8) circle (0.4);
\node at (1.8,6) {$\theta_1$};
\node at (4.4,6.8) {$\theta_2$};
\node at (8.3,3.8) {$\theta_3$};
\node at (8.3,5.8) {$\theta_4$};
}
\end{tikzpicture}
\caption{A constellation of the transitive cardinals associated with $\Mwf$ (Open question).}
\label{conj4Mt}
\end{figure}

Lastly, regarding~$\Ewf$, in~\autoref{sec:intro}, we mentioned that $\aleph_1=\bfrak=\non(\EAwf)<\cov(\Nwf)=\aleph_2$ holds in the model obtained by a FS iteration of length $\aleph_2$ of random forcing. There,  $\aleph_1=\add_t^*(\Ewf)=\non_t(\Ewf)<\cov_t(\Ewf)=\cof_t(\Ewf)=\aleph_2$ also holds. On the other hand, the constellation of~\autoref{SepofE} is forced in the first author's matrix iteration with ultrafilters from~\cite[Thm.~5.4]{Car23}, but the values of $\add_t^*(\Ewf)$, $\add_t(\Ewf)$, $\cov_t(\Ewf)$ and $\cof_t(\Ewf)$ are unclear. As yet it is not known how to separate more than three transitive cardinals associated with $\Ewf$.

\begin{figure}[ht]
\centering
\begin{tikzpicture}[scale=1.06]
\small{
\node (aleph1) at (-2.8,3) {$\aleph_1$};
\node (addn) at (-1,3){$\add(\Nwf)$};
\node (covn) at (-1,8.5){$\cov(\Nwf)$};
\node (nonn) at (10,3) {$\non(\Nwf)$} ;
\node (cfn) at (10,8.5) {$\cof(\Nwf)$} ;
\node (addm) at (3.19,3) {$\subiii{\add(\Mwf)=\add(\Ewf)}$} ;
\node (covm) at (7.4,3) {$\cov(\Mwf)$} ;
\node (nonm) at (3.19,8.5) {$\non(\Mwf)$} ;
\node (cfm) at (7.4,8.5) {$\subiii{\cof(\Mwf)=\cof(\Ewf)}$} ;
\node (nonea) at (0.5,3.8) {\red{$\add_t^*(\Ewf)$}};
\node (addte) at (2.2,5.75) {\red{$\add_t(\Ewf)$}};
\node (b) at (3.19,5.5) {$\bfrak$};
\node (d) at (7.4,5.5) {$\dfrak$};
\node (c) at (11.5,8.5) {$\cfrak$};
\node (none) at (0.5,5) {$\subiii{\non(\Ewf)}$};
\node (cove) at (5.8,4.5) {$\subiii{\cov(\Ewf)}$};
\node (covte) at (8.5,7.25) {$\red{\cov_t(\Ewf)}$};
\node (cofte) at (4.25,7.5) {$\red{\cof_t(\Ewf)}$};
\draw (aleph1) edge[->] (addn)
      (addn) edge[->] (covn)
      (covn) edge [->] (nonm)
      (nonm)edge [->] (cfm)
      (cfm)edge [->] (cfn)
      (cfn) edge[->] (c);

\draw(addm) edge [->]  (cove);
\draw
   (addn) edge [->]  (addm)
   (addm) edge [->]  (covm)
   (covm) edge [->]  (nonn)
   (nonn) edge [->]  (cfn);
\draw (addm) edge [->] (b)
      (b)  edge [->] (nonm);
\draw (covm) edge [->] (d)
      (d)  edge[->] (cfm);
\draw (b) edge [->] (d);

 \draw   (nonea) edge [->] (none);      
 \draw   (none) edge [->] (nonm);
     
\draw (none) edge [line width=.15cm,white,-] (nonn)
      (none) edge [->] (nonn);
      
\draw (cove) edge [line width=.15cm,white,-] (covn)
      (cove) edge [<-] (covn);
\draw (cove) edge [line width=.15cm,white,-] (covm)
      (cove) edge [<-] (covm);

\draw (addm) edge [->] (nonea); 
\draw (addm) edge [line width=.15cm,white,-] (addte)
(addm) edge [->] (addte);
\draw (cofte) edge [->] (cfm);
\draw (addte) edge [line width=.15cm,white,-] (cofte)
(addte) edge [->] (cofte);
\draw (cove) edge [line width=.15cm,white,-] (covte)
      (cove) edge [->] (covte);
\draw (cove) edge [line width=.15cm,white,-] (cfm)
(cove) edge [->] (cfm);
\draw (nonea) edge [line width=.15cm,white,-] (covte)
(nonea) edge [->] (covte);
\draw (covte) edge [line width=.15cm,white,-] (c)
      (covte) edge [->] (c);
\draw[color=sug,line width=.05cm] (-2.25,2.5)--(-2.25,9);
\draw[color=sug,line width=.05cm] (-2.25,4.45)--(1.4,4.45);
\draw[color=sug,line width=.05cm] (1.4,2.5)--(1.4,6.8);
\draw[color=sug,line width=.05cm] (1.4,6.8)--(4.9,6.8);
\draw[color=sug,line width=.05cm] (4.9,2.5)--(4.9,9);
\draw[color=sug,line width=.05cm] (4.9,5)--(11,5); 

\draw[circle, fill=yellow,color=yellow] (-1.6,3.8) circle (0.4);
\draw[circle, fill=yellow,color=yellow] (2.2,4) circle (0.4);
\draw[circle, fill=yellow,color=yellow] (0.45,7) circle (0.4);
\draw[circle, fill=yellow,color=yellow] (8.75,4.2) circle (0.4);
\draw[circle, fill=yellow,color=yellow] (8.75,6.5) circle (0.4);
\node at (-1.6,3.8) {$\theta_0$};
\node at (2.2,4) {$\theta_1$};
\node at (0.45,7) {$\theta_2$};
\node at (8.75,4.2) {$\theta_3$};
\node at (8.75,6.5) {$\theta_4$};
}
\end{tikzpicture}
\caption{Separation of the cardinals associated with $\Ewf$ where $\theta_0\leq\theta_1\leq \theta_2\leq \theta_3$ are uncountable regular cardinals, and $\theta_4$ is a cardinal such that $\theta_3\leq\theta_4=\theta_4^{{<}\theta_1}$, as
forced in~\cite[Thm.~5.6]{Car23}. The values of $\add^*_t(\Ewf)$, $\add_t(\Ewf)$, $\cov_t(\Ewf)$ and $\cof_t(\Ewf)$ are unclear.}
\label{SepofE}
\end{figure}
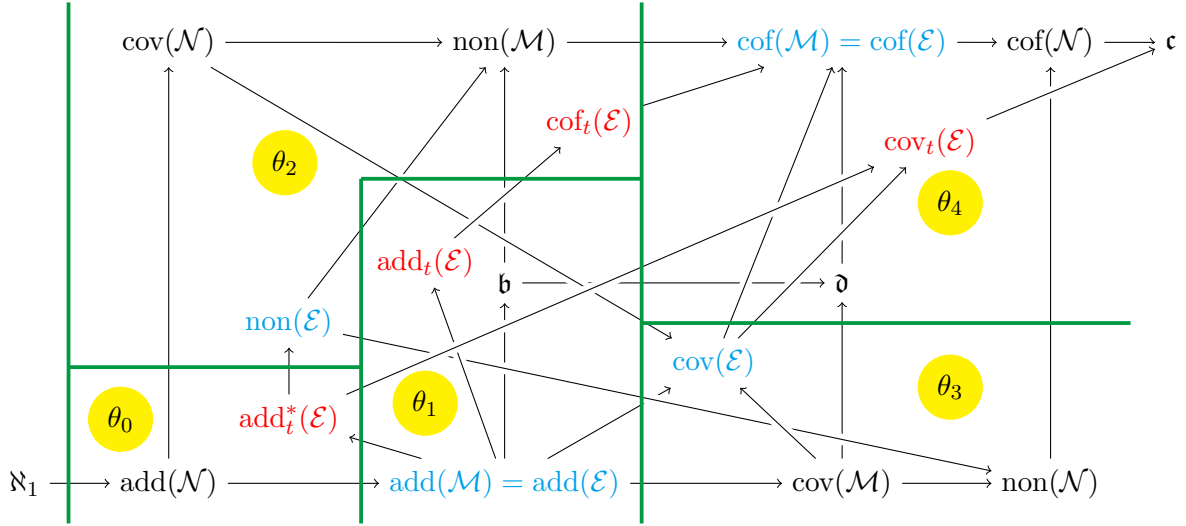



\newpage

{\small
\bibliography{bibli}
\bibliographystyle{alpha}
}


\end{document}